\documentclass{article}
\usepackage {amsmath, amsfonts, amssymb, mathrsfs, amsthm,stmaryrd, sectsty,hyphenat,enumitem,calc}
\usepackage{color}
\usepackage{palatino}
\usepackage[all]{xy}
\def\mf#1{\mathfrak{#1}}
\def\mc#1{\mathcal{#1}}
\def\mb#1{\mathbb{#1}}
\def\tx#1{\textrm{#1}}

\def\R{\mathbb{R}}

\def\C{\mathbb{C}}
\def\Q{\mathbb{Q}}
\def\A{\mathbb{A}}
\def\Z{\mathbb{Z}}
\def\N{\mathbb{N}}

\def\lmod{\setminus}

\def\ol#1{\overline{#1}}
\def\ul#1{\underline{#1}}

\def\hat{\widehat}

\def\rw{\rightarrow}
\def\lw{\leftarrow}

\def\lrw{\longrightarrow}

\def\hrw{\hookrightarrow}

\def\lw{\leftarrow}
\def\sm{\smallsetminus}

\def\<{\langle}
\def\>{\rangle}

\newenvironment{mytitle}
{\begin{center}\large\sc}
{\end{center}}

\newtheorem{thm}{Theorem}[subsection]
\newtheorem{lem}[thm]{Lemma}
\newtheorem{pro}[thm]{Proposition}
\newtheorem{cor}[thm]{Corollary}
\newtheorem{dfn}[thm]{Definition}
\newtheorem{fct}[thm]{Fact}

\newtheorem{cnds}[thm]{Conditions}

\sectionfont{\center\sc\normalsize}
\subsectionfont{\bf\normalsize}

\numberwithin{equation}{section}

\newlength{\sumcorr}

\def\ssum#1{\setlength{\sumcorr}{(\widthof{$\displaystyle\sum_{#1}$}-\widthof{$\displaystyle\sum$})/2} \hspace{-\sumcorr}\sum_{#1}\hspace{-\sumcorr} }

\begin{document}

\begin{mytitle} Global rigid inner forms and multiplicities of discrete automorphic representations \end{mytitle}
\begin{center} Tasho Kaletha \end{center}

\begin{abstract}
We study the cohomology of certain Galois gerbes over number fields. This cohomology provides a bridge between refined local endoscopy, as introduced in \cite{KalRI}, and classical global endoscopy. As particular applications, we express the canonical adelic transfer factor that governs the stabilization of the Arthur-Selberg trace formula as a product of normalized local transfer factors, we give an explicit constriction of the pairing between an adelic $L$-packet and the corresponding $S$-group (based on the conjectural pairings in the local setting) that is the essential ingredient in the description of the discrete automorphic spectrum of a reductive group, and we give a proof of some expectations of Arthur.
\end{abstract}
{\let\thefootnote\relax\footnotetext{This research is supported in part by NSF grant DMS-1161489.}}

\tableofcontents

\section{Introduction}

Let $F$ be a number field and let $G$ be a connected reductive group defined  over $F$. A central question in the theory of automorphic forms is the  decomposition of the right regular representation of $G(\A)$ on the Hilbert space $L^2_\tx{disc}(Z(\A)G(F)\lmod G(\A))$ or slight variants of it. Work of Labesse-Langlands \cite{LL79}, Langlands \cite{Lan83}, and Kottwitz \cite{Kot84}, on the Arthur-Selberg trace formula and its stabilization has provided a conjectural answer to this question. It is believed that there exists a group $L_F$ having the Weil group $W_F$ of $F$ as a quotient. Admissible tempered discrete homomorphisms $\phi$ from $L_F$ into the $L$-group of $G$ correspond to $L$-packets of tempered representations of $G(\A)$ and each tempered discrete automorphic representation belongs to one of these $L$-packets. However, not all representations in a given $L$-packet $\Pi_\phi$ are automorphic. It is expected that there exists a complex-valued pairing $\<-,-\>$ between a finite group $\mc{S}_\phi$ (closely related to the centralizer in $\hat G$ of the image of $\phi$) and the packet $\Pi_\phi$ corresponding to $\phi$, which realizes each $\pi \in \Pi_\phi$ as (the character of) some representation of the group $\mc{S}_\phi$. The number
\[ m(\phi,\pi) = |\mc{S}_\phi|^{-1}\sum_{x \in \mc{S}_\phi} \<x,\pi\>  \]
is then a non-negative integer and is expected to give the contribution of the pair $(\pi,\phi)$ to the $G(\A)$-representation $L^2_\tx{disc}(Z(\A)G(F)\lmod G(\A))$. In other words, the multiplicity of $\pi$ in the discrete spectrum is expected to be the sum of the numbers $m(\phi,\pi)$ as $\phi$ runs over all equivalence classes of parameters with $\pi \in \Pi_\phi$. Arthur \cite{ArtUARC,ArtUARGM} has extended this conjecture to encompass discrete automorphic representations that are non-tempered. We refer the reader to \cite{BR94} for more details.

One of the main goals of this paper is to give an explicit construction of the pairing $\<-,-\>$ based on the local conjectures formulated in \cite{KalRI}. Such an explicit construction was previously known only under the assumption that $G$ is quasi-split. In that case, the pairing  $\<-,-\>$ can be given as a product over all places of $F$ of (again conjectural) local pairings $\<-,-\>_v$ between the local analogs of $\mc{S}_\phi$ and $\Pi_\phi$. Our contribution here is to remove the assumption that $G$ is quasi-split. The main difficulty that arises once this assumption is dropped is that the local pairings $\<-,-\>_v$ stop being well-defined, even conjecturally. This is a reflection of the fact that there is no canonical normalization of the Langlands-Shelstad endoscopic transfer factor. Since the endoscopic character identities tie the values of the local pairing with the values of the transfer factor, we cannot expect to have a well-defined local pairing without having a normalization of the transfer factor. The problem runs deeper, however: Examples as basic as the case of inner forms of $\tx{SL}_2$ show that the local analog of $\mc{S}_\phi$ doesn't afford any pairing with the local analog of $\Pi_\phi$ that realizes elements of $\Pi_\phi$ as characters of $\mc{S}_\phi$.

One can attempt to deal with these  difficulties in an ad-hoc way. For example, one could fix an arbitrary normalization of the transfer factor. This is certainly sufficient for the purposes of stabilizing the geometric side of the trace formula.  However, it is insufficient for the purposes of interpreting the spectral side of the stabilized trace formula, and in particular for the construction of the global pairing. For example, it doesn't resolve the problem that the local $\mc{S}_\phi$-group doesn't afford a pairing with the local $L$-packet. One could also introduce an ad-hoc modification of the local $\mc{S}_\phi$-group, but then it is not clear what normalizations of the transfer factor correspond to pairings on such a modification.

In this paper, we overcome these difficulties by developing a bridge between the refined local endoscopy introduced in \cite[\S5]{KalRI} and classical global endoscopy. The results of \cite[\S5]{KalRI} resolve the local problems listed above. Namely, based on refinements of the notions of inner twist and endoscopic datum, they provide a natural normalization of the Langlands-Shelstad transfer factor and a modification of the classical local $\mc{S}_\phi$-group together with a precise conjectural description of the internal structure of local $L$-packets. The latter takes the form of a canonically normalized perfect pairing between the local $L$-packet and the modified $\mc{S}_\phi$-group. Using these objects, we obtain in the present paper the following results.

\begin{enumerate}
\item We prove that the product of the normalized transfer factors of \cite[\S5.3]{KalRI} is equal to the canonical adelic transfer factor involved in the stabilization of the trace formula (Proposition \ref{pro:tfprod}, resp. Equation \eqref{eq:tfprod1}).
\item We give an explicit formula for the global pairing $\<-,-\>$ in terms of the conjectural local pairings introduced in \cite[\S5]{KalRI}. Moreover, we prove that the global pairing $\<-,-\>$ is independent of the auxiliary cohomology classes involved in its construction and is thus canonically associated to the group $G$ (Proposition \ref{pro:pairprod}, resp. Equation \eqref{eq:pairprod1}).
\item We prove an expectation of Arthur \cite[Hypothesis 9.5.1]{Art13} about the existence of globally coherent collections of local mediating functions (Subsection \ref{sub:arthur}).
\end{enumerate}

Let us recall that the normalized local transfer factor and the conjectural local pairings depend on two refinements of the local endoscopic set-up: A rigidification of the local inner twist datum \cite[\S5.1]{KalRI} and a refinement of the local endoscopic datum \cite[\S5.3]{KalRI}. If we change either of these refinements, the normalized local transfer factors and local pairings also change. An obvious question is then, whether it is necessary to introduce the global analogs of these refinements into the global theory of endoscopy. The answer to this question is: no. The global theory of endoscopy needs no refinement. In particular, the results of the present paper fit seamlessly with the established stabilization of the Arthur-Selberg trace formula.

However, we do need a bridge between refined local endoscopy and usual global endoscopy. The basis of refined local endoscopy was the cohomology functor $H^1(u \rw W)$ introduced in \cite[\S3]{KalRI}. In order to relate this object to the global setting it is necessary to introduce certain Galois gerbes over number fields and study their cohomology. There are localization maps between these global cohomology groups and the local cohomology groups $H^1(u \rw W)$. These localization maps are the foundation of the bridge between refined local and global endoscopy.

The roles of the local and global cohomology groups are thus quite different. The local cohomology groups influence the normalizations in local endoscopy and must therefore enter into the statement of the local theorems and conjectures. On the other hand, the global groups serve to produce coherent collections of local objects, but the global objects obtained from these coherent collections are independent of the global cohomology classes used. This means that the global cohomology groups are much less present in the statements of global theorems and conjectures, but they are indispensable in their proofs.

Besides establishing a clear conjectural picture for general reductive groups, these results provide an important step towards proving these conjectures. This is due to the fact that the approach developed by Arthur \cite{ArtUARC,ArtUARGM} to prove both the local Langlands correspondence and the description of the discrete automorphic spectrum using the stabilization of the trace formula relies on the interplay between the local and the global conjectures. This approach was successfully carried out for quasi-split symplectic and orthogonal groups in \cite{Art13}, but its application to non-quasi-split groups was impeded by the lack of proper normalizations of the local objects, both transfer factors and spectral pairings. This paper, in conjunction with \cite{KalRI}, removes this obstruction.

Building on our results, Ta\"ibi has recently treated some non-quasi-split symplectic and orthogonal groups \cite{Taib}. Another example of Arthur's approach being carried out in the setting of non-quasi-split groups was given in \cite{KMSW}, where inner forms of unitary groups were treated. The results of \cite{KMSW} were obtained prior to the present work and use Kottwitz's theory of isocrystals with additional structure \cite{Kot85}, \cite{Kot97}, \cite{KotBG} to refine the local endoscopic objects instead. That theory works very well for reductive groups that have connected centers and satisfy the Hasse principle. When the center is not connected, however, not all inner forms can be obtained via isocrystals. Moreover, when the Hasse principle fails, we do not see a way to use the global theory of isocrystals to build a bridge between local and global endoscopy. The constructions of the present paper remove these conditions and work uniformly for all connected reductive groups. Moreover, our statement of the local conjecture is very closely related to the work of Adams-Barbasch-Vogan \cite{ABV92} on real groups as well as to Arthur's statement given in \cite{Art06}. The relationship with \cite{ABV92} was partially explored in \cite[\S5.2]{KalRI} and the relationship with \cite{Art06} is discussed at the end of the current paper.

We will now explain in some detail the construction of the global pairing. While in the body of the paper we treat general connected reductive groups, for the purposes of the introduction we assume that $G$ is semi-simple and simply connected. This case in easier to describe, because the notation is simpler and because we do not need to address the possible failure of the Hasse principle. At the same time, this case is in some sense the hardest from the point of view of local inner forms, so it will serve as a good illustration.

First let us assume in addition that $G$ is quasi-split and briefly recall the situation in this case. One fixes a Whittaker datum for $G$, which is a pair $(B,\psi)$ of a Borel subgroup of $G$ defined over the number field $F$ and a generic character $\psi : U(\A_F) \rw \C^\times$  of the unipotent radical $U$ of $B$ that is trivial on the subgroup $U(F)$ of $U(\A_F)$. For each place $v$ of $F$, restriction of $\psi$ to the subgroup $U(F_v)$ of $U(\A_F)$ provides a local Whittaker datum $(B,\psi_v)$ for $G$. Given a discrete global generic Arthur parameter $ \phi : L_F \rw {^LG}$, let $\phi_v : L_{F_v} \rw {^LG}$ be its localization at each place $v$ of $F$. We consider the global centralizer group $S_\phi=\tx{Cent}(\phi(L_F),\hat G)$ as well as, for each place $v$ of $F$, its local analog $S_{\phi_v}=\tx{Cent}(\phi(L_ {F_v}),\hat G)$. We also set $\mc{S}_\phi=\pi_0(S_\phi)$ and $\mc{S}_{\phi_v}=\pi_0(S_{\phi_v})$. In fact, since $\phi$ is discrete and $G$ is semi-simple we don't need to take $\pi_0$ in the global case, but we do need to take $\pi_0$  in the local case, as $\phi_v$ will usually not be discrete. There is a conjectural packet $\Pi_{\phi_v}$ of irreducible admissible tempered representations of $G(F_v)$ and a conjectural pairing $\<-,-\>_v : \mc {S}_{\phi_v} \times \Pi_{\phi_v} \rw \C$ such that $\<-,\pi_v\>_v$ is an irreducible character of the finite group $\mc{S}_{\phi_v}$ for each $\pi_v \in \Pi_{\phi_v}$. Once the $L$-packets are given, this pairing is uniquely determined by the endoscopic character identities taken with respect to the transfer factor that is normalized according to the local Whittaker datum $(B,\psi_v)$. According to a conjecture of Shahidi \cite[\S9]{Sha90}, the local pairing is expected to satisfy $\<-,\pi_v\>_v=1$ for the unique $(B,\psi_v)$-generic member $\pi_v \in \Pi_{\phi_v}$. The global packet is defined as $\Pi_\phi = \{ \pi = \otimes  \pi_v| \pi_v \in \Pi_{\phi_v}, \<-,\pi_v\>_v = 1 \textrm{ for a.a. } v  \}$. Via the natural map $\mc{S}_{\phi} \rw \mc{S}_{\phi_v}$ each local pairing  $\<-,\pi_v\>_v$ can be restricted to $\mc{S}_\phi$ and the global  pairing
\[ \<-,-\> : \mc{S}_\phi \times \Pi_\phi \rw \C \]
is defined as the product $\<x,\pi\>=\prod_v \<x,\pi_v\>_v$.

We now drop the assumption that $G$ is quasi-split, while still  maintaining the assumption that it is semi-simple and simply connected.  There exists a quasi-split semi-simple and simply connected group $G^*$  and an isomorphism $\xi : G^* \rw G$ defined over $\ol{F}$ such that  for each $\sigma \in \Gamma=\tx{Gal}(\ol{F}/F)$ the automorphism $\xi^{-1}\sigma(\xi)$ of  $G^*$ is inner. We fix this data, thereby realizing $G$ as an inner twist of $G^*$. For a moment we turn to the quasi-split group $G^*$ and fix objects as above: a global Whittaker datum $(B,\psi)$ for $G^*$ with localizations $(B,\psi_v)$, and a discrete  generic Arthur parameter $\phi : L_F \rw {^LG^*}$ with localizations $\phi_v$. The groups $S_\phi$ and $S_{\phi_v}$ are also defined as above and are subgroups of $\hat G^*$. We also keep the group $\mc{S}_\phi=\pi_0(S_\phi)$, but the conjecture of \cite[\S5.4]{KalRI} tells us that we need to replace the group $\mc{S}_{\phi_v}$ by $\pi_0(S_{\phi_v}^+)$, where $S_{\phi_v}^+$ is the preimage of $S_{\phi_v}$ in $\hat G^*_\tx{sc}$, the simply connected cover of the (in this case adjoint) group $\hat G^*$. We will now obtain for each place $v$ of $F$ a conjectural local pairing $\<-,-\>_{z_v} : \pi_0(S_{\phi_v}^+) \times \Pi_{\phi_v}(G) \rw \C$, but for this we need to endow $G$ with the structure of a rigid inner twist of $G^*$ at the place $v$, which is a lift $z_v \in Z^1(u_v \rw  W_v,Z(G^*_\tx{sc}) \rw G^*_\tx{sc})$ of the element of $Z^1(\Gamma_v,G^*_\tx{ad})$ given by $\sigma \mapsto \xi^{-1}\sigma(\xi)$. Here we are using the cohomology functor $H^1(u_v \rw W_v,-)$ defined in \cite{KalRI} as  well as the corresponding cocycles. Once $z_v$ is fixed, according to \cite[\S5.4]{KalRI} there exists a pairing $\<-,-\>_{z_v} : \pi_0(S_{\phi_v}^+) \times \Pi_{\phi_v}(G) \rw \C$ such that for each $\pi_v \in \Pi_{\phi_v}(G)$ the function $\<-,\pi_v\>_{z_v}$ is an irreducible character of $\pi_0(S_{\phi_v}^+)$ and such that the endoscopic character identities are satisfied with respect to the normalization of the transfer factor given in \cite[\S5.3]{KalRI}, which involves both the local Whittaker datum $(B,\psi_v)$ and the datum $z_v$. As we noted above, the normalization of the transfer factor is established in \cite[\S5.3]{KalRI} unconditionally and the endoscopic character identities specify the pairing $\<-,\pi_v\>_{z_v}$ uniquely, provided it can be shown to exist. We now define the global packet as $\Pi_\phi(G)=\{\pi=\otimes\pi_v| \pi_v \in \Pi_{\phi_v}(G), \<-,\pi_v\>_{z_v}=1 \textrm{ for a.a. } v \}$. Letting $S_\phi^+$ be the preimage in $\hat G^*_\tx{sc}$ of the global centralizer group $S_\phi$ and taking the product of the local pairings over all places $v$ we obtain a global pairing
\[ \<-,-\> : S_\phi^+ \times \Pi_\phi(G) \rw \C. \]
We must now discuss the dependence of this global pairing on the choices of $z_v$. At each place $v$ there will usually be several choices for $z_v$. Changing from one choice to another affects both the transfer factor at that place as well as the local pairing $\<-,-\>_{z_v}$. It is then clear that the global pairing that we have defined depends on the choices of $z_v$ at each place $v$. What we will show however is that if we impose a certain coherence condition on the collection $(z_v)_v$ of local cocycles (here $v$ runs over the set of places of $F$) then the pairing becomes independent of this choice. One instance of this coherence should be that the product over all places of the normalized transfer factors is equal to the canonical adelic transfer factor involved in the stabilization of the Arthur-Selberg trace formula. Another instance of this coherence should be that the global pairing is independent of the collection $(z_v)_v$ used in its construction, and furthermore descends from the group $S_\phi^+$ to its quotient $\mc{S}_\phi$.

It is this coherence that necessitates the development of the global Galois gerbes and their cohomology, and this work takes up a large part of this paper. We construct a functor $H^1(P_{\dot V} \rw \mc{E}_{\dot V})$ that assigns to each affine algebraic group $G$ defined over $F$ and each finite central subgroup $Z \subset G$ defined over $F$ a set $H^1(P_{\dot V} \rw \mc{E}_{\dot V}, Z \rw G)$. When $G$ is abelian this set is an abelian group. At each place $v$ of $F$ there is a functorial localization map $H^1(P_{\dot V} \rw \mc{E}_{\dot V}, Z \rw G) \rw H^1(u_v \rw W_v,Z \rw G)$. A collection $(z_v)_v$ of local cohomology classes is coherent if it is in the image of the total localization map.

The construction of the global Galois gerbes and the study of their cohomology is considerably more involved than that of their local counterpart. Let us briefly recall that to construct the local gerbe we consider the pro-finite algebraic group
\[ u_v = \varprojlim_{E_v/F_v,n} \tx{Res}_{E_v/F_v}\mu_n/\mu_n \] over the given local field $F_v$. Here $E_v$ runs over the finite Galois extensions of $F_v$ contained in $\ol{F_v}$ and $n$ runs over all natural numbers. We show that $H^1(\Gamma_v,u_v)$ vanishes and $H^2(\Gamma_v,u_v)$ is pro-cyclic with a distinguished generator. The isomorphism class of the gerbe is determined by this generator and the gerbe itself is unique up to (essentially) unique isomorphism due to the vanishing of $H^1(\Gamma_v,u_v)$. Using this one introduces for each affine algebraic group $G$ and each finite central subgroup $Z \subset G$, both defined over $F_v$, a cohomology set denoted by $H^1(u_v \rw W_v,Z \rw G)$. Two central features are the surjectivity of the map $\tx{Hom}(u_v,Z)^{\Gamma_v} \rw H^2(\Gamma_v,Z)$ evaluating a homomorphism at the distinguished generator of $H^2(\Gamma,u_v)$ for all finite multiplicative groups $Z$, and the existence of a Tate-Nakayama-type isomorphism for the cohomology set $H^1(u_v \rw W_v,Z \rw T)$ whenever $T$ is a torus (and a generalization thereof when $G$ is connected and reductive).

Now consider a number field $F$. The analog of the groups $\tx{Res}_{E_v/F_v}\mu_n/\mu_n$ is more technical. To describe it, let $S$ be a finite set of places of $F$ containing the archimedean places. Fix a finite Galois extension $E/F$ unramified outside of $S$ and a natural number $N$ that is a unit away from $S$. We then define the finite $\tx{Gal}(E/F)$-module
\[ M_{E,S,N} = \tx{Maps}(\tx{Gal}(E/F) \times S_E,\frac{1}{N}\Z/\Z)_{0,0}. \]
Here $S_E$ is the set of all places of $E$ lying above $S$ and we are considering all maps $f$ from the finite set $\tx{Gal}(E/F) \times S_E$ to the finite abelian group $\frac{1}{N}\Z/\Z$ which satisfy the two conditions
\[ \forall w \in S_E: \sum_{\sigma \in \tx{Gal}(E/F)} f(\sigma,w) = 0 \quad\tx{and}\quad \forall \sigma \in \tx{Gal}(E/F): \sum_{w \in S_E} f(\sigma,w)=0. \]
Let $\dot S_E \subset S_E$ be a set of lifts for the places of $S$. We consider the submodule $M_{E,\dot S_E,N} \subset M_{E,S,N}$ consisting of all $f$ that satisfy the condition
\[ f(\sigma,w) \neq 0 \Rightarrow \sigma^{-1}w \in \dot S_E. \]
We take this submodule as the character module of a finite multiplicative group defined over $O_{F,S}$ and consider its $O_S$-points
\[ P_{E,\dot S_E,N}(O_S) = \tx{Hom}(M_{E,\dot S_E,N},O_S^\times), \]
where $O_S^\times$ is the group of units in the ring of integers of $F_S$, the maximal extension of $F$ unramified outside of $S$. After placing multiple technical restrictions on the set $S$ and its lift $\dot S_E$ we are able to show that $H^2(\Gamma_S,P_{E,\dot S_E,N}(O_S))$ has a distinguished element (even though it is no longer pro-cyclic). As we vary $E$, $S$, and $N$, the groups $P_{E,\dot S_E,N}$ form an inverse system and the transition maps identify the distinguished cohomology classes. This gives a distinguished element in $\varprojlim H^2(\Gamma,P_{E,\dot S_E,N})$. Let $P_{\dot V}$ be the inverse limit of the $P_{E,\dot S_E,N}$. The natural map $H^2(\Gamma,P_{\dot V}) \to \varprojlim H^2(\Gamma,P_{E,\dot S_E,N})$ is surjective, but unlike in the local case, it is not injective. Nonetheless, with some additional work we are able to obtain a canonical element of $H^2(\Gamma,P_{\dot V})$ that lifts the distinguished element of $\varprojlim H^2(\Gamma,P_{E,\dot S_E,N})$. As in the local case, the map $\tx{Hom}(P_{\dot V},Z) \rw H^2(\Gamma,Z)$ given by the canonical class is surjective. The cohomology group $H^1(\Gamma,P_{\dot V})$ vanishes, and so do its local analogs $H^1(\Gamma_v,P_{\dot V})$. The canonical class specifies an isomorphism class of gerbes $\mc{E}_{\dot V}$ for the absolute Galois group $\Gamma$ of the number field $F$ bound by $P_{\dot V}$. This isomorphism class leads to a cohomology functor $H^1(P_{\dot V} \rw \mc{E}_{\dot V})$ that comes equipped with functorial localization maps $H^1(P_{\dot V} \rw \mc{E}_{\dot V}) \rw H^1(u_v \rw W_v)$.

For an algebraic torus $T$ (and more generally for a connected reductive group $G$), we prove a Tate-Nakayama-type duality theorem that identifies the group $H^1(P_{\dot V} \rw \mc{E}_{\dot V},Z \rw T)$ with an abelian group constructed from the character module of $T$. We furthermore give an explicit formula for the localization maps in terms of this duality. The proof of the duality theorem proceeds through the study of cohomology groups $H^1(P_{E,\dot S_E,N} \rw \mc{E}_{E,\dot S_E,N},Z \rw T)$ that are related to $H^1(P_{\dot V} \rw \mc{E}_{\dot V},Z \rw T)$. They are defined relative to the isomorphism class of gerbes $\mc{E}_{E,\dot S_E,N}$ for the Galois group $\Gamma_S$ of the extension $F_S/F$ bound by $P_{E,\dot S_E,N}(O_S)$ that is given by the distinguished element of $H^2(\Gamma_S,P_{E,\dot S_E,N}(O_S))$. We formulate and prove a duality theorem for these cohomology groups first. We then construct inflation maps that relate these cohomology groups for different $E,S,N$ to each other and to $H^1(P_{\dot V} \rw \mc{E}_{\dot V},Z \rw T)$ and obtain the duality theorem for the latter by taking the colimit over $E,S,N$. A technical difficulty that is presented by the gerbe $\mc{E}_{E,\dot S_E,N}$ is that it has non-trivial automorphisms, because $H^1(\Gamma_S,P_{E,\dot S_E,N}(O_S))$ doesn't vanish. This implies that for an arbitrary affine algebraic group $G$ defined over $O_{F,S}$ the cohomology set $H^1(P_{E,\dot S_E,N} \rw \mc{E}_{E,\dot S_E,N},Z \rw G)$ depends on the particular realization of $\mc{E}_{E,\dot S_E,N}$, and not just on its isomorphism class. In the case of tori we are able to circumvent this difficulty by first using a specific realization $\mc{\dot E}_{E,\dot S_E,N}$, depending on various choices, and proving the duality theorem for this realization. The duality theorem implies certain group-theoretic properties of $H^1(P_{E,\dot S_E} \rw \mc{\dot E}_{E,\dot S_E},Z \rw T)$, where we have now taken the colimit over all $N$, that allow us to conclude that the automorphisms of $\mc{\dot E}_{E,\dot S_E}$ act trivially on its cohomology. As a consequence we see that
the cohomology group $H^1(P_{E,\dot S_E} \rw \mc{E}_{E,\dot S_E},Z \rw T)$ depends only on the isomorphism class of $\mc{E}_{E,\dot S_E}$. Once the case of tori is handled, its generalization to connected reductive groups follows without much additional effort.

The reader familiar with \cite{KalRI} will notice that our treatment of the Tate-Nakayama isomorphism in the global case is less direct and less explicit than in the local case. While it leads to the same result, it would be difficult to use it for explicit computations. Recently Ta\"ibi has been able to find a more direct and explicit construction of the canonical class and of the Tate-Nakayama isomorphism in the global case that is much closer to our treatment of the local case \cite{TaibGRI}, and will likely prove useful for explicit computations, for example of automorphic multiplicities.

Let us now comment on the role of the canonical class. The results discussed so far can all be obtained using a gerbe whose isomorphism class is given by an arbitrary element of $H^2(\Gamma,P_{\dot V})$ that lifts the distinguished element of $\varprojlim H^2(\Gamma,P_{E,\dot S_E,N})$, not necessarily the canonical class. Consider again the localization maps $H^1(P_{\dot V} \rw \mc{E}_{\dot V}) \rw H^1(u_v \rw W_v)$. Using the Tate-Nakayama isomorphism it can be shown that when $G$ is connected and reductive, the localization of a given element of $H^1(P_{\dot V} \rw \mc{E}_{\dot V},Z \to G)$ is trivial at almost all places. It turns out that for applications to automorphic forms this, while necessary, is not quite sufficient. One needs a strengthening of this result that begins with the observation that the localization maps are already well-defined on the level of cocycles, namely as maps
\[ Z^1(P_{\dot V} \rw \mc{E}_{\dot V},Z \to G) \rw Z^1(u_v \rw W_v,Z \to G)/B^1(F_v,Z). \]
It then makes sense to ask if, given an element of $Z^1(P_{\dot V} \rw \mc{E}_{\dot V},Z \to G)$, its localization at alomst all places is \emph{unramified}, i.e. an element of $Z^1(u_v \rw W_v,Z \to G)/B^1(F_v,Z)$ that is inflated from $Z^1(\tx{Gal}(F_v^\tx{ur}/F_v),G(O_{F_v^{\tx{ur}}}))$. It is a result of Ta\"ibi that this is true \cite[Proposition 6.1.1]{TaibGRI} if the global gerbe $\mc{E}_{\dot V}$ is the one given by the canonical class, but can fail \cite[Proposition 6.3.1]{TaibGRI} if one uses a global gerbe corresponding to another lift to $H^2(\Gamma,P_{\dot V})$ of the distinguished element of $\varprojlim H^2(\Gamma,P_{E,\dot S_E,N})$.

We will now give a brief overview of the contents of this paper. Section \ref{sec:cohf} is devoted to the construction of the cohomology set $H^1(P_{\dot V} \rw \mc{E}_{\dot V},Z \rw G)$. We begin in Subsection \ref{sub:tatetori} with a review of Tate's description of the cohomology groups $\hat H^i(\tx{Gal}(E/F),T(O_{E,S}))$ of algebraic tori and prove some results about the compatibility of this description with variations of $E$ and $S$. In Subsection \ref{sub:h2z} we derive a description of the cohomology group $H^2(\Gamma,Z)$ for a finite multiplicative group $Z$ defined over $F$. This description is the basis of the study of the finite multiplicative groups $P_{E,\dot S_E,N}$ in Subsection \ref{sub:pes}. After defining these groups we show that $H^2(\Gamma_S,P_{E,\dot S_E,N}(O_S))$ has a distinguished element and study how this element varies with $E$, $S$, and $N$. We also define the pro-finite multiplicative group $P_{\dot V}$ and discuss its relationship with the local groups $u_v$. In Subsection \ref{sub:h1van} we prove the vanishing of $H^1(\Gamma,P_{\dot V})$ and $H^1(\Gamma_v,P_{\dot V})$. These technical results allow us to specify, in Subsection \ref{sub:canclass}, a canonical lift $\xi \in H^2(\Gamma,P_{\dot V})$ of the distinguished element of $\varprojlim H^2(\Gamma,P_{E,\dot S_E,N})$. In Subsection \ref{sub:coh} we define the cohomology set $H^1(P_{\dot V} \rw \mc{E}_{\dot V},Z \rw G)$ for an arbitrary affine algebraic group $G$ and a finite central subgroup $Z$, both defined over $F$, as well as the localization maps. These constructions also rely crucially on the vanishing results of Subsection \ref{sub:h1van}. In Subsection \ref{sub:tn+} we formulate and prove the Tate-Nakayama-type duality theorem for $H^1(P_{\dot V} \rw \mc{E}_{\dot V},Z \rw T)$ when $T$ is a torus and give a formula for the localization map on the dual side. The cohomology groups $H^1(P_{E,\dot S_E} \rw \mc{\dot E}_{E,\dot S_E},Z \rw T)$ are defined and studied in this Subsection. In Subsection \ref{sub:tn+g} we extend the duality theorem to the case of connected reductive groups. One implication of the results of Subsections \ref{sub:tn+} and \ref{sub:tn+g} is the description in terms of linear algebra of the coherent collections of local cohomology classes, i.e. those which arise as the localization of global cohomology classes. This linear algebraic description has an immediate interpretation in terms of the Langlands dual group.

Section \ref{sec:mult} contains the applications to automorphic forms. Subsections \ref{sub:endolg} and \ref{sub:cohfam} constitute the bridge between refined local and global endoscopy. Given a global inner twist we construct a coherent collection of rigid local inner twists, and given a global endoscopic datum we construct a coherent collection of refined local endoscopic data. The local collections depend on some choices, but these choices are ultimately seen to have no influence on the global objects obtained from the local collections. In Subsection \ref{sub:tfprod} we show that the product of normalized local transfer factors is equal to the canonical adelic transfer factor, and in Subsection \ref{sub:mult} we show that the product of normalized local pairings descends to a canonical global pairing. In both cases we use the coherent local collections obtained from the global inner twist and the global endoscopic datum, but the product of normalized local transfer factors and the product of normalized local pairings are independent of these collections. Both product formulas contain some extra factors that one might not initially expect. These factors are explicit and non-conjectural and their appearance is just a matter of formulation, as we discuss in the next paragraph.

In the final Subsection \ref{sub:arthur} we summarize the results of Section \ref{sec:mult} using the language of \cite{Art06}. This subsection carries little mathematical content, but provides a dictionary between the language of \cite{KalRI} and that of \cite{Art06}. More precisely, it shows that the local conjecture of \cite{KalRI} implies (and is in fact equivalent to) a stronger version of the local conjecture of \cite{Art06}. We believe that this dictionary can be useful, because both languages have their advantages. The formulation of \cite{Art06} is from the point of view of the simply connected cover of the given reductive group and has the advantage that the passage between local and global endoscopy and the global product formulas take a slightly simpler form, in which in particular the extra factors referred to above are not visible, because they become part of the local normalizations. The formulation of \cite{KalRI} is from the point of view of the reductive group itself and has the advantage of making the local statements more flexible and transparent. In particular, certain basic local manipulations, like Levi descent, are easier to perform in this setting. Moreover, this language accommodates the notions of pure and rigid inner twists, which have been shown in the work of Adams-Barbasch-Vogan \cite{ABV92} on real groups to play an important role in the study of the local conjecture and in particular in its relationship to the geometry of the dual group.

In this paper we have restricted ourselves to discussing only tempered local and global parameters. This restriction is made just to simplify notation and because the problems we are dealing with here are independent of the question of temperedness. The extension to general Arthur parameters is straightforward and is done in the same way as described in \cite{ArtUARC}. We refer the reader to \cite{Taib} for an example of our results being used in the non-tempered setting.

A remark may be in order about the relationship between the cohomological constructions of this paper and those of \cite{KotBG}. While both papers study the cohomology of certain Galois gerbes, the gerbes used in the two papers are very different. The gerbes of \cite{KotBG} are bound by multiplicative groups whose character modules are torsion-free. In the local case they are the so called Dieudonne- and weight-gerbes of \cite{LR87} and are bound by split multiplicative groups. On the other hand, the gerbes of \cite{KalRI} and the present paper are bound by non-split multiplicative groups whose character modules are both torsion and divisible. This makes the difficulties involved in studying them and the necessary techniques quite different. It is at the moment not clear to us what the relationship between the two global constructions is. The relationship between the two local constructions is discussed in \cite{KalRIBG}.

It should be clear to the reader that this paper owes a great debt to the ideas of Robert Kottwitz. The author wishes to express his gratitude to Kottwitz for introducing him to this problem and generously sharing his ideas. The author thanks Olivier Ta\"ibi for his careful reading of the manuscript and for his many comments and suggestions that lead to considerable improvements of both the results and their exposition. For example, Ta\"ibi's comments led the author to develop the material in \S\ref{sub:h1van},\S\ref{sub:canclass}, and the statement and proof of Lemma \ref{lem:ramification} and its necessity were communicated by him. The author also thanks James Arthur for his interest and support and Alexander Schmidt for bringing to the author's attention the results in \cite{NSW08} on the cohomology of Galois groups of number fields with ramification conditions.

\section{Notation}

Throughout the paper, $F$ will denote a fixed number field, i.e. a finite extension of the field $\Q$ of rational numbers. We will often use finite Galois extensions $E$ of $F$, which are all assumed to lie within a fixed algebraic closure $\ol{F}$ of $F$. For any such $E/F$, we denote by $\Gamma_{E/F}$ the finite Galois group $\tx{Gal}(E/F)$, which is the quotient of the absolute Galois group $\Gamma_F=\tx{Gal}(\ol{F}/F)$ by its subgroup $\Gamma_E$. We will write $I_E=\A_E^\times$ for the idele group of $E$ and $C(E)=I_E/E^\times$ for the idele class group of $E$.

We write $V_E$ for the set of all places (i.e. valuations up to equivalence) of $E$ and $V_{E,\infty}$ for the subset of archimedean valuations. A subset $S \subset V_E$ will be called \emph{full}, if it is the preimage of a subset of the set of places $V_\Q$ of $\Q$. If $K/E/F$ is a tower of extensions and $S \subset V_E$, we write $S_K$ for the set of all places of $K$ lying above $S$, and we write $S_F$ for the set of all places of $F$ lying below $S$. We write $p: S_K \rw S_E$ for the natural projection map. Given a subset $S \subset V_\Q$, we write $\N_S$ for the set of those natural numbers whose prime decomposition involves only primes contained in $S$. We equip $\N_S$ with the partial order given by divisibility. Given a subset $S \subset V_F$, we write $\N_S$ for $\N_{S_\Q}$ by abuse of notation.

If $V_{F,\infty} \subset S \subset V_F$ is a set of places, and $E/F$ is a finite Galois extension unramified outside of $S$ we will borrow the following notation from \cite[\S VIII.3]{NSW08}. Let $O_{E,S}$ be the ring of $S$-integers of $E$, i.e. the subring of $E$ consisting of elements whose valuation at all places away from $S_E$ is non-negative. Let $I_{E,S} = \prod_{w \in S_E} E_w^\times$, $U_{E,S}=\prod_{w \notin S_E} O_{E_w}^\times$. Then $J_{E,S}=I_{E,S} \times U_{E,S}$ is the subgroup of $E$-ideles which are units away from $S$. We set $C_{E,S}=I_{E,S}/O_{E,S}^\times$ and $C_S(E)=I_E/E^\times U_{E/S}$.

Taking the limit of these objects over all finite Galois extensions $E/F$ that are unramified away from $S$ we obtain the following. We denote by $F_S$ the maximal extension of $F$ that is unramified away from $S$. We write $\Gamma_S$ for the Galois group of $F_S/F$. We denote by $O_S$ the direct limit of $O_{E,S}$ and by $I_S$ the direct limit of $I_{E,S}$, the limits being taken over all finite extensions $F_S/E/F$. We denote by $C_S$ the direct limit of either $C_{E,S}$ or $C_S(E)$. These two limits are the same \cite[Prop 8.3.6]{NSW08} and are in turn equal to the
quotient $C(F_S)/U_S$, where $C(F_S) = \varinjlim_{F_S/E/F} C(E)$ is the idele class group of $F_S$, and $U_S=\varinjlim_{F_S/E/F} U_{E,S}$. In the special case $S=V_F$ we will drop the subscript $S$ from the notation, thereby obtaining for example $I=\A_{\ol{F}}^\times$ and $C=\A_{\ol{F}}^\times/\ol{F}^\times$.

Given two finite sets $X,Y$, we write $\tx{Maps}(X,Y)$ for the set of all maps from $X$ to $Y$. If $X,Y$ are endowed with an action of $\Gamma_F$, or a related group, we endow $\tx{Maps}(X,Y)$ with the action given by $\sigma(f)(x)=\sigma(f(\sigma^{-1}x))$.

Given an abelian group $A$, we write $A[n]$ for the subgroup of $n$-torsion elements for any $n \in \N$, and we write $\exp(A)$ for the exponent of $A$, allowing $\tx{exp}(A)=\infty$. We follow the standard conventions of orders of pro-finite groups.

If $B$ is a profinite group acting on a (sometimes abelian) group $X$, we write $H^i(B,X)$ for the set of continuous cohomology classes, where $i \geq 0$ if $X$ is abelian and $i=0,1$ if $X$ is non-abelian. The group $X$ is understood to have the discrete topology \emph{unless} it is explicitly defined as an inverse limit, in which case we endow it with the inverse limit topology. When $B$ is finite and $X$ is abelian we also have the modified cohomology groups $\hat H^i(B,X)$ defined for all $i \in \Z$ by Tate. For us usually $B=\Gamma_{E/F}$ for some finite Galois extension $E/F$. In this case we will write $N_{E/F} : X \rw X$ for the norm map of the action of $\Gamma_{E/F}$ on $X$, and we will denote by $X^{N_{E/F}}$ its kernel. We will also write $\Z[\Gamma_{E/F}]$ for the group ring of $\Gamma_{E/F}$ and $I_{E/F}$ for its augmentation ideal. Then $I_{E/F}X$ is a subgroup of $X$.

\section{Cohomology with $\ol{F}$-coefficients} \label{sec:cohf}

\subsection{Some properties of Tate duality over $F$} \label{sub:tatetori}
The purpose of this subsection is to review Tate's description \cite{Tate66} of the cohomology group $H^1(\Gamma_{E/F},T(O_{E,S}))$, where $T$ is an algebraic torus defined over $F$ and split over a finite Galois extension $E/F$ and $S$ is a (finite or infinite) set of places of $F$ subject to some conditions. We will also establish some facts about this description that we have been unable to find in the literature.

We assume that $S$ satisfies the following.
\begin{cnds} \label{cnds:tate}
\begin{enumerate}
	\item $S$ contains all archimedean places and all places that ramify in $E$.
	\item Every ideal class of $E$ has an ideal with support in $S_E$.
\end{enumerate}
\end{cnds}
Let $T$ be an algebraic torus defined over $F$ and split over $E$. We write $X=X^*(T)$ and $Y=X_*(T)$. The group $T(O_{E,S}) = \tx{Hom}(X,O_{E,S}^\times) = Y \otimes O_{E,S}^\times$ is a $\Gamma_{E/F}$-module. In \cite{Tate66} Tate provides a description of the cohomology groups $\hat H^i(\Gamma_{E/F},T(O_{E,S}))$. Consider the abelian group $\Z[S_E]$ consisting of formal finite integral linear combinations $\sum_{w \in S_E} n_w [w]$ of elements of $S_E$, and its subgroup $\Z[S_E]_0$ defined by the property $\sum_{w \in S_E} n_w=0$. Tate constructs a canonical cohomology class $\alpha_3(E,S) \in H^2(\Gamma_{E/F},\tx{Hom}(\Z[S_E]_0,O_{E,S}^\times))$ and shows that cup product with this class induces for all $i \in \Z$ an isomorphism
\begin{equation} \label{eq:tnisotorus} \hat H^{i-2}(\Gamma_{E/F},Y[S_E]_0) \rw \hat H^i(\Gamma_{E/F},T(O_{E,S})). \end{equation}
We have written here $Y[S_E]_0$ as an abbreviation for $\Z[S_E]_0 \otimes Y$. For our applications, it will be important to know how this isomorphism behaves when we change $S$ and $E$. Before we can discuss this, we need to review the construction of the class $\alpha_3(E,S)$.

Tate considers the group $\tx{Hom}((B_{E,S}),(A_{E,S}))$ consisting of triples of homomorphisms $(f_3,f_2,f_1)$ making the following diagram commute
\[ \xymatrix{
(A_{E,S}):\ 1\ar[r]&O_{E,S}^\times\ar[r]^{a'}&J_{E,S}\ar[r]^a&C(E)\ar[r]&1\\
(B_{E,S}):\ 0\ar[r]&\Z[S_E]_0\ar[r]^{b'}\ar[u]^{f_3}&\Z[S_E]\ar[r]^b\ar[u]^{f_2}&\Z\ar[r]\ar[u]^{f_1}&0\\
}\]
The two sequences $(A_{E,S})$ and $(B_{E,S})$ are exact, with $a$ being the natural projection, which is surjective due to Conditions \ref{cnds:tate}, and $b$ being the augmentation map $b(\sum_w n_w[w])=\sum_w n_w$. It is clear that extracting the individual entries of a triple $(f_3,f_2,f_1)$ provides maps
\[ \xymatrix{
&\tx{Hom}((B_{E,S}),(A_{E,S}))\ar[dl]\ar[d]\ar[dr]&\\
\tx{Hom}(\Z[S_E]_0,O_{E,S}^\times)&\tx{Hom}(\Z[S_E],J_{E,S})&\tx{Hom}(\Z,C(E))
}\]
The product of the vertical and right-diagonal maps is an injection and induces an injection on the level of $H^2(\Gamma_{E/F},-)$. The class $\alpha_3(E,S)$ is the image of a class $\alpha(E,S) \in H^2(\Gamma_{E/F},\tx{Hom}((B_{E,S}),(A_{E,S})))$ under the left diagonal map. The class $\alpha(E,S)$ itself is the unique class mapping to the pair of classes
\[ (\alpha_2(E,S),\alpha_1(E)) \in H^2(\Gamma_{E/F},\tx{Hom}(\Z[S_E],J_{E,S})) \times H^2(\Gamma_{E/F},\tx{Hom}(\Z,C(E))), \]
which we will now describe.

The class $\alpha_1(E) \in H^2(\Gamma_{E/F},\tx{Hom}(\Z,C(E)))$ is the fundamental class associated by global class field theory to the extension $E/F$. Using the Shapiro\-isomorphism
\[ H^2(\Gamma_{E/F},\tx{Hom}(\Z[S_E],J_{E,S}) \cong \prod_{v \in S} H^2(\Gamma_{E_{\dot v}/F_v},J_{E,S}),  \]
where for each $v \in S$ we have chosen an arbitrary lift $\dot v \in S_E$, Tate defines the class $\alpha_2(E,S) \in H^2(\Gamma_{E/F},\tx{Hom}(\Z[S_E],J_{E,S}))$ to correspond to the element $(\alpha_{E_{\dot v}/F_v})_{v \in S}$ of the right hand side of the above isomorphism, where each $\alpha_{E_{\dot v}/F_v}$ is the image of the fundamental class in $H^2(\Gamma_{E_{\dot v}/F_v},E_{\dot v}^\times)$ under the natural inclusion $E_{\dot v}^\times \rw J_{E,S}$.

Having recalled the construction of \eqref{eq:tnisotorus} we are now ready to study how it varies with respect to $S$ and $E$.

\begin{lem} \label{lem:tnstori}
Let $S \subset S' \subset V_F$. The inclusion $S_E \rw S'_E$ provides maps $\Z[S_E]_0 \rw \Z[S'_E]_0$ and $O_{E,S} \rw O_{E,S'}$ which fit in the commutative diagram
\[ \xymatrix{
\hat H^{i-2}(\Gamma_{E/F},Y[S_E]_0)\ar[d]\ar[r]&\hat H^{i}(\Gamma_{E/F},T(O_{E,S}))\ar[d]\\
\hat H^{i-2}(\Gamma_{E/F},Y[S'_E]_0)\ar[r]&\hat H^{i}(\Gamma_{E/F},T(O_{E,S'}))
} \]
\end{lem}
\begin{proof}
The top horizontal map is given by cup product with the class $\alpha_3(E,S) \in H^2(\Gamma_{E/F},\tx{Hom}(\Z[S_E]_0,O_{E,S}))$, while the bottom horizontal map is given by cup product with the analogous class $\alpha_3(E,S')$. We must relate the two classes. For this, let $\tx{Hom}_S(\Z[S'_E],J_{E,S'})$ be the subgroup of $\tx{Hom}(\Z[S'_E],J_{E,S'})$ consisting of those homomorphisms which map the subgroup $\Z[S_E]$ of $\Z[S'_E]$ into the subgroup $J_{S,E}$ of $J_{S',E}$. We define analogously $\tx{Hom}_S(\Z[S'_E]_0,O_{E,S'}^\times)$ and $\tx{Hom}_S((B_{E,S'}),(A_{E,S'}))$. We have the maps
\[ \tx{Hom}(\Z[S_E]_0,O_{E,S}^\times) \lw \tx{Hom}_S(\Z[S'_E]_0,O_{E,S'}^\times) \rw \tx{Hom}(\Z[S'_E]_0,O_{E,S'}^\times), \]
the left one given by restriction to $\Z[S_E]_0$ and the right one given by the obvious inclusion. It will be enough to show that there exists a class
\[ \alpha_3(E,S,S') \in H^2(\Gamma_{E/F},\tx{Hom}_S(\Z[S'_E]_0,O_{E,S'}^\times)) \]
mapping to $\alpha_3(E,S)$ via the left map and to $\alpha_3(E,S')$ via the right map. In order to do this, we consider
\[ \tx{Hom}((B_{E,S}),(A_{E,S})) \lw \tx{Hom}_S((B_{E,S'}),(A_{E,S'})) \rw \tx{Hom}((B_{E,S'}),(A_{E,S'})) \]
and show that there exists a class
\[ \alpha(E,S,S') \in H^2(\Gamma_{E/F},\tx{Hom}_S((B_{E,S'}),(A_{E,S'}))) \]
mapping to $\alpha(E,S)$ under the left map and mapping to $\alpha(E,S')$ under the right map. Indeed we have the sequence
\[\begin{aligned} 0\rw\tx{Hom}_S((B_{E,S'}),(A_{E,S'}))&\rw \tx{Hom}_S(\Z[S'_E],J_{E,S'}) \times \tx{Hom}(\Z,C(E))\rw\\
&\rw \tx{Hom}(\Z[S'_E],C(E)) \rw 0
\end{aligned}\]
in which the last map sends $(f_2,f_1)$ to $a\circ f_2 - f_1\circ b$. This map is surjective, because $\tx{Hom}(\Z[S'_E],J_{E,S})$ is a subgroup of $\tx{Hom}_S(\Z[S'_E],J_{E,S'})$, and since $\Z[S'_E]$ is free and $a:J_{E,S} \rw C(E)$ is surjective we see that $f_2 \mapsto a\circ f_2$ is already surjective. The above sequence is thus exact. We obtain maps
\[ \xymatrix{
&\tx{Hom}_S((B_{E,S'}),(A_{E,S'}))\ar[dl]\ar[d]\ar[dr]&\\
\tx{Hom}_S(\Z[S'_E]_0,O_{E,S'}^\times)&\tx{Hom}_S(\Z[S'_E],J_{E,S'})&\tx{Hom}(\Z,C(E)),
}\]
the product of the vertical and right-diagonal maps is an injection and the long exact cohomology sequence associated to the above exact sequence shows that this injection induces an injection on the level of $H^2(\Gamma_{E/F},-)$. Under the Shapiro isomorphism the group $H^2(\Gamma_{E/F},\tx{Hom}_S(\Z[S'_E],J_{E,S'}))$ is identified with
\[ \prod_{v \in S} H^2(\Gamma_{E_{\dot v}/F_v},J_{E,S}) \times \prod_{v \in S' \sm S} H^2(\Gamma_{E_{\dot v}/F_v},J_{E,S'}). \]
We apply Tate's construction of $\alpha_2(E,S')$ in this setting and obtain a class $\alpha_2(E,S,S') \in H^2(\Gamma_{E/F},\tx{Hom}_S(\Z[S'_E],J_{E,S}))$. The pair $(\alpha_2(E,S,S'),\alpha_1(E))$ is seen to belong to the group $H^2(\Gamma_{E/F},\tx{Hom}_S((B_{E,S'}),(A_{E,S'})))$. This is the class $\alpha(E,S,S')$ that we wanted to construct, and the class $\alpha_3(E,S,S')$ is the image of $\alpha(E,S,S')$ under the left diagonal map above.

\end{proof}

We will now discuss the variation of \eqref{eq:tnisotorus} with respect to $E$. This has been studied by Kottwitz in \cite[\S8.3-\S8.5]{KotBG}. Since the discussion there is embedded in the theory of $B(G)$, we will reproduce the arguments here, in a modified form suggested to us by Ta\"ibi.

Let $K/E/F$ be a tower of finite Galois extensions and $S \subset V_F$ a set of places which satisfies Conditions \ref{cnds:tate} with respect to both $E$ and $K$. We consider the two maps
\[ \xymatrix{\Z[S_E]\ar@<1ex>[r]^{p_{K/E}}&\Z[S_K]\ar@<1ex>[l]^{j_{K/E}}}, \]
where $p_{K/E}([w])=\sum_{u|w} |\Gamma_{K/E,u}| \cdot [u]$ and $j_{K/E}([u])=[p(u)]$, with $p : S_K \to S_E$ being the natural projection. Written in slightly different form, we have
\begin{eqnarray*}
p_{K/E}(\sum_{w \in S_E} n_w [w])&=&\sum_{u \in S_K} |\Gamma_{K/E,u}|\cdot n_{p(u)}\cdot [u]\\
j_{K/E}(\sum_{u \in S_K} n_u [u])&=&\sum_{w \in S_E} (\sum_{u|w} n_u) \cdot [w].
\end{eqnarray*}

Both maps are equivariant with respect to the action of $\Gamma_{K/F}$ and satisfy $j_{K/E} \circ p_{K/E}=[K:E] \cdot \tx{id}$ and $p_{K/E} \circ j_{K/E} = N_{K/E}$. The restriction of $p_{K/E}$ to $\Z[S_E]_0$ takes image in $\Z[S_K]_0$, and conversely the restriction of $j_{K/E}$ to $\Z[S_K]_0$ takes image in $\Z[S_E]_0$.

For any $\Gamma_{E/F}$-module $M$ we write $M[S_E]=M \otimes_\Z \Z[S_E]$. The $\Gamma_{K/F}$\-equivariance of $p_{K/E}$ and $j_{K/E}$ implies in particular $j_{K/E} \circ N_{K/F} = N_{K/F} \circ j_{K/E} = [K:E]N_{E/F} \circ j_{K/E}$ and $N_{K/F}\circ p_{K/E} = p_{K/E} \circ N_{K/F} = [K:E]p_{K/E} \circ N_{E/F}$. This shows that if we interpret the cohomology groups $\hat H^{-1}(\Gamma_{E/F},M[S_E]_0)$ resp. $\hat H^0(\Gamma_{E/F},M[S_E]_0)$ as the kernel resp. cokernel of the norm map $N_{E/F} : (M[S_E]_0)_{\Gamma_{E/F}} \to M[S_E]_0^{\Gamma_{E/F}}$, then $j_{K/E}$ and $p_{K/E}$ induce maps
\begin{eqnarray*}
j_{K/E} : \hat H^{0}(\Gamma_{K/F},M[S_K]_0) \to \hat H^{0}(\Gamma_{E/F},M[S_E]_0)\\
p_{K/E} : \hat H^{-1}(\Gamma_{E/F},M[S_E]_0) \to \hat H^{-1}(\Gamma_{K/F},M[S_K]_0).
\end{eqnarray*}
However, we will need the map $j_{K/E}$ in degree $-1$ and the map $p_{K/E}$ in degree $0$, where their existence is slightly less obvious.

\begin{lem} The map $p_{K/E} : H^0(\Gamma_{E/F},M[S_E]_0) \to H^0(\Gamma_{K/F},M[S_K]_0)$ transports $N_{E/F}(M[S_E]_0)$ into $N_{K/F}(M[S_K]_0)$ and thus induces a well-defined map
\[ p_{K/E} : \hat H^{0}(\Gamma_{E/F},M[S_E]_0) \to \hat H^{0}(\Gamma_{K/F},M[S_K]_0). \]
\end{lem}
\begin{proof} Choose a section $s : S_E \to S_K$ of the natural projection $p : S_K \to S_E$. Given $\sum_w n_w[w] \in M[S_E]_0$ we have $\sum_w n_w [s(w)] \in M[S_K]_0$. A short computation reveals that $p_{K/E}N_{E/F}(\sum_w n_w[w])=N_{K/F}(\sum_w n_w [s(w)])$.
\end{proof}

In order to obtain the map $j_{K/E}$ in degree $-1$ we need to assume that $M=Y$ is \emph{torsion-free}. Then $j_{K/E} \circ N_{K/F} = [K:E]N_{E/F} \circ j_{K/E}$ implies that there is a well-defined map
\begin{eqnarray*}
j_{K/E} : \hat H^{-1}(\Gamma_{K/F},Y[S_K]_0) \to \hat H^{-1}(\Gamma_{E/F},Y[S_E]_0).
\end{eqnarray*}

\begin{lem}[Kottwitz] \label{lem:tninftori}
The diagrams
\[ \xymatrix{
	\hat H^{-1}(\Gamma_{E/F},Y[S_E]_0)\ar[r]&H^1(\Gamma_{E/F},T(O_{E,S}))\ar[d]^{\tx{Inf}}\\
	\hat H^{-1}(\Gamma_{K/F},Y[S_K]_0)\ar[u]_{j_{K/E}}\ar[r]&H^1(\Gamma_{K/F},T(O_{K,S}))\\
}\]
and
\[ \xymatrix{
	\hat H^0(\Gamma_{E/F},Y[S_E]_0)\ar[r]\ar[d]^{p_{K/E}}&H^2(\Gamma_{E/F},T(O_{E,S}))\ar[d]^{\tx{Inf}}\\
	\hat H^0(\Gamma_{K/F},Y[S_K]_0)\ar[r]&H^2(\Gamma_{K/F},T(O_{K,S}))\\
}\]
commute, where all horizontal maps are given by \eqref{eq:tnisotorus}.
\end{lem}

\begin{proof}
We present Ta\"ibi's modification of Kottwitz's arguments. As a first step, we will show that the commutativity of the second diagram implies that of the first. We choose a finitely generated free $\Z[\Gamma_{E/F}]$-module $Y'$ with a $\Gamma_{E/F}$-equivariant surjection $Y' \to Y$ and let $Y''$ be the kernel of this surjection. This leads to the exact sequences $1 \to T'' \to T' \to T \to 1$
of tori defined over $F$ and split over $E$, whose cocharacter modules are given by $Y''$, $Y'$, and $Y$. The sequence of $O_{E,S}$-points of these tori is also exact -- this follows from $H^1(\Gamma_{E,S},T''(O_S)) \cong H^1(\Gamma_{E,S},O_S^\times)^{\tx{rk}(T')}=0$ by \cite[Prop. 8.3.11]{NSW08} and Conditions \ref{cnds:tate}. The same is true for the sequence of $O_{K,S}$-points.

\begin{lem}
The diagram
\[ \xymatrix{
	\hat H^{-1}(\Gamma_{E/F},Y[S_E]_0)\ar[r]&\hat H^0(\Gamma_{E/F},Y''[S_E]_0)\ar[d]^{p_{K/E}}\\
	\hat H^{-1}(\Gamma_{K/F},Y[S_K]_0)\ar[u]_{j_{K/E}}\ar[r]&\hat H^0(\Gamma_{K/F},Y''[S_K]_0)
}\]
commutes, where the horizontal maps are the connecting homomorphisms.
\begin{proof}
This is an immediate computation. Recall that the connecting homomorphism is given by the norm map on $Y'$.
\end{proof}
\end{lem}
Since $T'$ is induced, the connecting homomorphism $H^1(\Gamma_{E/F},T(O_{E,S})) \to H^2(\Gamma_{E/F},T''(O_{E,S}))$ and its $K/F$-analog are injective. By naturality of inflation the commutativity of the first diagram in Lemma \ref{lem:tninftori} is equivalent to the commutativity of
\[ \xymatrix{
	\hat H^{-1}(\Gamma_{E/F},Y[S_E]_0)\ar[r]&H^2(\Gamma_{E/F},T''(O_{E,S}))\ar[d]^{\tx{Inf}}\\
	\hat H^{-1}(\Gamma_{K/F},Y[S_K]_0)\ar[u]_{j_{K/E}}\ar[r]&H^2(\Gamma_{K/F},T''(O_{K,S}))\\
}\]
where the horizontal maps are now given by composing the Tate-Nakayama isomorphisms with the connecting homomorphisms. But the Tate-Nakayama isomorphisms are given by cup-product with the classes $\alpha_3(E,S)$ and $\alpha_3(K,S)$. Since cup-product respects connecting homomorphisms, we now see that the commutativity of the above diagram is implied by the commutativity of the second diagram in Lemma \ref{lem:tninftori}, applied to the torus $T''$.

To prove the latter, we observe that the cup-product in degrees $(0,2)$ is simply given by tensor product and hence the desired commutativity follows from the following lemma relating the Tate classes $\alpha_3(E,S)$ and $\alpha_3(K,S)$.

\begin{lem} The image of $\alpha_3(E,S)$ under the inflation map
\[ H^2(\Gamma_{E/F},\tx{Hom}(\Z[S_E]_0,O_{E,S})) \to H^2(\Gamma_{K/F},\tx{Hom}(\Z[S_E]_0,O_{K,S})) \]
coincides with the image of $\alpha_3(K,S)$ under the map \[ H^2(\Gamma_{K/F},\tx{Hom}(\Z[S_K]_0,O_{K,S})) \to H^2(\Gamma_{K/F},\tx{Hom}(\Z[S_E]_0,O_{K,S})) \] induced by $p_{K/E}$.
\end{lem}
\begin{proof}
The map $p_{K/E} : \Z[S_E] \to \Z[S_K]$ induces a $\Gamma_{K/F}$-equivariant map of exact sequences $(B_{E,S}) \to (B_{K,S})$, whose components we denote by $p_3$, $p_2$, and $p_1$ for short. Explicitly, $p_3$ and $p_2$ are given by the same formula that defines $p_{K/E}$, while $p_1 : \Z \to \Z$ is multiplication by $[K:E]$. Recalling that $\alpha_3(E,S)$ is the image of $\alpha(E,S) \in H^2(\Gamma_{E/F},\tx{Hom}((B_{E,S}),(A_{E,S})))$, it will be enough to show that the image of $\alpha(E,S)$ under the inflation map
\[ H^2(\Gamma_{E/F},\tx{Hom}((A_{E,S}),(B_{E,S}))) \to H^2(\Gamma_{K/F},\tx{Hom}((A_{E,S}),(B_{K,S}))) \]
coinsides with the image of $\alpha(K,S)$ under the map
\[ H^2(\Gamma_{K/F},\tx{Hom}((A_{K,S}),(B_{K,S}))) \to H^2(\Gamma_{K/F},\tx{Hom}((A_{E,S}),(B_{K,S}))) \]
given by $p=(p_3,p_2,p_1)$. Recalling that $\alpha(E,S)$ is determined by $\alpha_2(E,S) \in H^2(\Gamma_{E/F},\tx{Hom}(\Z[S_E],J_{E,S}))$ and $\alpha_1(E,S) \in H^2(\Gamma_{E/F},\tx{Hom}(\Z,C(E)))$, it is enough to prove the above statement for $\alpha_2$ and $\alpha_1$, with respect to the maps $p_2$ and $p_1$.

Identifying $\tx{Hom}(\Z,C(E))$ with $C(E)$, the statement about $\alpha_1$ is immediate from the compatibility of the canonical classes in global class field theory with inflation.

For the statement about $\alpha_2$ we consider the commutative diagram

\[\xymatrix{
H^2(\Gamma_{E/F},\tx{Hom}(\Z[S_E],J_{E,S}))\ar[d]^{\tx{Inf}}\ar[r]&\prod_{v \in S} H^2(\Gamma_{E/F,w},J_{E,S})\ar[d]\\
H^2(\Gamma_{K/F},\tx{Hom}(\Z[S_E],J_{K,S}))\ar[r]&\prod_{v \in S} H^2(\Gamma_{K/F,w},J_{K,S})\\
H^2(\Gamma_{K/F},\tx{Hom}(\Z[S_K],J_{K,S}))\ar[u]_{p_2}\ar[r]&\prod_{v \in S} H^2(\Gamma_{K/F,u},J_{K,S})\ar[u]\\
}\]
where on the right we choose arbitrarily for each $v \in S$ places $w \in S_E$ and $u \in S_K$, with $u|w|v$. The horizontal maps are the Shapiro isomorphisms. The top right vertical map is again inflation, induced by the projection $\Gamma_{K/F,w} \to \Gamma_{E/F,w}$ and the inclusion $J_{E,S} \to J_{K,S}$, while the bottom right vertical map is given by $[K_u:E_w] \cdot \tx{cor}$, with corestriction induced by the inclusion $\Gamma_{K/F,u} \to \Gamma_{K/F,w}$. Now $\alpha_2(E,S)$ in the top left corner corresponds to the collection $(\alpha_{E_w/F_v})_{v \in S}$ in the top right corner, where $\alpha_{E_w/F_v}$ is the image of the canonical class in $H^2(\Gamma_{E/F,w},E_w^\times)$ under the natural inclusion $E_w^\times \to J_{E,S}$. Thus we are led to consider the diagram
\[ \xymatrix{
H^2(\Gamma_{E/F,w},J_{E,S})\ar[d]&H^2(\Gamma_{E/F,w},E_w^\times)\ar[l]\ar[d]\\
H^2(\Gamma_{K/F,w},J_{K,S})&H^2(\Gamma_{K/F,w},K_w^\times)\ar[l]\\
H^2(\Gamma_{K/F,u},J_{K,S})\ar[u]&H^2(\Gamma_{K/F,u},K_u^\times)\ar[l]\ar[u]
}
\]
where $K_w = E_w \otimes_E K$, the horizontal maps are the natural inclusions, the left vertical maps are as in the previous digram, and the right vertical maps are the corresponding maps -- the top is inflation, and the bottom is $[K:E]$ times corestriction. We now have to show that the canonical classes in $H^2(\Gamma_{E/F,w},E_w^\times)$ and $H^2(\Gamma_{K/F,u},K_u^\times)$ meet in $H^2(\Gamma_{K/F,w},K_w^\times)$. But the corestriction homomorphism $H^2(\Gamma_{K/F,u},K_u^\times) \to H^2(\Gamma_{K/F,w},K_w^\times)$ is an isomorphism whose inverse is given by the projection $K_w^\times \to K_u^\times$ followed by restriction along the inclusion $\Gamma_{K/F,u} \to \Gamma_{K/F,w}$. Composing both right vertical maps with this inverse gives on the one hand $[K:E] \cdot \tx{id}$ on $H^2(\Gamma_{K/F,u},K_u^\times)$, and on the other hand the usual inflation map $H^2(\Gamma_{E/F,w},E_w^\times) \to H^2(\Gamma_{K/F,u},K_u^\times)$ and local class field theory asserts that the fundamental classes for $E_w/F_v$ and $K_u/F_v$ meet under these two maps.
\end{proof}
With this the proof of Lemma \ref{lem:tninftori} is complete.
\end{proof}

We will soon need a variant of Lemma \ref{lem:tninftori} for finite multiplicative groups instead of tori. As we saw above, in this case the map $j_{K/E}$ does not induce a map on the level of $\hat H^{-1}$. In order to state the result, we will need to work with a map in the opposite direction, that for torsion-free $M=Y$ is an inverse of $j_{K/E}$. We will now discuss this inverse and then restate the lemma with it. Fix an arbitrary section $s : S_E \rw S_K$ of the natural projection $p : S_K \rw S_E$ and let $s_! : \Z[S_E]_0 \rw \Z[S_K]_0$ be the map sending $\sum_{w \in S_E} n_w [w]$ to $\sum_{w \in S_E} n_w [s(w)]$.

\begin{lem} \label{lem:tnshriek}
Assume that for each $\sigma \in \Gamma_{K/E}$ there exists $u \in S_K$ such that $\sigma u = u$. For any $\Gamma_{E/F}$-module $M$ the map $s_!$ induces a well-defined map
\[ !:\hat H^{-1}(\Gamma_{E/F},M[S_E]_0) \rw \hat H^{-1}(\Gamma_{K/F},M[S_K]_0) \]
that is functorial in $M$ and independent of the choice of $s$. When $M=Y$ is torsion-free it is a right-inverse to the map $j_{K/E} : \hat H^{-1}(\Gamma_{K/F},M[S_K]_0) \rw \hat H^{-1}(\Gamma_{E/F},M[S_E]_0)$.
\end{lem}
\begin{proof}
Let $s,s'$ be two sections and let $x \in M[S_E]_0$ be given as $x=\sum_w n_w[w]$ with $n_w \in M$. We claim that $s_!(x) = s'_!(x)$ in $H_0(\Gamma_{K/F},M[S_K]_0)$. We have
\[
s_!(x)-s'_!(x) = \sum_w (n_{w}[s(w)]-n_{w}[s'(w)]),
\]
where the sum is taken over the finite set $\{w \in S_E|n_w \neq 0\}$. It is enough to show that each summand belongs to $I_{K/F}M[S_K]_0$, so fix $w_0 \in \{w \in S_E|n_w \neq 0\}$ and let $\sigma \in \Gamma_{K/E}$ be such that $\sigma(s(w_0))=s'(w_0)$. By assumption there exists $u_0 \in S_K$ with $\sigma u_0=u_0$. Then we have
\[ s_!(x)-s'_!(x) = (n_{w_0}[s(w_0)]-n_{w_0}[u_0]) -\sigma(n_{w_0}[s(w_0)]-n_{w_0}[u_0]), \]
noting that $\Gamma_{K/E}$ acts trivially on $M$, and the claim is proved.

One checks easily that for $x \in M[S_E]_0$ and $\sigma \in \Gamma_{K/F}$ we have $\sigma(s_!(x)) = (\sigma s)_!(\sigma x)$. From this and the claim we just proved it follows that $s_!$ maps $I_{E/F}M[S_E]_0$ to $I_{K/F}M[S_K]_0$. We conclude that $s_!$ determines a well-defined map $H_0(\Gamma_{E/F},M[S_E]_0) \rw H_0(\Gamma_{K/F},M[S_K]_0)$ and that this map is independent of $s$.

We now argue that if $x \in M[S_E]_0$ satisfies $N_{E/F}(x)=0$, then $N_{K/F}(s_!(x))=0$. Write $x=\sum_{w \in S_E} n_w [w]$. Then $s_!(x)=\sum_{u \in S_K} \delta_{u,sp(u)}n_{p(u)}[u]$, where $\delta_{u,u'}$ is equal to $1$ if $u=u'$ and to $0$ otherwise. A quick computation shows
\[ N_{K/F}(s_!(x))=\sum_{u \in S_K}\sum_{\tau \in \Gamma_{E/F}} |\tx{Stab}(s(\tau^{-1}p(u)),\Gamma_{K/E})|\tau n_{\tau^{-1}p(u)}[u]. \]
For fixed $u \in S_K$ and $\tau \in \Gamma_{E/F}$ let $\sigma \in \Gamma_{K/F}$ be such that $\sigma u = s(\tau^{-1}p(u))$. Then
\[ \tx{Stab}(s(\tau^{-1}p(u)),\Gamma_{K/E}) = \sigma\tx{Stab}(u,\Gamma_{K/F})\sigma^{-1} \cap \Gamma_{K/E}. \]
Since $E/F$ is Galois, the cardinality of this group is equal to the cardinality of $\tx{Stab}(u,\Gamma_{K/E})$. We thus conclude that
\[ N_{K/F}(s_!(x))=\sum_{u \in S_K}|\tx{Stab}(u,\Gamma_{K/E})|\sum_{\tau \in \Gamma_{E/F}}\tau n_{\tau^{-1}p(u)}[u]. \]
The inner sum vanishes due to the assumption $N_{E/F}(x)=0$.
\end{proof}

\begin{cor} \label{cor:tnshriek}
Under the assumptions of Lemma \ref{lem:tnshriek}, for $M=Y$ torsion-free the diagram
\[ \xymatrix{
	\hat H^{-1}(\Gamma_{E/F},Y[S_E]_0)\ar[r]\ar[d]^{!}&H^1(\Gamma_{E/F},T(O_{E,S}))\ar[d]^{\tx{Inf}}\\
	\hat H^{-1}(\Gamma_{K/F},Y[S_K]_0)\ar[r]&H^1(\Gamma_{K/F},T(O_{K,S}))\\
}\]
commutes and both vertical maps are isomorphisms.
\end{cor}
\begin{proof}
Lemma \ref{lem:inf1} below shows that the right vertical map is an isomorphism. Then so ist the left vertical map in Lemma \ref{lem:tninftori}, and then $!$ must be its inverse by Lemma \ref{lem:tnshriek}, from which the commuttivity of the square follows.
\end{proof}

We now want to relate the group $H^i(\Gamma_{E/F},T(O_{E,S}))$ to $H^i(\Gamma_S,T(O_S))$ and $H^i(\Gamma,T(\ol{F}))$ in the cases $i=1,2$.

\begin{lem} \label{lem:inf1}
	The inflation map $H^i(\Gamma_{E/F},T(O_{E,S})) \rw H^i(\Gamma_S,T(O_S))$ is bijective for $i=1$ and injective for $i=2$.
\end{lem}
\begin{proof}
The inflation-restriction sequence takes the form
\[ \begin{aligned}
1&\rw H^1(\Gamma_{E/F},T(O_{E,S})) \rw H^1(\Gamma_S,T(O_S))\rw  H^1(\Gamma_{E,S},T(O_S))\rw \\
&\rw H^2(\Gamma_{E/F},T(O_{E,S}))\rw H^2(\Gamma_S,T(O_S))
\end{aligned} \]
Since $T$ splits over $E$ we have
\[ H^1(\Gamma_{E,S},T(O_S)) \cong H^1(\Gamma_{E,S},O_S^\times)^{\tx{rk}(T)} = 0 \]
by \cite[Prop. 8.3.11]{NSW08} and Conditions \ref{cnds:tate} and the proof is complete.
\end{proof}

\begin{lem} \label{lem:inf2}
Assume that for each $w \in V_E$ there exists $w' \in S_E$ with the property that $\tx{Stab}(w,\Gamma_{E/F})=\tx{Stab}(w',\Gamma_{E/F})$. The inflation map $H^i(\Gamma_S,T(O_S)) \rw H^i(\Gamma,T(\ol{F}))$ is injective for $i=1,2$.
\end{lem}
\begin{proof}
We begin by constructing a $\Gamma_{E/F}$-invariant left-inverse of the natural inclusion $S_E \rw V_E$. For each $v \in V_F$ choose one $\dot v \in V_E$ and one $\ddot v \in S_E$ such that $p(\dot v)=v$ and $\tx{Stab}(\dot v,\Gamma_{E/F})=\tx{Stab}(\ddot v,\Gamma_{E/F})$. If $v \in S$ then we demand $\dot v = \ddot v$. Define a map $f : V_E \rw S_E$ as follows. Given $w \in V_E$ let $v=p(w)$ and choose $\sigma \in \Gamma_{E/F}$ such that $w=\sigma \dot v$. Set $f(w)=\sigma \ddot v$. The place $\sigma\ddot v$ is independent of the choice of $\sigma$ and so this map is well-defined. It is moreover $\Gamma_{E/F}$-equivariant, because given $\tau \in \Gamma_{E/F}$ we have $p(\tau w)=v$, $\tau w = \tau\sigma\dot v$, and hence $f(\tau w)=\tau\sigma\ddot v= \tau f(w)$. By construction $f$ is left-inverse to the natural inclusion $S_E \rw V_E$.

Now $f$ extends linearly to a map $Y[V_E] \rw Y[S_E]$, which in turn restricts to a map $Y[V_E]_0 \rw Y[S_E]_0$ that is $\Gamma_{E/F}$-equivariant and left-inverse to the natural inclusion $Y[S_E]_0 \rw Y[V_E]_0$. We conclude that for all $i \in \Z$ the map $\hat H^i(\Gamma_{E/F},Y[S_E]_0) \rw \hat H^i(\Gamma_{E/F},Y[V_E]_0)$ has a left-inverse and hence must be injective.

Applying Lemmas \ref{lem:tnshriek} and \ref{lem:inf1} we obtain the diagram
\[ \xymatrix{
\hat H^{-1}(\Gamma_{E/F},Y[S_E]_0)\ar[r]^\cong\ar[d]&H^1(\Gamma_{E/F},T(O_{E,S}))\ar[r]^\cong\ar[d]&H^1(\Gamma_S,T(O_S))\ar[d]\\
\hat H^{-1}(\Gamma_{E/F},Y[V_E]_0)\ar[r]^\cong&H^1(\Gamma_{E/F},T(E))\ar[r]^\cong&H^1(\Gamma,T(\ol{F}))
}\]
and the injectivity of the left vertical map implies the injectivity of the right vertical map.

To treat the case $i=2$ we consider the diagram
\[\xymatrix{
0\ar[r]&H^2(\Gamma_{E/F},T(O_{E,S}))\ar[r]\ar[d]&H^2(\Gamma_S,T(O_S))\ar[r]\ar[d]&H^2(\Gamma_{E,S},T(O_S))\ar[d]\\
0\ar[r]&H^2(\Gamma_{E/F},T(E))\ar[r]&H^2(\Gamma,T(\ol{F}))\ar[r]&H^2(\Gamma_E,T(\ol{F}))
}\]
By the five-lemma, we need to show the injectivity of the left and right vertical maps. The injectivity of the left vertical map follows from Lemma \ref{lem:tnstori} and the injectivity of $\hat H^0(\Gamma_{E/F},Y[S_E]_0) \rw \hat H^0(\Gamma_{E/F},Y[V_E]_0)$. Turning to the right vertical map, since $T$ splits over $E$ it is enough to prove the injectivity of the inflation map $H^2(\Gamma_{E,S},O_S^\times) \rw H^2(\Gamma_E,\ol{F}^\times)$. For this we consider the exact sequences $1 \rw O_S^\times \rw I_S \rw C_S \rw 1$ and $1 \rw \ol{F}^\times \rw I \rw C \rw 1$ and obtain the diagram

\[\xymatrix{
	{\underbrace{H^1(\Gamma_{E,S},C_S)}_{=0}}\ar[r]&H^2(\Gamma_{E,S},O_S^\times)\ar[r]\ar[d]&H^2(\Gamma_{E,S},I_S)\ar[d]\\
	{\overbrace{H^1(\Gamma_{E},C)}^{=0}}\ar[r]&H^2(\Gamma_{E},\ol{F}^\times)\ar[r]&H^2(\Gamma_{E},I)\\
}
\]
The map on the right is the map
\[ \varinjlim_{F_S/K/E}H^2(\Gamma_{K/E},I_{K,S}) \rw \varinjlim_{\ol{F}/K/E}H^2(\Gamma_{K/E},I_K), \]
where the transition maps in both limits are the usual inflation maps and the map $H^2(\Gamma_{K/E},I_{K,S}) \rw H^2(\Gamma_{K/E},I_K)$ is induced by the inclusion $I_{K,S} \rw I_K$. Since $I_{K,S}$ is a direct factor of $I_K$, the map $H^2(\Gamma_{K/E},I_{K,S}) \rw H^2(\Gamma_{K/E},I_K)$ is injective. Moreover, the vanishing of $H^1(\Gamma_{K/E},I_K)$ and $H^1(\Gamma_{K/E},I_{K,S})$ for all $K$ shows that the transition maps in the two limits are injective. It follows that the map $H^2(\Gamma_{E,S},I_S) \rw H^2(\Gamma_{E},I)$ in the above diagram is also injective, and then so is the map $H^2(\Gamma_{E,S},O_S^\times) \rw H^2(\Gamma_E,\ol{F}^\times)$. This completes the proof of the case $i=2$.
\end{proof}

\subsection{A description of $H^2(\Gamma,Z)$} \label{sub:h2z}

Let $Z$ be a finite multiplicative group defined over $F$. Write $A=X^*(Z)$ and $A^\vee=\tx{Hom}(A,\Q/\Z)$. The purpose of this Subsection is to establish a functorial isomorphism
\begin{equation} \label{eq:tnziso} \Theta: \varinjlim_{E,S} \hat H^{-1}(\Gamma_{E/F},A^\vee[S_E]_0) \rw H^2(\Gamma,Z). \end{equation}
We can think of $A^\vee[S_E]_0=\Z[S_E]_0 \otimes A^\vee$ either as formal finite linear combinations of elements of $S_E$ with coefficients in $A^\vee$, the sum of which equals zero, or as maps $S_E \to A^\vee$ the sum of whose values is zero. When we use the latter interpretation we will sometimes write $\tx{Maps}(S_E,A^\vee)_0$ instead.

\begin{fct} \label{fct:pair1}
	Let $n$ be a multiple of $\exp(Z)$. We have a functorial isomorphism
	\[ \Phi_{A,n}: A^\vee \rw \tx{Hom}(\mu_n,Z),\qquad a(\Phi_{A,n}(\lambda)(x)) = x^{n\lambda(a)}, \]
	where $a \in A$, $\lambda \in A^\vee$, and $x \in \mu_n$. If $n|m$ and $x \in \mu_m$ we have
	\[ \Phi_{A,m}(\lambda)(x)=\Phi_{A,n}(\lambda)(x^\frac{m}{n}). \]
\end{fct}

Choose a finite Galois extension $E/F$ splitting $Z$ and a finite full set $S \subset V_F$ satisfying the following conditions.

\begin{cnds} \label{cnds:tnz}
\begin{enumerate}
\item $S$ contains all archimedean places and all places that ramify in $E$.
\item Every ideal class in $C(E)$ contains an ideal supported on $S_E$.
\item $\exp(Z) \in \N_S$.
\item For each $w \in V_E$ there exists $w' \in S_E$ with $\tx{Stab}(w,\Gamma_{E/F})=\tx{Stab}(w',\Gamma_{E/F})$.
\end{enumerate}
\end{cnds}
Note that such finite sets exist and if $S$ satisfies these conditions and $S' \supset S$, then $S'$ also satisfies these conditions.

\begin{fct} \label{fct:pair2}
	Let $n$ be a multiple of $\exp(Z)$. We have a functorial isomorphism
	\[ \Phi_{A,S,n} : \tx{Maps}(S_E,A^\vee)_0 \rw \tx{Hom}(\tx{Maps}(S_E,\mu_n)/\mu_n,Z), \]
	which for $g \in \tx{Maps}(S_E,A^\vee)_0, f \in \tx{Maps}(S_E,\mu_n), a \in A$ is given by the formula
	\[ a(\Phi_{A,S,n}(g)(f)) = \prod_{w \in S_E} f(w)^{ng(w,a)}. \]
	If $n|m$ then $\Phi_{A,S,m}(g)(f)=\Phi_{A,S,n}(g)(f^\frac{m}{n})$.
\end{fct}

Let $\alpha_3(E,S) \in Z^2(\Gamma_{E/F},\tx{Hom}(\Z[S_E]_0,O_{E,S}^\times))$ represent Tate's class discussed in Subsection \ref{sub:tatetori}. We have
\[ \tx{Hom}(\Z[S_E]_0,O_{E,S}^\times) = \tx{Maps}(S_E,O_{E,S}^\times)/O_{E,S}^\times \hrw \tx{Maps}(S_E,O_{S}^\times)/O_{S}^\times, \]
where we are identifying $O_{E,S}^\times$ and $O_S^\times$ as the subgroups of constant functions. By \cite[Proposition 8.3.4]{NSW08} the group $O_S^\times$ is $\N_S$-divisible. For any $n \in \N_S$ the $n$-th power map fits into the exact sequence
\[ 1 \rw \frac{\tx{Maps}(S_E,\mu_n)}{\mu_n}\rw \frac{\tx{Maps}(S_E,O_{S}^\times)}{O_{S}^\times} \rw \frac{\tx{Maps}(S_E,O_{S}^\times)}{O_{S}^\times} \rw 1. \]
Fix a co-final sequence $n_i \in \N_S$ as well as functions
\[ k_i :\frac{\tx{Maps}(S_E,O_{S}^\times)}{O_{S}^\times} \to \frac{\tx{Maps}(S_E,O_{S}^\times)}{O_{S}^\times} \]
satisfying $k_i(x)^{n_i}=x$ and $k_{i+1}(x)^\frac{n_{i+1}}{n_i}=k_i(x)$. Then we have
\[ dk_i\alpha_3(E,S) \in Z^{3,2}(\Gamma_S,\Gamma_{E/F},\tx{Maps}(S_E,\mu_{n_i})/\mu_{n_i})\]
where we are using the notation $Z^{i,j}$ from \cite[\S4.3]{KalRI}.

We now define the map
\[ \Theta_{E,S} : \hat H^{-1}(\Gamma_{E/F},A^\vee[S_E]_0) \rw H^2(\Gamma_S,Z(O_S)),\quad g \mapsto dk_i\alpha_3(E,S) \sqcup_{E/F} g, \]
where we have identified $A^\vee[S_E]_0$ with $\tx{Maps}(S_E,A^\vee)_0$, and have employed the unbalanced cup product of \cite[\S4.3]{KalRI} as well as the pairing
\begin{equation} \label{eq:pair2}  \tx{Maps}(S_E,\mu_{n_i})/\mu_{n_i} \otimes \tx{Maps}(S_E,A^\vee)_0 \rw Z \end{equation}
provided by $\Phi_{E,S,n_i}$ of Fact \ref{fct:pair2}. Here we must choose $n_i$ to be a multiple of $\exp(Z)$, which is possible since $n_i$ are a co-final sequence in $\N_S$ and $\exp(Z) \in \N_S$ by assumption. Moreover, the map $\Theta_{E,S}$ is independent of the choice of $n_i$.

\begin{pro} \label{pro:tnz_es}
	The map $\Theta_{E,S}$ is a functorial injection independent of the choices of $\alpha_3(E,S)$ and $k_i$.
\end{pro}

The proof is based on the following lemma, which will also have other uses later.

\begin{lem} \label{lem:tnzdiag}
Let $T$ be a torus defined over $F$ and split over $E$ and let $Z \rw T$ be an injection with cokernel $\bar T$. We write $Y=X_*(T)$ and $\bar Y=X_*(\bar T)$.
Then the following diagram commutes and its columns are exact.

\begin{equation} \label{eq:tnzdiag}
\xymatrix{
\hat H^{-1}(\Gamma_{E/F},Y[S_E]_0)\ar[r]^\cong_{\tx{TN}}\ar[d]&H^1(\Gamma_{E/F},T(O_{E,S}))\ar[r]^\cong\ar[d]&H^1(\Gamma_S,T(O_S))\ar[d]\\
\hat H^{-1}(\Gamma_{E/F},\bar Y[S_E]_0)\ar[r]^\cong_{\tx{TN}}\ar[d]&H^1(\Gamma_{E/F},\bar T(O_{E,S}))\ar[r]^\cong&H^1(\Gamma_S,\bar T(O_S))\ar[d]\\
\hat H^{-1}(\Gamma_{E/F},A^\vee[S_E]_0)\ar[rr]^{\Theta_{E,S}}\ar[d]&&H^2(\Gamma_S,Z(O_S))\ar[d]\\
\hat H^{0}(\Gamma_{E/F},Y[S_E]_0)\ar[r]^\cong_{-\tx{TN}}\ar[d]&H^2(\Gamma_{E/F},T(O_{E,S}))\ar@{^(->}[r]\ar[d]&H^2(\Gamma_S,T(O_S))\ar[d]\\
\hat H^{0}(\Gamma_{E/F},\bar Y[S_E]_0)\ar[r]^\cong_{-\tx{TN}}&H^2(\Gamma_{E/F},\bar T(O_{E,S}))\ar@{^(->}[r]&H^2(\Gamma_S,\bar T(O_S))\\^{}
}
\end{equation}
\end{lem}

\begin{proof}
We first explain the arrows. The short exact sequence
\[ 0 \rw Y \rw \bar Y \rw A^\vee \rw 0 \]
remains short exact after tensoring over $\Z$ with the free $\Z$-module $\Z[S_E]_0$, and the left column of the diagram is the associated long exact sequence in Tate cohomology. Writing $X = X^*(T)$ and $\bar X=X^*(\bar T)$, the short exact sequence
\[ 0 \rw \bar X \rw X \rw A \rw 0 \]
remains short exact after applying $\tx{Hom}(-,O_S^\times)$ because $O_S^\times$ is $\N_S$-divisible \cite[Prop. 8.3.4]{NSW08} and we are assuming $\exp(Z) \in \N_S$. The right column of the diagram is the corresponding long exact sequence for $H^i(\Gamma_S,-)$. The horizontal maps labelled ``TN'' are the isomorphisms constructed by Tate and discussed in Subsection \ref{sub:tatetori}. The bottom two that are labelled ``-TN'' are obtained from these isomorphisms by composing them with multiplication by $-1$. The horizontal maps on the right are the inflation maps discussed in Lemma \ref{lem:inf1}.

The only non-obvious commutativity is that of the two squares involving the map $\Theta_{E,S}$. For the top square, let $\bar\Lambda \in \hat Z^{-1}(\Gamma_{E/F},\bar Y[S_E]_0)$. Then $\tx{TN}(\bar\Lambda)=\alpha_3(E,S) \cup \bar\Lambda \in Z^1(\Gamma_{E/F}(\bar T(O_{E,S})))$ and the image of this in $H^2(\Gamma_S,Z(O_S))$ can be computed as follows. Choose $n_i$ to be a multiple of $\exp(Z)$. Then $n_i\bar\Lambda \in \hat Z^{-1}(\Gamma_{E/F},Y[S_E]_0)$ and $k_i\alpha_3(E,S) \in C^{2,2}(\Gamma_S,\Gamma_{E/F},\tx{Maps}(S_E,O_S^\times)/O_S^\times)$. The unbalanced cup product $k_i\alpha_3(E,S) \sqcup_{E/F}n_i\bar\Lambda$ belongs to $C^{1,1}(\Gamma_S,\Gamma_{E/F},T(O_S))$ and lifts the 1-cocycle $\tx{TN}(\bar\Lambda)$. The differential of this $1$-cochain is the image we want. According to \cite[Fact 4.3]{KalRI} it is equal to
\[ d(k_i\alpha_3(E,S) \sqcup_{E/F}n_i\bar\Lambda) = dk_i\alpha_3(E,S) \sqcup_{E/F} n_i\bar\Lambda = \Theta_{E,S}(n_i\bar\Lambda). \]
But the restriction of $n_i\bar\Lambda : \mb{G}_m \rw T$ to $\mu_{n_i}$ takes image in $Z$ and the resulting map $\mu_{n_i} \rw Z$ corresponds to the element of $A^\vee$ that is the image of $\bar\Lambda$ under $\bar Y \rw A^\vee$. This proves the commutativity of the top square involving $\Theta_{E,S}$.

For the commutativity of the bottom square  let $g \in \hat Z^{-1}(\Gamma_{E/F},A^\vee[S_E]_0)$ and let $\bar\Lambda \in \bar Y = \hat C^{-1}(\Gamma_{E/F},\bar Y[S_E]_0)$ be any preimage of $g$. According to the above discussion we have
\[ \Theta_{E/S}(g) = dk_i\alpha_3(E,S) \sqcup_{E/F} n_i\bar\Lambda. \]
Mapping this under $H^2(\Gamma_S,Z(O_S)) \rw H^2(\Gamma_S,T(O_S))$ we can write this using \cite[Fact 3.4]{KalRI} as
\[ d(k_i\alpha_3(E,S) \sqcup n_i\bar\Lambda) - k_i\alpha_3(E,S) \sqcup dn_i\bar\Lambda. \]
The first term is a coboundary in $C^2(\Gamma_S,T(O_S))$. Since $\bar\Lambda \in \hat C^{-1}(\Gamma_{E/F},\bar Y[S_E]_0)$ was a lift of the $(-1)$-cocycle $g$, its differential belongs to $\hat Z^0(\Gamma_{E/F},Y[S_E]_0)$. It follows that
\[ k_i\alpha_3(E,S) \sqcup dn_i\bar\Lambda = \alpha_3(E,S) \cup d\bar\Lambda \]
and this proves the commutativity of the bottom square involving $\Theta_{E,S}$.
\end{proof}

\begin{proof}[Proof of Proposition \ref{pro:tnz_es}]
In the setting of Lemma \ref{lem:tnzdiag} choose $\bar Y$ to be a free $\Z[\Gamma_{E/F}]$-module. Then $\bar Y[S_E]_0$ is also a free $\Z[\Gamma_{E/F}]$-module and consequently both $\hat H^{-1}(\Gamma_{E/F},\bar Y[S_E]_0)$ and $H^1(\Gamma_S,\bar T(O_S))$ vanish, showing that $\Theta_{E,S}$ is the restriction of ``-TN''. The latter is an injective map and is independent of the choices of $\alpha_3(E,S)$ and $k_i$.
\end{proof}

The map $\Theta_{E,S}$ has a local analog that is implicit in the constructions of \cite{KalRI}. To describe it, let $v \in S_{\ol{F}}$ and let $\alpha_v \in Z^2(\Gamma_{E_v/F_v},E_v^\times)$ represent the canonical class. Let $n \in \N$ be a multiple of $\exp(Z)$ and let $k : \ol{F_v} \rw \ol{F_v}$ be such that $k(x)^n=x$. Then we define
\[ \Theta_{E,v} : \hat H^{-1}(\Gamma_{E_v/F_v},A^\vee) \rw H^2(\Gamma_v,Z(\ol{F_v})),\qquad g \mapsto dk\alpha_v \sqcup_{E_v/F_v} \Phi_{A,n}(g). \]
Similar arguments to those employed for $\Theta_{E,S}$ show that $\Theta_{E,v}$ is independent of the choices of $\alpha_v$, $k$, and $n$, and fits into the local analog of Diagram \eqref{eq:tnzdiag}. The following lemma relates the map $\Theta_{E,S}$ to $\Theta_{E,v}$.

\begin{lem} \label{lem:zloc}
Let $v \in S_{\ol{F}}$. We have the commutative diagram
\[ \xymatrix{
	\hat H^{-1}(\Gamma_{E/F},A^\vee[S_E]_0)\ar[r]^-{\Theta_{E,S}}\ar[d]&H^2(\Gamma_S,Z(O_S))\ar[d]\\
	\hat H^{-1}(\Gamma_{E_v/F_v},A^\vee)\ar[r]^-{\Theta_{E,v}}&H^2(\Gamma_v,Z(\ol{F_v}))
}\]
Here the right vertical map is the localization map given by restriction to $\Gamma_v$ followed by the inclusion $Z(O_S)\rw Z(\ol{F_v})$ and the left vertical map is given by restriction to $\Gamma_{E_v/F_v}$ followed by the projection $A^\vee[S_E]_0 \rw A^\vee$ onto the $v$-factor.
\end{lem}
\begin{proof}
	We choose $\bar Y$ to be a free $\Z[\Gamma_{E/F}]$-module. Then $\bar Y$ is also a free $\Z[\Gamma_{E_v/F_v}]$-module and consequently the four cohomology groups $\hat H^{-1}(\Gamma_{E/F},\bar Y[S_E]_0)$, $\hat H^{-1}(\Gamma_{E_v/F_v},\bar Y)$, $H^1(\Gamma_S,\bar T(O_S))$, and $H^1(\Gamma_v,\bar T(\ol{F_v}))$ all vanish. Looking at Diagram \eqref{eq:tnzdiag} and its local analog we see that it is enough to show the commutativity of the diagram below, which stems directly from the construction of the Tate-Nakayama isomorphism for tori in \cite{Tate66}.
	\[ \xymatrix{
		\hat H^0(\Gamma_{E/F},Y[S_E]_0)\ar[r]^{''-TN''}\ar[d]&H^2(\Gamma_S,T(O_S))\ar[d]\\
		\hat H^0(\Gamma_{E_v/F_v},Y)\ar[r]^{''-TN''}&H^2(\Gamma_v,T(\ol{F_v}))
	}\]
\end{proof}

We will now study how the map $\Theta_{E,S}$ behaves when we change $E$ and $S$.

\begin{lem} \label{lem:inf3} The inflation map $H^2(\Gamma_S,Z(O_S)) \rw H^2(\Gamma,Z(\ol{F}))$ is injective.\end{lem}
\begin{proof}
Choose again $\bar Y$ to be a free $\Z[\Gamma_{E/F}]$-module. Then $H^1(\Gamma_S,\bar T(O_S))$ and $H^1(\Gamma,\bar T(\ol{F}))$ vanish. The inflation map in question, composed with the injection $H^2(\Gamma,Z(\ol{F})) \to H^2(\Gamma,T(\ol{F}))$ is equal to the composition of the injection $H^2(\Gamma_S,Z(O_S)) \to H^2(\Gamma_S,T(O_S))$ with the map $H^2(\Gamma_S,T(O_S)) \to H^2(\Gamma,T(\ol{F}))$, which by Lemma \ref{lem:inf2} is also injective.
\end{proof}

\begin{lem} \label{lem:tnz_var}
	Let $K/F$ be a finite Galois extension containing $E$. Let $S'$ be a finite set of places satisfying Conditions \ref{cnds:tnz} with respect to $K$ and containing $S$.
	\begin{enumerate}
		\item The map $\hat H^{-1}(\Gamma_{E/F},A^\vee[S_E]_0) \rw \hat H^{-1}(\Gamma_{E/F},A^\vee[S'_E]_0)$ given by the inclusion $S \rw S'$ is injective and fits in the commutative diagram
		\[ \xymatrix{
		\hat H^{-1}(\Gamma_{E/F},A^\vee[S_E]_0)\ar[r]^{\Theta_{E,S}}\ar[d]&H^2(\Gamma_S,Z(O_S))\ar^{\tx{Inf}}[d]\\
		\hat H^{-1}(\Gamma_{E/F},A^\vee[S'_E]_0)\ar[r]^{\Theta_{E,S'}}&H^2(\Gamma_{S'},Z(O_{S'}))
		}\]
		\item The map $! : \hat H^{-1}(\Gamma_{E/F},A^\vee[S'_E]_0) \rw \hat H^{-1}(\Gamma_{K/F},A^\vee[S'_K]_0)$ of Lemma \ref{lem:tnshriek} is injective and fits in the commutative diagram
		\[ \xymatrix{
		\hat H^{-1}(\Gamma_{E/F},A^\vee[S'_E]_0)\ar[r]^{\Theta_{E,S'}}\ar[d]&H^2(\Gamma_{S'},Z(O_{S'}))\ar@{=}[d]\\
		\hat H^{-1}(\Gamma_{K/F},A^\vee[S'_K]_0)\ar[r]^{\Theta_{K,S'}}&H^2(\Gamma_{S'},Z(O_{S'}))
		}\]
	\end{enumerate}
\end{lem}
\begin{proof}
Notice first that since for every $\sigma \in \Gamma_{K/F}$ there exists $u \in V_K$ with $\sigma u = u$, the assumption of Lemma \ref{lem:tnshriek} follows from part 4 of Conditions \ref{cnds:tnz}.

According to Proposition \ref{pro:tnz_es}, the maps $\Theta_{E,S}$ and $\Theta_{E,S'}$ are injective. According to Lemma \ref{lem:inf3}, the inflation map in the first diagram is injective. This proves the injectivity claims.

To prove the commutativity of the first diagram, choose $X$ to be a free $\Z[\Gamma_{E/F}]$-module. Then so is $Y$ as well as $Y[S_E]_0$ and we get $\hat H^i(\Gamma_{E/F},Y[S_E]_0)=0$ for all $i \in \Z$. Looking at Diagram \eqref{eq:tnzdiag}, this implies that $\hat H^{-1}(\Gamma_{E/F},\bar Y[S_E]_0) \rw \hat H^{-1}(\Gamma_{E/F},A^\vee[S_E]_0)$ is bijective. We consider the following cube.

\[ \xymatrix@R=4pc@C=0pc{
	\hat H^{-1}(\Gamma_{E/F},\bar Y[S_E]_0)\ar[dd]\ar[rr]\ar[rd]^{\tx{TN}_{E,S}}&&\hat H^{-1}(\Gamma_{E/F},A^\vee[S_E]_0)\ar[dd]\ar[rd]^{\Theta_{E,S}}\\
	&H^1(\Gamma_S,\bar T(O_S))\ar[dd]^(.7){\tx{Inf}}\ar[rr]&&H^2(\Gamma_S,Z(O_S))\ar[dd]^(.7){\tx{Inf}}\\
	\hat H^{-1}(\Gamma_{E/F},\bar Y[S'_E]_0)\ar[rr]\ar[rd]^{\tx{TN}_{E,S'}}&&\hat H^{-1}(\Gamma_{E/F},A^\vee[S'_E]_0)\ar[rd]^{\Theta_{E,S'}}\\
	&H^1(\Gamma_{S'},\bar T(O_{S'}))\ar[rr]&&H^2(\Gamma_{S'},Z(O_{S'}))
}\]
We want to prove that the right face of this cube commutes. By the bijectivity of the back top map, it is enough to show that all the other faces commute. The back face commutes by functoriality of the map induced by the inclusion $S_E \rw S'_E$. The front face commutes by $\delta$-functoriality of $\tx{Inf}$. The top and bottom faces commute by Lemma \ref{lem:tnzdiag}. The left face commutes by Corollary \ref{cor:tnshriek}.
This proves the commutativity of the first of the two diagrams in the statement. The proof of the second is analogous.
\end{proof}

\begin{cor} \label{cor:tnz}
	The maps $\Theta_{E,S}$ splice to a functorial isomorphism
	\[ \Theta : \varinjlim_{E,S}\hat H^{-1}(\Gamma_{E/F},A^\vee[S_E]_0) \rw H^2(\Gamma,Z(\ol{F})). \]
\end{cor}
\begin{proof}
	According to Proposition \ref{pro:tnz_es} and Lemmas \ref{lem:tnz_var} and \ref{lem:inf3} we obtain a functorial injective homomorphism $\Theta$ as displayed. We will now argue that it is also surjective. Let $h \in H^2(\Gamma,Z(\ol{F}))$. Choose a finite Galois extension $E/F$ so that $h$ is inflated from $H^2(\Gamma_{E/F},Z(E))$. Choose $S$ large enough so that it satisfies Conditions \ref{cnds:tnz} with respect to $E$ and so that $Z(E)=Z(O_{E,S})$. Thus $h$ is in the image of the inflation $H^2(\Gamma_{E/F},Z(O_{E,S})) \rw H^2(\Gamma,Z(\ol{F}))$ and we can pick a preimage $h_{E,S}$.

	Choose $\bar Y$ to be a free $\Z[\Gamma_{E/F}]$-module and consider Diagram \eqref{eq:tnzdiag}. Let $h_{T,E,S} \in H^2(\Gamma_{E/F},T(O_{E,S}))$ be the image of $h_{E,S}$. The image of $h_{T,E,S}$ in the group $H^2(\Gamma_{E/F},\bar T(O_{E,S}))$ is zero, so the preimage of $h_{T,E,S}$ in $\hat H^0(\Gamma_{E/F},Y[S_E]_0)$ under ``-TN'' lifts to an element $g \in \hat H^{-1}(\Gamma_{E/F},A^\vee[S_E]_0)$. Let $h' \in H^2(\Gamma,Z(\ol{F}))$ be the inflation of $\Theta_{E,S}(g) \in H^2(\Gamma_S,Z(O_S))$. Then $h'$ and $h$ have the same image in $H^2(\Gamma,T(\ol{F}))$. But since we chose $\bar Y$ to be free, the map $H^2(\Gamma,Z) \rw H^2(\Gamma,T)$ is injective, so $h'=h$.
\end{proof}

\subsection{The finite multiplicative groups $P_{E,\dot S_E,N}$} \label{sub:pes}
Let $E/F$ be a finite Galois extension. Let $S \subset V_F$ be a finite full set of places and $\dot S_E \subset S_E$ a set of lifts for the places in $S$ (that is, over each $v \in S$ there is a unique $w \in \dot S_E$). We assume that the pair $(S,\dot S_E)$ satisfies the following.

\begin{cnds} \label{cnds:pes}
\begin{enumerate}
\item $S$ contains all archimedean places and all places that ramify in $E$.
\item Every ideal class of $E$ contains an ideal with support in $S_E$.
\item For every $w \in V_E$ there exists $w' \in S_E$ with $\tx{Stab}(w,\Gamma_{E/F})=\tx{Stab}(w',\Gamma_{E/F})$.
\item For every $\sigma \in \Gamma_{E/F}$ there exists $\dot v \in \dot S_E$ such that $\sigma\dot v = \dot v$.
\end{enumerate}
\end{cnds}
Pairs $(S,\dot S_E)$ that satisfy these conditions exist. Moreover, if $(S',\dot S'_E)$ is any pair containing $(S,\dot S_E)$ in the sense that $S \subset S'$ and $\dot S_E \subset \dot S_E'$, and if $(S,\dot S_E)$ satisfies these conditions, then so does $(S',\dot S_E')$. If $Z$ is a finite multiplicative group defined over $F$ and split over $E$ and if $\exp(Z) \in \N_S$, then we may apply all results of Subsection \ref{sub:h2z} to $Z$.

Fix $N \in \N_S$ and consider the finite abelian group
\[ \tx{Maps}(\Gamma_{E/F} \times S_E, \frac{1}{N}{\Z/\Z}) \]
as well as the following three conditions on an element $f$ of it:
\begin{enumerate}
\item For every $\sigma \in \Gamma_{E/F}$ we have $\sum_{w \in S_E} f(\sigma,w) = 0$.
\item For every $w \in S_E$ we have $\sum_{\sigma \in \Gamma_{E/F}} f(\sigma,w) = 0$.
\item Given $\sigma \in \Gamma_{E/F}$ and $w \in S_E$, if $f(\sigma,w) \neq 0$ then $\sigma^{-1}w \in \dot S_E$.
\end{enumerate}
Define $M_{E,S,N}$ to be the subgroup consisting of elements that satisfy the first two conditions, and $M_{E,\dot S_E,N}$ to be the subgroup of $M_{E,S,N}$ consisting of elements satisfying in addition the third condition. Note that both $M_{E,S,N}$ and $M_{E,\dot S_E,N}$ are $\Gamma_{E/F}$-stable.

\begin{lem} \label{lem:mesn_h2}
	Let $A$ be a finite $\Z[\Gamma_{E/F}]$-module.
	\begin{enumerate}
	\item If $\exp(A)|N$, the map
	\[ \Psi_{E,S,N} : \tx{Hom}(A,M_{E,S,N})^\Gamma \rw \hat Z^{-1}(\Gamma_{E/F},\tx{Maps}(S_E,A^\vee)_0), \quad H \mapsto h, \]
	defined by $h(w,a)=H(a,1,w)$ is an isomorphism of finite abelian groups, functorial in $A$. It restricts to an isomorphism
	\[ \tx{Hom}(A,M_{E,\dot S_E,N})^\Gamma \rw \tx{Maps}(\dot S_E,A^\vee)_0 \cap \hat Z^{-1}(\Gamma_{E/F},\tx{Maps}(S_E,A^\vee)_0). \]
	\item For $N|M$ the isomorphism $\Psi_{E,S,N}$ and $\Psi_{E,S,M}$ are compatible with the natural inclusion $M_{E,\dot S_E,N} \rw M_{E,\dot S_E,M}$. Setting $M_{E,S} = \varinjlim_N M_{E,S,N}$ we thus obtain an isomorphism
	\[ \Psi_{E,S} : \tx{Hom}(A,M_{E,S})^\Gamma \rw \hat Z^{-1}(\Gamma_{E/F},\tx{Maps}(S_E,A^\vee)_0). \]
	\item The map
	\[ \tx{Maps}(\dot S_E,A^\vee)_0 \cap \hat Z^{-1}(\Gamma_{E/F},\tx{Maps}(S_E,A^\vee)_0) \rw \hat H^{-1}(\Gamma_{E/F},\tx{Maps}(S_E,A^\vee)_0) \]
	is surjective.
	\end{enumerate}
\end{lem}
\begin{proof}
For the first, the inverse of the claimed bijection is given by $H(a,\sigma,w)=(\sigma h)(w,a)$. The second point follows from the trivial equality $\tx{Hom}(A,\frac{1}{N}\Z/\Z)=A^\vee=\tx{Hom}(A,\frac{1}{M}\Z/\Z)$.

For the third point we claim that every class in $\hat H^{-1}(\Gamma_{E/F},\tx{Maps}(S_E,A^\vee)_0)$ contains a representative supported on $\dot S_E$. Let $h \in \hat Z^{-1}(\Gamma_{E/F},\tx{Maps}(S_E,A^\vee)_0)$ and write $\tx{supp}(h) \subset S_E$ for its support. Suppose the set $\tx{supp}(h) \sm \dot S_E$ is non-empty and choose a place $w$ in it. Choose $\sigma \in \Gamma_{E/F}$ such that $\sigma w \in \dot S_E$ as well as $\dot v_0 \in \dot S_E$ with $\sigma\dot v_0=\dot v_0$, the latter being possible by Conditions \ref{cnds:pes}. Consider the element $h(w) \otimes (\delta_w - \delta_{\dot v_0}) \in A^\vee \otimes \tx{Maps}(S_E,\Z)_0 = \tx{Maps}(S_E,A^\vee)_0$, where $\delta_w$ is the map with value $1$ on $w \in S_E$ and value zero on $S_E \sm \{w\}$. Then $h' = h+\sigma(h(w)\otimes(\delta_w-\delta_{\dot v_0}))-h(w)\otimes(\delta_w - \delta_{\dot v_0})$ has the same image in $\hat H^{-1}(\Gamma_{E/F},\tx{Maps}(S_E,A^\vee)_0)$ as $h$, but $\tx{supp}(h') \sm \dot S_E = \tx{supp}(h) \sm (\dot S_E \cup \{w\})$. Applying this procedure finitely many steps we obtain a representative $h''$ of the cohomology class of $h$ with $\tx{supp}(h'') \subset \dot S_E$.
\end{proof}

Let $P_{E,\dot S_E,N}$ be the finite multiplicative group over $O_{F,S}$ with $X^*(P_{E,\dot S_E,N})=M_{E,\dot S_E,N}$.
Let $A$ be a finite $\Z[\Gamma_{E/F}]$-module with $\exp(A)|N$ and let $Z$ be the finite multiplicative group over $O_{F,S}$ with $X^*(Z)=A$. Composing the map
\[ \Psi_{E,S,N}: \tx{Hom}(A,M_{E,\dot S_E,N})^\Gamma \rw \hat H^{-1}(\Gamma_{E/F},\tx{Maps}(S_E,A^\vee)_0) \]
of the above lemma with the map $\Theta_{E,S}$ introduced in the previous Subsection we obtain a map
\begin{equation} \label{eq:mesn_h2}
	\Theta_{E,\dot S_E,N}^P : \tx{Hom}(P_{E,\dot S_E,N},Z)^\Gamma =  \tx{Hom}(A,M_{E,\dot S_E,N})^\Gamma \rw H^2(\Gamma_S,Z(O_S)).
\end{equation}
which is functorial in $Z$.

We may apply this map to the special case $A=M_{E,\dot S_E,N}$. In that case we have the canonical element $\tx{id}$ of the source of \eqref{eq:mesn_h2} and we let $\xi_{E,\dot S_E,N} \in H^2(\Gamma_S,P_{E,\dot S_E,N}(O_S))$ be its image. We will now study how the classes $\xi_{E,\dot S_E,N}$ vary with $N$, $S$, and $E$. First, let $N|M$. The obvious inclusion $M_{E,\dot S_E,N} \rw M_{E,\dot S_E,M}$ gives rise to a surjection $P_{E,\dot S_E,M} \rw P_{E,\dot S_E,N}$. Let
\[ P_{E,\dot S_E} = \varprojlim_N P_{E,\dot S_E,N}. \]

\begin{lem} \label{lem:mesn_n} We have the equality
\[ H^2(\Gamma_S,P_{E,\dot S_E}(O_S)) = \varprojlim_N H^2(\Gamma_S,P_{E,\dot S_E,N}(O_S)). \]
The elements $\xi_{S,\dot S_E,N}$ form a compatible system and thus lead to an element $\xi_{S,\dot S_E} \in H^2(\Gamma_S,P_{E,\dot S_E}(O_S))$.
\end{lem}
\begin{proof}
The claimed equality follows from \cite[Cor. 2.7.6]{NSW08} and \cite[Thm. 8.3.20]{NSW08}. To prove that the classes $\xi_{E,\dot S_E,N}$ form an inverse system, consider the diagram
\[ \xymatrix@C=3.5pc{
	\tx{Hom}(M_{E,\dot S_E,N},M_{E,\dot S_E,N})^\Gamma\ar[r]^-{\Psi_{E,S,N}}\ar[d]&\hat H^{-1}(\Gamma_{E,F},\tx{Maps}(S_E,M_{E,\dot S_E,N}^\vee)_0)\ar@{=}[d]\\
	\tx{Hom}(M_{E,\dot S_E,N},M_{E,\dot S_E,M})^\Gamma\ar[r]^-{\Psi_{E,S,M}}&\hat H^{-1}(\Gamma_{E,F},\tx{Maps}(S_E,M_{E,\dot S_E,N}^\vee)_0)\\
	\tx{Hom}(M_{E,\dot S_E,M},M_{E,\dot S_E,M})^\Gamma\ar[r]^-{\Psi_{E,S,M}}\ar[u]&\hat H^{-1}(\Gamma_{E,F},\tx{Maps}(S_E,M_{E,\dot S_E,M}^\vee)_0)\ar[u]\\
}\]
All vertical maps are induced by the inclusion $M_{E,\dot S_E,N} \rw M_{E,\dot S_E,M}$. The diagram commutes by Lemma \ref{lem:mesn_h2}, part 2 being responsible for the top square and the functoriality statement in part 1 for the bottom square. The elements $\tx{id} \in \tx{Hom}(M_{E,\dot S_E,N},M_{E,\dot S_E,N})^\Gamma$ and $\tx{id} \in \tx{Hom}(M_{E,\dot S_E,M},M_{E,\dot S_E,M})^\Gamma$ both map to the natural inclusion in the middle left term, hence the corresponding classes in $\hat H^{-1}$ also meet in the middle right term. The functoriality of the map $\Theta_{E,S}$ now implies the claim that the classes $\xi_{E,\dot S_E,M}$ form an inverse system.
\end{proof}

Next we consider a finite Galois extension $K/F$ containing $E$ and a pair $(S',\dot S'_K)$ satisfying Conditions \ref{cnds:pes} with respect to $K$ and such that $S \subset S'$ and $\dot S_E \subset (\dot S'_K)_E$. We shall abbreviate this by saying
\begin{equation} \label{eq:mesn_order}
	(E,\dot S_E,N) < (K,\dot S'_K,M).
\end{equation}
Note that, given $E$, $(S,\dot S_E)$ and $K$, it is always possible to find such a pair $(S',\dot S'_K)$.
Let
\begin{equation} \label{eq:mesn_es} M_{E,\dot S_E,N} \rw M_{K,\dot S'_K,N},\qquad f \mapsto f^K \end{equation}
be defined by
\[ f^K(\sigma,u) = \begin{cases} f(\sigma, p(u)),&\sigma^{-1}u \in \dot S'_K \cap S_K \\ 0,&else \end{cases}. \]
This map is immediately verified to be $\Gamma_{K/F}$-equivariant, where $\Gamma_{K/F}$ acts on $M_{E,\dot S_E,N}$ via its quotient $\Gamma_{E/F}$.

\begin{lem} \label{lem:mesn_es1}
For any finite $\Gamma_{E/F}$-module $A$ with $\exp(A)|N$, the following diagram commutes
\[ \xymatrix@C=4pc{
	\tx{Hom}(A,M_{E,\dot S_E,N})^\Gamma\ar[r]^{\Theta_{E,\dot S_E,N}^P}\ar[d]^{\eqref{eq:mesn_es}}&H^2(\Gamma_S,Z(O_S))\ar[d]^{\tx{Inf}}\\
	\tx{Hom}(A,M_{K,\dot S'_K,N})^\Gamma\ar[r]^{\Theta_{K,\dot S'_K,N}^P}&H^2(\Gamma_{S'},Z(O_{S'}))
}\]
\end{lem}
\begin{proof}
This follows immediately from Lemma \ref{lem:tnz_var}: Composing the left map in this diagram with $\Psi_{K,S',N}$ gives the same result as composing $\Psi_{E,S,N}$ with the map $\hat H^{-1}(\Gamma_{E/F},\tx{Maps}(S_E,A^\vee)_0) \rw \hat H^{-1}(\Gamma_{K/F},\tx{Maps}(S'_K,A^\vee)_0)$ that is the composition of the two left maps in the two diagrams of Lemma \ref{lem:tnz_var}.
\end{proof}

\begin{lem} \label{lem:mesn_es2}
The image of $\xi_{K,\dot S'_K,N}$ under the map $H^2(\Gamma_{S'},P_{K,\dot S'_K,N}(O_{S'})) \rw H^2(\Gamma_{S'},P_{E,\dot S_E,N}(O_{S'}))$ induced by \eqref{eq:mesn_es} is equal to the image of $\xi_{E,\dot S_E,N}$ under the inflation map $H^2(\Gamma_S,P_{E,\dot S_E,N}(O_S)) \rw H^2(\Gamma_{S'},P_{E,\dot S_E,N}(O_{S'}))$.
\end{lem}
\begin{proof}
The proof is similar to that of Lemma \ref{lem:mesn_n}. We consider the diagram
\[ \xymatrix@C=4pc{
	\tx{Hom}(M_{E,\dot S_E,N},M_{E,\dot S_E,N})^\Gamma\ar[r]^-{\Theta^P_{E,\dot S_E,N}}\ar[d]&H^2(\Gamma_S,P_{E,\dot S_E,N}(O_S))\ar[d]^{\tx{Inf}}\\
	\tx{Hom}(M_{E,\dot S_E,N},M_{K,\dot S'_K,N})^\Gamma\ar[r]^-{\Theta^P_{K,\dot S'_K,N}}&H^2(\Gamma_{S'},P_{E,\dot S_E,N}(O_{S'}))\\
	\tx{Hom}(M_{K,\dot S'_K,N},M_{K,\dot S'_K,N})^\Gamma\ar[r]^-{\Theta^P_{K,\dot S'_K,N}}\ar[u]&H^2(\Gamma_{S'},P_{K,\dot S'_K,N}(O_{S'}))\ar[u]\\
}\]
All vertical maps except for the inflation map are induced by \eqref{eq:mesn_es}. The top square commutes due to Lemma \ref{lem:mesn_es1}, while the bottom square commutes by functoriality of $\Theta^P_{K,\dot S'_K,N}$. Since the two elements $\tx{id}$ of the top and bottom left term meet in the middle left term, the corresponding elements $\xi_{E,\dot S_E,N}$ and $\xi_{K,\dot S'_K,N}$ of the top and bottom right term meet in the middle right term.
\end{proof}

It is easy to see that the maps \eqref{eq:mesn_es} are compatible with respect to $N$ and thus splice to a map
\[ P_{K,\dot S'_K} \rw P_{E,\dot S_E} \]
which maps the class $\xi_{K,\dot S'_K}$ to the class $\xi_{E,\dot S_E}$. Let $E_i$ be an exhaustive tower of finite Galois extensions of $F$, let $S_i$ be an exhaustive tower of finite subsets of $V_F$, and let $\dot S_i \subset S_{i,E_i}$ be a set of lifts for $S_i$. We assume that $\dot S_i \subset \dot S_{i+1,E_i}$ and that $(S_i,\dot S_i)$ satisfies  Conditions \ref{cnds:pes} with respect to $E_i/F$. Then
\begin{equation} \label{eq:slim} \dot V = \varprojlim_i \dot S_i \end{equation}
is a subset of $V_{\ol{F}}$ of lifts of $V_F$
and
\begin{equation} \label{eq:plim} P_{\dot V} = \varprojlim_i P_{E_i,\dot S_i} \end{equation}
is a pro-finite algebraic group defined over $F$. For each finite multiplicative group $Z$ defined over $F$ (now without any condition on its exponent) we obtain from $\Theta^P_{E_i,\dot S_i,N}$ a homomorphism
\begin{equation} \label{eq:m_h2} \Theta^P_{\dot V} : \tx{Hom}(P_{\dot V},Z)^\Gamma \rw H^2(\Gamma,Z(\ol{F})), \end{equation}
which is surjective according to Corollary \ref{cor:tnz}.

We can reinterpret $\Theta^P_{\dot V}$ in the following way that will be useful later. Let $\xi_i = \xi_{E_i,\dot S_i}$ denote both the distinguished element of $H^2(\Gamma_S,P_{E_i,\dot S_i}(O_S))$ and its image in $H^2(\Gamma,P_{E_i,\dot S_i}(\ol{F}))$. According to Lemma \ref{lem:mesn_es2} the sequence $(\xi_i)$ provides a distinguished element of $\varprojlim H^2(\Gamma,P_{E_i,\dot S_i}(\ol{F}))$. According to \cite[Theorem 2.7.5]{NSW08} the natural map $H^2(\Gamma,P_{\dot V}(\ol{F})) \rw \varprojlim H^2(\Gamma,P_{E_i,\dot S_i}(\ol{F}))$ is surjective. Then $\Theta^P_{\dot V}(\varphi)=\varphi(\tilde\xi)$, where $\tilde\xi \in H^2(\Gamma,P_{\dot V}(\ol{F}))$ is any preimage of $(\xi_i)$. Since any $\varphi$ factors through the projection $P_{\dot V} \rw P_{E_i,\dot S_i}$ for some $i$, the choice of $\tilde\xi$ is irrelevant.

\begin{lem} \label{lem:hompdotvz} Let $Z$ be a finite multiplicative group defined over $F$ and let $A=X^*(Z)$. Let $A^\vee[\dot V]_0$ denote the group of finitely-supported maps $f : \dot V \rw A^\vee$ satisfying $\sum_{\dot v \in \dot V}f(\dot v)=0$. Let $A^\vee[\dot V]_{0,\infty}$ be the subgroup of those $f$ that in addition satisfy $f(\dot v)=0$ if $v \in V_F$ is complex, and $f(\dot v)+\sigma f(\dot v)=0$ if $v \in V_F$ is real and $\sigma \in \Gamma_{\dot v}$ is the non-trivial element.

There is a natural isomorphism
\[ \tx{Hom}_F(P_{\dot V},Z) \rw A^\vee[\dot V]_{0,\infty}.\]
\end{lem}
\begin{proof}
By the finiteness of $A$ we have $\tx{Hom}_F(P_{\dot V},Z)=\tx{Hom}(A,\varinjlim M_{E_i,\dot S_i})^\Gamma=\varinjlim \tx{Hom}(A,M_{E_i,\dot S_i})^\Gamma$. By Lemma \ref{lem:mesn_h2} the latter is equal to $\varinjlim A^\vee[\dot S_i]_0^{N_{E_i/F}}$, where the transition maps are given by the inclusion $\dot S_i \rw \dot S_{i+1}$ which sends $w \in \dot S_i$ to the unique place of $\dot S_{i+1}$ lying above $w$. Here $A^\vee[\dot S_i]_0^{N_{E_i/F}}$ signifies the subgroup of those elements of $A^\vee[S_{i,E_i}]_0$ which are supported on $\dot S_i$ and killed by the norm $N_{E_i/F}$. In view of the support condition, the second condition is equivalent to $N_{E_{\dot v}/F_v} f(\dot v)=0$ for all $\dot v \in \dot S_i$.
\end{proof}

We will now show that \eqref{eq:m_h2} behaves well with respect to localization. To describe this, we need to recall the local counterpart of $P_{\dot V}$ from \cite[\S3.1]{KalRI}. Let $v \in \dot V$. Associated to the local field $F_v$ there is the pro-finite multiplicative group $u_v$ introduced in \cite[\S3.1]{KalRI}. It is defined as $\varprojlim_{E_v/F_v,N} u_{E_v/F_v,N}$, the limit being taken over all finite Galois extensions $E_v/F_v$ and all natural numbers $N$. Here $X^*(u_{E_v/F_v,N}) = I_{E_v/F_v,N}$ is the set of those maps $f : \Gamma_{E_v/F_v} \rw \frac{1}{N}\Z/\Z$ that satisfy $\sum_{\tau \in \Gamma_{E_v/F_v}}f(\tau)=0$ and for a tower of finite Galois extensions $K_v/E_v/F_v$ and $N|M$, the transition map $I_{E_v/F_v,N} \rw I_{K_v/F_v,M}$ is given simply by composition with the natural projection $\Gamma_{K_v/F_v} \rw \Gamma_{E_v/F_v}$ and the natural inclusion $\frac{1}{N}\Z/\Z \to \frac{1}{M}\Z/\Z$.

If $Z$ is a finite multiplicative group defined over $F_v$ with $\exp(Z)|N$, and we set $A=X^*(Z)$, then we have the isomorphism
\[ \Psi_{E_v,N} : \tx{Hom}(A,I_{E_v/F_v,N})^\Gamma \rw \hat Z^{-1}(\Gamma_{E_v/F_v},A^\vee), \qquad H \mapsto h \]
defined by $h(a)=H(a,1)$. The composition of this isomorphism with the map $\Theta_{E,v}$ discussed in Subsection \ref{sub:h2z} provides a homomorphism
\[ \Theta_{E_v,N}^u : \tx{Hom}(u_{E_v/F_v,N},Z)^{\Gamma_v} \rw H^2(\Gamma_v,Z). \]
Applying this homomorphism to the case $Z=u_{E_v/F_v,N}$ we obtain a distinguished element $\xi_{E_v/F_v,N} \in H^2(\Gamma_v,u_v)$ as the image of the identity map. We can then reinterpret $\Theta_{E_v,N}^u$ as the homomorphism that maps $\varphi$ to $\varphi(\xi_{E_v/F_v,N})$. The transition map $u_{K_v/F_v,M} \rw u_{E_v/F_v,N}$ sends $\xi_{K_v/F_v,M}$ to $\xi_{E_v/F_v,N}$ and the system $\xi_{E_v/F_v,N}$ thus leads to a distinguished element $\xi_v \in H^2(\Gamma_v,u_v)$. We refer to \cite[Lemma 4.5, Fact 4.6]{KalRI} for more details. We obtain a map
\[ \Theta^u_v : \tx{Hom}(u_v,Z)^{\Gamma_v} \rw H^2(\Gamma_v,Z) \]
that sends $\varphi$ to $\varphi(\xi_v)$. It is surjective \cite[Proposition 3.2]{KalRI}.

We now define a localization map
\begin{equation} \label{eq:locvp} loc_v^P : u_v \rw P_{\dot V} \end{equation}
for $v \in \dot V$ as follows. Fix a finite Galois extension $E/F$, a pair $(S,\dot S_E)$ satisfying conditions \ref{cnds:pes} with $v \in S$, and a natural number $N \in \N_S$, and consider the map
\[ loc_v^{M_{E,\dot S_E,N}} : M_{E,\dot S_E,N} \rw I_{E_v/F_v,N},\qquad H \mapsto H_v, \]
given by $H_v(\tau)=H(\tau,v)$ for $\tau \in \Gamma_{E_v/F_v}$. We have
\[ 0 = \sum_{\sigma \in \Gamma_{E/F}}H(\sigma,v) = \sum_{\tau \in \Gamma_{E_v/F_v}}H(\tau,v) = \sum_{\tau \in \Gamma_{E_v/F_v}}H_v(\tau), \]
and this shows that $loc_v^{M_{E,\dot S_E,N}}$ is well-defined. It is evidently $\Gamma_{E_v/F_v}$\-equivariant. We will write $loc_v^{P_{E,\dot S_E,N}} : u_{E_v/F_v,N} \rw P_{E,\dot S_E,N}$ for the dual of this map. For varying $N$, these maps splice to a map $loc_v^{P_{E,\dot S_E}} : u_v \rw P_{E,\dot S_E}$. For varying $i$, the maps $loc_v^{P_{E_i,\dot S_i}}$ in turn splice together to form the map \eqref{eq:locvp}.

\begin{lem} \label{lem:locvp}
Assume that $E/F$ splits $Z$, $(S,\dot S_E)$ satisfies Conditions \ref{cnds:pes}, and $N \in \N_S$ is a multiple of $\exp(Z)$. We have the commutative diagram
\[ \xymatrix{
	\tx{Hom}(P_{E,\dot S_E,N},Z)^\Gamma\ar[r]^-{\Theta^P_{E,\dot S_E,N}}\ar[d]^{\tx{loc}_v^{P_{E,\dot S_E,N}}}&H^2(\Gamma,Z)\ar[d]\\
	\tx{Hom}(u_{E_v/F_v,N},Z)^{\Gamma_v}\ar[r]^-{\Theta^u_{E_v,N}}&H^2(\Gamma_v,Z)
}\]
where the right vertical map is the localization map given by restriction to $\Gamma_v$ followed by the inclusion $Z(\ol{F}) \rw Z(\ol{F_v})$.
\end{lem}
\begin{proof}
Set $A=X^*(Z)$. According to Lemma \ref{lem:zloc}, it is enough to show that the following diagram commutes
\[ \xymatrix{
\tx{Hom}(A, M_{E,\dot S_E,N})^\Gamma\ar[d]^{loc_v^{M_{E,\dot S_E,N}}}\ar[r]^{\Psi_{E,S,N}}&\hat H^{-1}(\Gamma_{E/F},A^\vee[S_E]_0)\ar[d]\\
\tx{Hom}(A,I_{E_v/F_v,N})^{\Gamma_v}\ar[r]^{\Psi_{E_v,N}}&\hat H^{-1}(\Gamma_{E_v/F_v},A^\vee)
}\]
where the right vertical map is restriction to $\Gamma_{E_v/F_v}$ followed by projection onto the $v$-th component. Explicitly, this map sends $h \in \tx{Maps}(S_E,A^\vee)_0$ to the element of $A^\vee$ given by
\[ a \mapsto \sum_{\tau \in \Gamma_{E_v/F_v} \lmod \Gamma_{E/F}} \dot\tau h(\dot\tau^{-1}v,\dot\tau^{-1}a) \]
where $\dot\tau \in \Gamma_{E/F}$ is any representative of the coset $\tau$. The resulting class in $\hat H^{-1}(\Gamma_{E_v/F_v},A^\vee)$ is independent of the choices of representatives.

The composition of this map with $\Psi_{E,S,N}$ thus sends $H \in \tx{Hom}(A,M_{E,\dot S_E,N})^\Gamma$ to
\[ a \mapsto \sum_{\tau \in \Gamma_{E_v/F_v} \lmod \Gamma_{E/F}} \dot\tau H(\dot\tau^{-1}a,1,\dot\tau^{-1}v). \]
Each summand is equal to $H(a,\dot\tau,v)$ and since $v \in \dot S_E$ the definition of $M_{E,\dot S_E,N}$ implies that $H(a,\dot\tau,v)=0$ unless $\dot\tau^{-1} \in \Gamma_{E_v/F_v}$. Thus all summands are zero except for the summand indexed by the trivial coset $\tau$, for which we may take the representative $\dot\tau=1 \in \Gamma_{E/F}$. We conclude that sending $H \in \tx{Hom}(A,M_{E,\dot S_E,N})^\Gamma$ first horizontally and then vertically provides the element of $\hat H^{-1}(\Gamma_{E_v/F_v}A^\vee)$ represented by $a \mapsto H(a,1,v)$. Recalling the definitions of $\tx{loc}_v^{M_{E,\dot S_E,N}}$ and $\Psi_{E_v,N}$ this element also represents the image of $H$ under the composition of these two maps.
\end{proof}

\begin{cor} \label{cor:xiloc}
For $v \in \dot V$ consider the maps
\[ H^2(\Gamma,P_{\dot V}(\ol{F})) \rw H^2(\Gamma_v,P_{\dot V}(\ol{F_v})) \lw H^2(\Gamma_v,u_v), \]
the left one being given by restriction to $\Gamma_v$ followed by inclusion $\ol{F}^\times \rw \ol{F_v}^\times$, and the right one being given by $\tx{loc}_v^P$. If $\tilde \xi \in H^2(\Gamma,P_{\dot V}(\ol{F}))$ is any preimage of $(\xi_i)$, then the images of $\tilde\xi$ and $\xi_v$ in the middle term are equal.
\end{cor}
\begin{proof}
We have $H^2(\Gamma_v,P_{\dot V}(\ol{F_v}))=\varprojlim H^2(\Gamma_v,P_{E_i,\dot S_i,N_i}(\ol{F_v}))$ and $H^2(\Gamma_v,u_v)=\varprojlim H^2(\Gamma_v,u_{E_{i,v}/F_v,N_i})$ according to \cite[Corollary 2.7.6]{NSW08}, where $N_i \in \N_{S_i}$ is a co-final sequence in $\N$. It will be enough to show that the image of $\xi_{E_{i,v}/F_v,N_i}$ in $H^2(\Gamma_v,P_{E_i,S_i,N_i}(\ol{F_v}))$ under $loc_v^{P_{E_i,\dot S_i,N_i}}$ coincides with the image of $\xi_{E_i,S_i,N_i}$ under the usual localization map $H^2(\Gamma,P_{E_i,S_i,N_i}(\ol{F})) \rw H^2(\Gamma_v,P_{E_i,S_i,N_i}(\ol{F_v}))$.

We have the commutative diagram
\[ \xymatrix@C=1.1pc{
	\tx{End}(P_{E_i,\dot S_i,N_i})^\Gamma\ar[d]^{\Theta^P_{E_i,\dot S_i,N_i}}\ar[r]&\tx{Hom}(u_{E_{i,v}/F_v,N_i},P_{E_i,\dot S_i,N_i})^{\Gamma_v}\ar[d]^{\Theta^u_{E_{i,v},N_i}}&\tx{End}(u_{E_{i,v}/F_v,N_i})^{\Gamma_v}\ar[d]^{\Theta^u_{E_{i,v},N_i}}\ar[l]\\
	H^2(\Gamma,P_{E_i,\dot S_i,N_i}(\ol{F}))\ar[r]&H^2(\Gamma_v,P_{E_i,\dot S_i,N_i}(\ol{F_v}))&H^2(\Gamma_v,u_{E_{i,v}/F_v,N_i})\ar[l]
}\]
where the left square commutes according to Lemma \ref{lem:locvp} applied to $Z=P_{E,\dot S_E,N}$, while the right square commutes by functoriality of $\Theta^u_{E_v,N}$ in $Z$.

Now $\xi_{E_i,\dot S_i,N_i} \in H^2(\Gamma,P_{E_i,\dot S_i,N_i})$ is the image of $\tx{id} \in \tx{End}(P_{E_i,\dot S_i,N_i})^\Gamma$, while $\xi_{E_{i,v}/F_v,N_i} \in H^2(\Gamma_v,u_{E_{i,v}/F_v,N_i})$ is the image of $\tx{id} \in \tx{End}(u_{E_{i,v}/F_v,N_i})^{\Gamma_v}$. The two elements $\tx{id}$ both map to $loc_v^{P_{E_i,\dot S_i,N_i}} \in \tx{Hom}(u_{E_{i,v}/F_v,N_i},P_{E_i,\dot S_i,N_i})^{\Gamma_v}$.
\end{proof}

\subsection{The vanishing of $H^1(\Gamma,P_{\dot V})$ and $H^1(\Gamma_v,P_{\dot V})$} \label{sub:h1van}

The purpose of this subsection is to prove the vanishing of the cohomology group $H^1(\Gamma,P_{\dot V}(\ol{F}))$ and, for each $v \in \dot V$, of the local cohomology group $H^1(\Gamma_v,P_{\dot V}(\ol{F_v}))$. Via local and global Poitou-Tate duality this will be reduced to the study of the cohomology of the finite modules $M_{E,\dot S_E,N}$. We begin by giving a description of these modules that is slightly different from their definition.

\begin{lem} \label{lem:mesn_phi}
The $\Gamma_{E/F}$-module $M_{E,\dot S_E,N}$ can be described as the set of functions
\[ \phi : S \times \Gamma_{E/F} \rw \frac{1}{N}\Z/\Z \]
satisfying the two conditions
\begin{enumerate}
\item $\sum_{\sigma \in \Gamma_{E/F,\dot v}}\phi(v,\theta\sigma)=0$ for all $\theta \in \Gamma_{E/F}$ and $v \in S$;
\item $\sum_{v \in S}\phi(v,\theta)=0$ for all $\theta \in \Gamma_{E/F}$.
\end{enumerate}
and with $\Gamma_{E/F}$-action given by $[\tau\phi](v,\theta)=\phi(v,\tau^{-1}\theta)$.
\end{lem}
\begin{proof}
This follows from the change of variables $f \mapsto \phi$ given by $\phi(v,\sigma)=f(\sigma,\sigma \dot v)$.
\end{proof}

The fist condition implies that for any complex archimedean place $v \in S$ and any $\theta \in \Gamma_{E/F}$ the value $\phi(v,\theta)$ is zero. We may thus replace the set $S$ by the complement $S^\C$ of the complex archimedean places. We will write $S^\C_E$ for the set of all places of $E$ lying above $S^\C$.

\begin{lem} \label{lem:mesn_phiseq} The sequence
\[ 0 \rw M_{E,\dot S_E,N} \rw \tx{Maps}(S^\C,\tx{Maps}(\Gamma_{E/F},\frac{1}{N}\Z/\Z))_0 \rw \tx{Maps}(S^\C_E,\frac{1}{N}\Z/\Z)_0 \rw 0 \]
is exact. Here, the second term consists of the set of maps $S^\C \times \Gamma_{E/F} \rw \frac{1}{N}\Z/\Z$ that satisfy the second of the two conditions in Lemma \ref{lem:mesn_phi}, and the map from this term to the next is given by $\phi \mapsto \xi$ with $\xi(\theta\dot v)=\sum_{\sigma \in \Gamma_{E/F,\dot v}}\phi(v,\theta\sigma)$ for $\theta \in \Gamma_{E/F}$ and $v \in S$.
\end{lem}
\begin{proof}
First we note that $\xi$ is well-defined and we have
\[ \sum_{w \in S^\C_E}\xi(w)=\sum_{v \in S^\C}\sum_{\theta \in \Gamma_{E/F}/\Gamma_{E/F,\dot v}}\xi(\theta\dot v)=\sum_{v \in S^\C}\sum_{\theta\in\Gamma_{E/F}}\phi(v,\theta)=0. \]
We need to show that for any $\xi$ there exists a $\phi$ mapping to it. We begin with the special case $\xi = a(\delta_{w_1} - \delta_{w_2})$, where $a \in \frac{1}{N}\Z/\Z$ and $w_1 \neq w_2 \in S_E$ lie over the same element $v_1\in S^\C$. Choose $\theta_1,\theta_2 \in \Gamma_{E/F}$ such that $\theta_1\dot v_1 = w_1$ and $\theta_2\dot v_1 = w_2$ for the representative $\dot v_1 \in \dot S_E$ of $v_1$. By part 4 of Conditions \ref{cnds:pes} there exists $\dot v_0 \in \dot S_E$ with $\theta_2^{-1}\theta_1\dot v_0=\dot v_0$. Note that automatically $v_0 \neq v_1$. We now define $\phi \in \tx{Maps}(S^\C,\tx{Maps}(\Gamma_{E/F},\frac{1}{N}\Z/\Z))_0$ by making all of its values zero except
\[ \phi(v_1,\theta_1)=a,\quad \phi(v_0,\theta_1)=-a,\quad\phi(v_1,\theta_2)=-a,\quad\phi(v_0,\theta_2)=a.\]
Then $\phi$ is a preimage of $\xi$.

Now consider a general $\xi \in \tx{Maps}(S^\C_E,\frac{1}{N}\Z/\Z)_0$. Using the special elements just discussed, we can modify $\xi$ to achieve that it is supported on $\dot S^\C_E$, in which case we can define a lift $\phi$ of it simply by $\phi(v,1)=\xi(\dot v)$ and $\phi(v,\theta)=0$ for $\theta \neq 1$.
\end{proof}

In terms of the variables $\phi$, the transition map \eqref{eq:mesn_es} takes the form
\begin{equation} \label{eq:mesn_esphi} \phi \mapsto \phi^K,\quad \phi^K(v,\theta)=\begin{cases} \phi(v,\theta),&v \in S\\ 0,&\tx{else}\end{cases}\end{equation}
for any $v \in S'$ and $\theta \in \Gamma_{K/F}$. This transition map fits into the commutative diagram
\begin{equation} \label{eq:mesn_trdiag} \xymatrix{
0\ar[r]&M_{E,\dot S_E,N}\ar[r]\ar[d]&\tx{Maps}(S^\C,M_{E,N})_0\ar[r]\ar[d]&\tx{Maps}(S^\C_E,\frac{1}{N}\Z/\Z)_0\ar[r]\ar[d]&0\\
0\ar[r]&M_{K,\dot S'_K,M}\ar[r]&\tx{Maps}(S'^\C,M_{K,M}))_0\ar[r]&\tx{Maps}(S'^\C_K,\frac{1}{M}\Z/\Z)_0\ar[r]&0
} \end{equation}
where the horizontal exact sequences come from Lemma \ref{lem:mesn_phiseq} and we have abbreviated
\[ M_{E,N} = \tx{Maps}(\Gamma_{E/F},\frac{1}{N}\Z/\Z). \]
The left and middle vertical arrows are both given by $\phi \mapsto \phi^K$. The right vertical arrow is given by $\xi \mapsto \xi^K$, where
\begin{equation} \label{eq:xies} \xi^K(u) = |\Gamma_{K/E,u}|\xi(p(u)) \end{equation}
for any $u \in S^\C_K$ with image $p(u) \in S^\C_E$, and $\xi^K(u)=0$ for $u \in S'^\C_K \sm S^\C_K$.

The following two lemmas will allow us to control the colimits of the terms in the exact sequence of Lemma \ref{lem:mesn_phiseq}.

\begin{lem} \label{lem:mesn_t1} Given $(E,\dot S_E,N)$ there exists $(K,\dot S'_K,M)>(E,\dot S_E,N)$ as in \eqref{eq:mesn_order} such that for all subgroups $\Delta \subset \Gamma_{K/F}$ the transition map
\[ \frac{\tx{Maps}(S_E^\C,\frac{1}{N}\Z/\Z)_0^{\Delta}}{\tx{Maps}(S^\C,M_{E,N})_0^{\Delta}}  \to \frac{\tx{Maps}(S'^\C_K,\frac{1}{M}\Z/\Z)_0^{\Delta}}{\tx{Maps}(S'^\C,M_{K,M})_0^{\Delta}} \]
is zero, where we have used the quotient notation as a short-hand for the corresponding cokernel.
\end{lem}
\begin{proof} Let $K_1/E$ be such that all real places of $F$ are complex in $K_1$ and that $|\Gamma_{K_1/E,u}|$ is a multiple of $N$ for all $p$-adic places $u \in S_{K_1}$. Let $N_1=2N$. For any $\xi \in \tx{Maps}(S_E^\C,\frac{1}{N}\Z/\Z)$ its image $\xi^{K_1} \in \tx{Maps}(S_{K_1}^\C,\frac{1}{N_1}\Z/\Z)$ is supported on $S_{\R,K_1}$ by \eqref{eq:xies} and its values are divisible by $2$.

Given $\xi \in \tx{Maps}(S_E^\C,\frac{1}{N}\Z/\Z)_0^\Delta$ we choose a $\Delta$-invariant function $\tilde\xi : S_{\R,K_1} \rw \frac{1}{N_1}\Z/\Z$ with $2\tilde\xi(w)=\xi^{K_1}(w)$. We may not be able to arrange that the sum of the values of $\tilde\xi$ is zero, but this will not matter. Fix an auxiliary $p$-adic place $v_0 \in S$. Define a function $\phi : S \times \Gamma_{K_1/F} \rw \frac{1}{N_1}\Z/\Z$ by $\phi(u,\theta)=\tilde\xi(\theta\dot u)$ for all $\theta \in \Gamma_{K_1/F}$ and all $u \in S_\R$. The function is to be zero at all $w \in S \sm S_\R$ except for $w=v_0$, in which case we set $\phi(v_0,\theta)=-\sum_{u \in S_\R} \tilde\xi(\theta\dot u)$. Then $\phi \in \tx{Maps}(S^\C,M_{K_1,N_1})_0^{\Delta}$. If $\xi_\phi$ is the image of $\phi$ in $\tx{Maps}(S_{K_1}^\C,\frac{1}{N_1}\Z/\Z)_0^{\Delta}$ then we have $\xi_\phi(u)=\xi^{K_1}(u)$ for all $u \in S_{\R,K_1}$. The function $\xi_\phi$ is supported on $S_{\R,K_1} \cup \{w|v_0\}$ and we may have $\xi_\phi(w)\neq \xi^{K_1}(w)$ for $w|v_0$. Using that $v_0$ is a $p$-adic place we now choose a Galois extension $K/F$ containing $K_1$ so that $|\Gamma_{K/K_1,v_0}|$ is a multiple of $N_1$. Then $\xi_\phi^K$ is supported on $S_{\R,K}$ by formula \eqref{eq:xies}.  We still have $\xi_\phi^K(u)=\xi^K(u)$ for all $u \in S_{\R,K}$ and thus $\xi_\phi^K=\xi^K$. But $\xi_\phi^K$ lifts to $\phi^K$.
\end{proof}

\begin{lem} \label{lem:mesn_t2} The following colimits over $(E,\dot S_E,N)$ vanish.
\begin{enumerate}
	\item $\varinjlim H^1(\Gamma,\tx{Maps}(S^\C,M_{E,N})_0)=0$;
	\item $\varinjlim H^2(\Gamma,\tx{Maps}(S^\C,M_{E,N})_0)=0$;
	\item $\varinjlim H^1(\Gamma_v,\tx{Maps}(S^\C,M_{E,N})_0)=0$ for all $v \in V(F)$.
\end{enumerate}
\begin{proof}
The inclusion of $\tx{Maps}(S^\C,M_{E,N})_0$ into $\tx{Maps}(S^\C,M_{E,N})$ has a $\Gamma_{E/F}$-equivariant splitting, given by choosing an arbitrary place of $S^\C$. Thus this inclusion induces an inclusion on the level of cohomology. We may thus replace $\tx{Maps}(S^\C,M_{E,N})_0$ by $\tx{Maps}(S^\C,M_{E,N})$ in the vanishing statements above. The latter $\Gamma$-module is the $S^\C$-power of $M_{E,N}$. In the $S$-variable the transition map is just the inclusion of $S$ into the larger set $S'$, so what we want to show is the vanishing of the cohomology of $M_{E,N}$, where the colimit is taken over $E$ and $N$.

We begin with $H^1(\Gamma,M_{E,N})$. Since $M_{E,N}=\tx{Ind}_{\Gamma_E}^\Gamma \frac{1}{N}\Z/\Z$ the transition map $H^1(\Gamma,M_{E,N}) \to H^1(\Gamma,M_{K,M})$ is translated via the Shapiro isomorphism to the restriction map $\tx{Hom}(\Gamma_E,\frac{1}{N}\Z/\Z) \to \tx{Hom}(\Gamma_K,\frac{1}{M}\Z/\Z)$. Any homomorphism $\Gamma_E \to \frac{1}{N}\Z/\Z$ has kernel of finite index in $\Gamma_E$ and there is thus a Galois extension $K/E$ s.t. $\Gamma_K$ is contained in that kernel.

Turning to $H^1(\Gamma_v,M_{E,N})$, according to the Mackey formula and Shapiro lemma we have
\[ H^1(\Gamma_v,M_{E,N}) = \bigoplus_{w|v_F} H^1(\Gamma_{E_w},\frac{1}{N}\Z/\Z ), \]
where the sum runs over all places $w \in S_E$ lying above $v_F$. The transition map $M_{E,N} \rw M_{K,M}$ induces on the right hand side of this equation the restriction map $H^1(\Gamma_{E_w},\frac{1}{N}\Z/\Z) \rw H^1(\Gamma_{K_u},\frac{1}{M}\Z/\Z)$ given for any $u|w$ by the inclusion $\Gamma_{K_u} \rw \Gamma_{E_w}$. Now the set $H^1(\Gamma_{E_w},\frac{1}{N}\Z/\Z)=\tx{Hom}(\Gamma_{E_w},\frac{1}{N}\Z/\Z)$ is finite and hence the intersection of the kernels of its members is an open subgroup of $\Gamma_{E_w}$. Thus for $K$ large enough the image of the inclusion $\Gamma_{K_u} \rw \Gamma_{E_w}$ is contained in this open subgroup.

For $H^2(\Gamma,M_{E,N})$ we use \cite[Theorem 8.4.4, Proposition 8.4.2]{NSW08} which tell us that we are studying $\varprojlim \hat H^0(\Gamma_E,C[N])$, the limit being taken over all finite Galois extensions $E/F$ and natural numbers $N$. For $K/E$ and $N|M$ the transition map $\hat H^0(\Gamma_K,C[M]) \rw \hat H^0(\Gamma_E,C[N])$ is given by the $M/N$-power map $C[M] \rw C[N]$ and the norm for $K/E$. Recall from \cite[I.\S9]{NSW08} that each individual $\hat H^0(\Gamma_E,C[N])$ is the limit of quotients $C_E[N]/N_{K/E}C_K[N]$ taken over all finite Galois extensions $K/E$. One now observes directly that this double limit is equal to zero: Let $(x_{E,N})$ be an inverse system of elements $x_{E,N} \in \hat H^0(\Gamma_E,C[N])$. Write each $x_{E,N}$ itself as an inverse system $x_{E,N,K} \in C_E[N]/N_{K/E}(C_K[N])$. Now fix $E,N,K$. Since $x_{K,N}$ maps to $x_{E,N}$, each $x_{E,N,K}$ is the image of some $x_{K,N,L}$ under $N_{K/E}$, but $N_{K/E}(x_{K,N,L})$ is zero in the quotient $C_E[N]/N_{K/E}(C_K[N])$.
\end{proof}

\end{lem}

We now come to the proof of the vanishing statements concerning $P_{\dot V}$.

\begin{pro} \label{pro:locrig} For any $v \in \dot V$ we have $H^1(\Gamma_v,P_{\dot V}(\ol{F_v}))=0$.
\end{pro}
\begin{proof}
According to \cite[Theorems 2.7.5, 7.2.6, 7.2.17]{NSW08} we have
\[ H^1(\Gamma_v,P_{\dot V}(\ol{F_v})) = \varprojlim H^1(\Gamma_v,P_{E,\dot S_E,N}(\ol{F_v})) = (\varinjlim H^1(\Gamma_v,M_{E,\dot S_E,N}))^\vee. \]
We obtain from Diagram \eqref{eq:mesn_trdiag} the short exact sequence
\[ 0 \rw \frac{\tx{Maps}(S_E^\C,\frac{1}{N}\Z/\Z)_0^{\Gamma_v}}{\tx{Maps}(S^\C,M_{E,N})_0^{\Gamma_v}} \rw H^1(\Gamma_v,M_{E,\dot S_E,N}) \rw H^1(\Gamma_v,\tx{Maps}(S^\C,M_{E,N})_0). \]
According to Lemmas \ref{lem:mesn_t1} and \ref{lem:mesn_t2} the colimits of the two outer terms are zero.
\end{proof}

\begin{pro} \label{pro:globrig} $H^1(\Gamma,P_{\dot V}(\ol{F}))=0$.

\end{pro}
\begin{proof}
Using the localization sequence
\[ 0 \rw \tx{ker}^1(\Gamma,P_{\dot V}) \rw H^1(\Gamma,P_{\dot V}) \rw \prod_{\dot v \in \dot V} H^1(\Gamma_{\dot v},P_{\dot V}) \]
and Proposition \ref{pro:locrig}
it is enough to show $\tx{ker}^1(\Gamma,P_{\dot V})=0$. According to \cite[Corollary 2.7.6]{NSW08} we have $H^1(\Gamma,P_{\dot V})=\varprojlim H^1(\Gamma,P_{E,\dot S_E,N})$ and thus $\tx{ker}^1(\Gamma,P_{\dot V})=\varprojlim\tx{ker}^1(\Gamma,P_{E,\dot S_E,N})$. According to Poitou-Tate duality \cite[Theorem 8.6.7]{NSW08}, we are trying to show $\varinjlim\tx{ker}^2(\Gamma,M_{E,\dot S_E,N})=0$.

We have the following piece of the long exact cohomology sequence
\[ \begin{aligned}
\rw H^1(\Gamma,\tx{Maps}(S^\C,M_{E,N})_0)&\rw H^1(\Gamma,\tx {Maps}(S^\C_E,\frac{1}{N}\Z/\Z)_0) \rw H^2(\Gamma,M_{E,\dot S_E,N}) \\
\rw H^2(\Gamma,\tx{Maps}(S^\C,M_{E,N})_0)&\rw
\end{aligned} \]
and after applying the exact functor $\varinjlim$ and Lemma \ref{lem:mesn_t2} we obtain the isomorphism
\begin{equation} \label{eq:isop1} \varinjlim H^1(\Gamma,\tx{Maps}(S^\C_E,\frac{1}{N}\Z/\Z)_0) \rw \varinjlim H^2(\Gamma,M_{E,\dot S_E,N}). \end{equation}
We will now study its source. We have the exact sequence
\[ 0 \rw \tx{Maps}(S^\C_E,\frac{1}{N}\Z/\Z)_0 \rw \tx{Maps}(S^\C_E,\frac{1}{N}\Z/\Z) \rw \frac{1}{N}\Z/\Z \rw 0. \]
The transition map for $K/E$ induces the multiplication by $[K:E]$-map on the $\frac{1}{N}\Z/\Z$-factor. It follows that the natural map
\[ \varinjlim H^1(\Gamma,\tx{Maps}(S^\C_E,\frac{1}{N}\Z/\Z)_0) \rw \varinjlim H^1(\Gamma,\tx{Maps}(S^\C_E,\frac{1}{N}\Z/\Z))\]
is an isomorphism. For a fixed $E$ we have the isomorphism
\[ H^1(\Gamma,\tx{Maps}(S^\C_E,\frac{1}{N}\Z/\Z)) \rw \bigoplus_{v \in S^\C} H^1(\Gamma_E \cdot \Gamma_{\dot v},\frac{1}{N}\Z/\Z). \]
On the target of this isomorphism the transition map is given, at each place $v \in S^\C$, by the restriction map $H^1(\Gamma_E\cdot \Gamma_{\dot v},\frac{1}{N}\Z/\Z) \rw H^1(\Gamma_K\cdot \Gamma_{\dot v},\frac{1}{N}\Z/\Z)$ followed by multiplication by $|\Gamma_{K/E,\dot v}|$. For any given element of the right hand side at level $E$ we may thus choose $K$ large enough so that the $v$-component at all $p$-adic places $v \in S$ becomes zero and moreover the restriction to $\Gamma_K$ of the $v$-component at all real places $v \in S$ is zero. We may also assume that $\dot v$ is complex in $E$ for any real $v \in S$, so that $\Gamma_{\dot v} \cap \Gamma_E = \{1\} = \Gamma_{\dot v} \cap \Gamma_K$. We obtain
\[ \varinjlim H^1(\Gamma,\tx{Maps}(S_E,\frac{1}{N}\Z/\Z)) \cong \bigoplus_{v \in V_{F,\R}}\tx{Hom}(\Gamma_{\dot v},\Q/\Z). \]
Consider now the injective map
\begin{equation} \label{eq:q1} \bigoplus_{v \in S_\R} H^1(\Gamma_{\dot v},\frac{1}{N}\Z/\Z) \rw H^1(\Gamma,\tx{Maps}(S_E,\frac{1}{N}\Z/\Z)), \end{equation}
which, according to the argument just given, becomes bijective in the limit. The composition of this map with the localization map
\[ H^1(\Gamma,\tx{Maps}(S_E,\frac{1}{N}\Z/\Z)) \rw \bigoplus_{v \in S_\R} H^1(\Gamma_{\dot v},\tx{Maps}(S_E,\frac{1}{N}\Z/\Z)) \]
stays injective, because it factors the identity on $\bigoplus_{v \in S_\R} H^1(\Gamma_{\dot v},\frac{1}{N}\Z/\Z)$.

Take an element of $\varinjlim\tx{ker}^2(\Gamma,M_{E,\dot S_E,N})$. Via \eqref{eq:isop1}, this element lifts to an element of $H^1(\Gamma,\tx{Maps}(S_E,\frac{1}{N}\Z/\Z)_0)$ for a suitable choice of $E$, $S$, and $N$. Enlarging this choice if necessary this element can be brought to lie in the image of the injection \eqref{eq:q1}. The element of $H^1(\Gamma,\tx{Maps}(S_E,\frac{1}{N}\Z/\Z)_0)$ has the property that its localization at each place $v \in \dot V$ is killed by the connecting homomorphism to $H^2(\Gamma_{\dot v},M_{E,\dot S_E,N})$ and thus lifts to an element of $H^1(\Gamma_{\dot v},\tx{Maps}(S,M_{E,N})_0)$. Since the colimit of these groups vanishes, we may enlarging $E$ sufficiently to arrange that for each $\dot v \in S_\R$ this lift becomes zero. We thus achieve that the localization at each $\dot v \in S_\R$ of the element of $H^1(\Gamma,\tx{Maps}(S_E,\frac{1}{N}\Z/\Z))$ that we are looking at is trivial. But this element came from $\bigoplus_{v \in S_\R} H^1(\Gamma_{\dot v},\frac{1}{N}\Z/\Z)$ and thus must have been trivial itself.
\end{proof}

\subsection{The canonical class} \label{sub:canclass}
In Subsection \ref{sub:pes} we constructed a canonical element $(\xi_i)$ of the inverse limit $\varprojlim H^2(\Gamma,P_i(\ol{F}))$. According to \cite[Corollary 2.7.6]{NSW08} we have the exact sequence
\[ 0 \to \varprojlim\nolimits^1 H^1(\Gamma,P_i(\ol{F})) \to H^2(\Gamma,P_{\dot V}(\ol{F})) \to \varprojlim H^2(\Gamma,P_i(\ol{F})) \to 0 \]
We do not know if the $\varprojlim^1$-term vanishes. The purpose of this section is to show that nonetheless there is a canonical element $\xi \in H^2(\Gamma,P_{\dot V}(\ol{F}))$ that lifts the system $(\xi_i)$. For most of our applications we can make do with an arbitrary lift of $(\xi_i)$, so the reader is encouraged to skip this subsection on a first reading.

Let us abbreviate $P_i=P_{E_i,\dot S_i,N_i}$ and $M_i=M_{E_i,\dot S_i,N_i}$. Put $M=\varinjlim M_i$ and $P(R)=\varprojlim P_i(R) = \tx{Hom}_\Z(M,R^\times)$ for any $F$-algebra $R$. Let $p_i : P(R) \to P_i(R)$ be the natural projection. We will in particular be interested in $R=\bar \A=\ol{F}\otimes_F \A = \varinjlim_E \A_E$, where $E$ runs over the set of finite Galois extensions of $F$. We have $P_i(\ol{\A})=\tx{Hom}_\Z(M_i,\ol{\A}^\times) = \varinjlim_E P_i(\A_E)$. Note here that $P(\bar \A)$ is not the same as $\varinjlim_E P(\A_E)$, the latter being in fact the trivial abelian group.

For each $v \in \dot V$ let $\bar \A_v = \ol{F} \otimes_F F_v$. The group of units $\bar \A_v^\times$ in this ring is a smooth $\Gamma$-module, and is in fact isomorphic to the smooth induction $\tx{Ind}_{\Gamma_v}^\Gamma \ol{F_v}^\times$. We have the Shapiro isomorphisms $S^k_v : H^k(\Gamma_v,P_i(\ol{F_v})) \to H^k(\Gamma,P_i(\bar\A_v))$. They are functorial and hence provide an inverse system with respect to $i$. Furthermore we have the isomorphisms $H^k(\Gamma,P(\bar \A_v)) = \varprojlim H^k(\Gamma,P_i(\bar\A_v))$ and $H^k(\Gamma_v,P(\ol{F_v})) = \varprojlim H^k(\Gamma,P_i(\ol{F_v}))$ of \cite[Theorem 3.5.8]{Weibel94} (note that $\bar F_v^\times$ and $\bar \A_v^\times$ are divisible groups, so for $j>i$ the maps $P_j(\bar F_v) \to P_i(\bar F_v)$ and $P_j(\bar \A_v) \to P_i(\bar \A_v)$ are surjective). The inverse system of Shapiro isomorphisms thus gives the Shapiro isomorphism $S^k_v : H^k(\Gamma_v,P(\ol{F_v})) \to H^k(\Gamma,P(\bar \A_v))$. The local canonical class $\xi_v \in H^2(\Gamma_v,u_v)$ maps via the map $S^2_v \circ \tx{loc}_v$ to a class in $H^2(\Gamma,P(\bar\A_v))$.

We will now construct a canonical class $x \in H^2(\Gamma,P(\ol{\A}))$ that, for each $v \in \dot V$, maps to $S^2_v(\tx{loc}_v(\xi_v)) \in H^2(\Gamma,P(\bar \A_v))$ under the map induced by the projection $\bar \A \to \bar \A_v$. This is not entirely obvious, because there is no a-priori reason for the natural map $H^2(\Gamma,P(\ol{\A})) \to \prod_v H^2(\Gamma,P(\ol{\A_v}))$ to be either injective or surjective.

For the construction of the class $x$ we fix for each $v$ a 2-cocycle $\dot\xi_v$ representing the class $\xi_v \in H^2(\Gamma_v,u_v)$. We also fix for each $v$ a section $s : \Gamma_v \lmod \Gamma \to \Gamma$ as in Appendix \ref{app:shap}, thereby obtaining an explicit realization of the Shapiro map $\dot S^2_v : C^2(\Gamma_v,X) \to C^2(\Gamma,\tx{Ind}_{\Gamma_v}^\Gamma X)$ on the level of cochains. It is functorial in $X$, commutes with the differentials on both sides, and is a section of the map in the opposite direction given by restriction to $\Gamma_v$ followed by evaluation of elements of $\tx{Ind}_{\Gamma_v}^\Gamma X$ at $1$. The functoriality of $\dot S^2_v$ gives us an inverse system of maps $\dot S^2_v : C^2(\Gamma_v,P_i(\ol{F_v})) \to C^2(\Gamma,P_i(\bar \A_v))$ which splices together to a map $\dot S^2_v : C^2(\Gamma_v,P(\ol{F_v})) \to C^2(\Gamma,P(\bar \A_v))$. Let $\dot x_v = \dot S^2_v(\tx{loc}_v(\dot\xi_v)) \in Z^2(\Gamma,P(\bar \A_v))$. It is important to observe at this point that for a fixed $i$, $p_i(\dot x_v)$ is the constant 2-cocycle $1$ for all $v \notin S_i$, because for such $v$ we have $p_i\circ \dot S^2_v \circ\tx{loc}_v=\dot S^2_v \circ p_i \circ\tx{loc}_v$ but $p_i\circ\tx{loc}_v=1$. It follows that $\prod_v p_i(\dot x_v) \in Z^2(\Gamma,P_i(\bar \A))$. Moreover, these elements form an inverse system in $i$ and hence lead to $\dot x \in Z^2(\Gamma,P(\bar \A))$.

We claim that the class of $\dot x$ in $H^2(\Gamma,P(\ol{\A}))$ is independent of the choices involved in its construction -- the choice of 2-cocycle $\dot\xi_v \in Z^2(\Gamma_v,u_v)$ representing the class $\xi_v$ for each $v$, and the choice of a section $\Gamma_{\dot v} \lmod \Gamma \to \Gamma$ for each $v$. First, if for each $v \in \dot V$ we replace $\dot\xi_v$ by $\dot\xi_v \cdot \partial c_v$ for some $c_v \in C^1(\Gamma,u_v)$, then $\dot x_v$ is replaced by $\dot x_v' := \dot x_v \cdot \partial \dot S^1_v(\tx{loc}_v(c_v))$. For a fixed $i$, we have that $p_i(\dot x_v')=p_i(\dot x_v) \cdot \partial C_{v,i}$, where $C_{v,i}=\dot S^1_v(p_i(\tx{loc}_v(c_v)))$. For all $v \notin S_i$ we have $C_{v,i}=1$ and so $C_i := \prod_v C_{v,i}$ belongs to $C^1(\Gamma,P_i(\bar \A))$ and the differential $\partial C_i$ measures the difference between $\prod_v p_i(\dot x_v)$ and $\prod_v p_i(\dot x_v')$. Letting $C=\varprojlim C_i \in C^1(\Gamma,P(\bar \A))$ we obtain a 1-cochain whose differential measures the difference between $\dot x$ and $\dot x'$.

Next, to see that the choices of sections $\Gamma_{\dot v} \lmod \Gamma \to \Gamma$ don't matter, for each $v \in \dot V$ let $s_v$ and $s'_v$ be two choices of sections and let $\dot S^2_v$ and $\dot S'^2_v$ denote the functorial Shapiro maps $C^2(\Gamma_v,X) \to C^2(\Gamma,\tx{Ind}_{\Gamma_v}^\Gamma X)$ respectively. Let $\dot x$ and $\dot x'$ be the two elements of $Z^2(\Gamma,P(\bar \A))$ obtained this way. Lemma \ref{lem:sh4} gives us an explicit formula for a 1-cochain $c_v \in C^1(\Gamma,P_i(\bar\A_v))$ such that $\partial c_v=\dot S_v(\tx{loc}_v(\dot\xi_v))\cdot \dot S'_v(\tx{loc}_v(\dot\xi_v))^{-1}$. From the explicit formula we see that for a fixed $i$ and $v \notin S_i$ we have $p_i(c_v)=1$ in $C^1(\Gamma,P_i(\bar \A_v))$. It follows that $c=\prod_v c_v$ belongs to $C^1(\Gamma,P(\bar \A))$ and satisfies $\dot x \cdot \dot x'^{-1}=\partial c$.

We have thus obtained the canonical class $x \in H^2(\Gamma,P(\ol{\A}))$. We now want to argue that there is a unique element of $H^2(\Gamma,P(\ol{F}))$ whose image in $H^2(\Gamma,P(\bar\A))$ is $x$ and whose image in $\varprojlim H^2(\Gamma,P_i(\ol{F}))$ is $(\xi_i)$. However, it turns out that working with $H^2(\Gamma,P(\bar \A))$ is too difficult for this task. We shall instead replace the groups $H^k(\Gamma,P_i(\bar \A))$ with $H^k(\A,T_i \to U_i)$, where $T_i \to U_i$ is an isogeny of tori with kernel $P_i$, and the latter cohomology groups are the ones defined in \cite[Appendix C]{KS99}. Those groups have the advantage of carrying a good topology, which we will use. We begin with the following.

\begin{lem} For each $i$ there exists an isogeny of tori $f_i : T_i \rw U_i$ defined over $F$ and with kernel $P_i$. We have the commutative diagram
\[ \xymatrix{
	&&1\ar[d]&1\ar[d]\\
	&&K_i\ar[d]\ar[r]&K'_i\ar[d]&\\
	1\ar[r]&P_{i+1}\ar[r]\ar[d]&T_{i+1}\ar[r]\ar[d]&U_{i+1}\ar[r]\ar[d]&1\\
	1\ar[r]&P_i\ar[r]\ar[d]&T_i\ar[r]\ar[d]&U_i\ar[r]\ar[d]&1\\
	&1&1&1
}
\]
in which $K_i$ and $K'_i$ are tori.
\end{lem}
\begin{proof}
We will inductively construct exact sequences of $\Gamma$-modules
\[ 0 \rw X_i \rw Y_i \rw M_i \rw 0, \]
where $M_i=M_{E_i,\dot S_i,N_i}$ and $X_i$ and $Y_i$ are finite-rank free $\Z$-modules. For $i=1$ we take $Y_1$ to be the free $\Z[\Gamma_{E_1/F}]$-module generated by the elements of $M_1$, on which we let $\Gamma$ operate through its quotient $\Gamma_{E_1/F}$, and take for $Y_1 \rw M_1$ the obvious map. Assuming the $i$-th exact sequence is constructed, we take $Y_{i+1}$ to be the direct sum of $Y_i \oplus Y_i'$, where $Y'_i$ is the free $\Z[\Gamma_{E_{i+1}/F}]$-module generated by the set $M_{i+1} \sm M_i$. The map $Y_{i+1} \rw M_{i+1}$ is given on $Y_i$ by the composition $Y_i \rw M_i \rw M_{i+1}$ and on $Y_i'$ by the obvious map to $M_{i+1}$. The result is a surjective map $Y_{i+1} \rw M_{i+1}$ of $\Gamma$-modules. We then obtain the diagram
\[ \xymatrix{
	0\ar[r]&X_{i+1}\ar[r]&Y_{i+1}\ar[r]&M_{i+1}\ar[r]&0\\
	0\ar[r]&X_i\ar[r]\ar[u]&Y_i\ar[r]\ar[u]&M_i\ar[r]\ar[u]&0
}\]
with exact rows, where $Y_i \rw Y_{i+1} = Y_i \oplus Y_i'$ is given by the obvious inclusion. By construction the cokernel of this map is just $Y_i'$ and hence $\Z$-free. The kernel-cokernel lemma implies that the cokernel of $X_i \rw X_{i+1}$ is a submodule of the cokernel of $Y_i \rw Y_{i+1}$ and is thus also $\Z$-free.
\end{proof}

For every $F$-algebra $R$ we have the exact sequence
\begin{equation} \label{eq:urex1} 1 \rw P_i(R) \rw T_i(R) \rw U_i(R). \end{equation}
The last map need not be surjective, even when $R$ is an $\ol{F}$-algebra, as for example in the case $R=\bar\A$. This exact sequence leads to an injective map
\[ H^1(\Gamma,P_i(\ol{\A})) \rw H^1(\A,T_i \rw U_i). \]
Let us recall that $H^j(\A,T_i \rw U_i)$ are the cohomology groups of the complex
\[ C^j(\A,T_i \rw U_i) := C^j(\Gamma,T_i(\ol{\A})) \oplus C^{j-1}(\Gamma,U_i(\ol{\A})) \]
with differential sending $(c_1,c_2)$ to $(\partial c_1,f_i(c_1)-\partial c_2)$. On the right we are taking continuous cochains of the profinite group $\Gamma$ valued in the discrete $\Gamma$-modules $T_i(\ol{\A})$ and $U_i(\ol{\A})$. Due to the fact the kernels of $T_{i+1} \rw T_i$ and $U_{i+1} \rw U_i$ are tori, the maps $T_{i+1}(\ol{\A}) \rw T_i(\ol{\A})$ and $U_{i+1}(\ol{\A}) \rw U_i(\ol{\A})$ are surjective, which implies that the induced map
\[ C^j(\A,T_{i+1} \rw U_{i+1}) \rw C^j(\A,T_i \rw U_i) \]
is also surjective. It follows that for fixed $j$ the inverse system $C^j(\A,T_i \rw U_i)$ satisfies the Mittag-Lefler condition and then from \cite[Theorem 3.5.8]{Weibel94} we obtain the exact sequence
\[ 1 \rw \varprojlim\nolimits^1 H^1(\A,T_i \rw U_i) \rw H^2(\A,T \rw U) \rw \varprojlim H^2(\A,T_i \rw U_i) \rw 1, \]
where in the middle we are taking the cohomology of the complex
\[ C^j(\A,T \rw U) := \varprojlim C^j(\A,T_i \rw U_i) = C^j(\Gamma,T(\ol{\A})) \oplus C^{j-1}(\Gamma,U(\ol{\A})), \]
with
\[ T(\ol{\A}) = \varprojlim T_i(\ol{\A})\qquad \tx{and}\qquad U(\ol{\A}) = \varprojlim U_i(\ol{\A}), \]
being continuous $\Gamma$-modules endowed with the topology of the inverse limit of discrete $\Gamma$-modules.

Taking the inverse limit of \eqref{eq:urex1} for $R=\ol{\A}$ we obtain the exact sequence
\[ 1 \rw P(\ol{\A}) \rw T(\ol{\A}) \rw U(\ol{\A}) \]
and hence a map (which may fail to be injective)
\[ H^2(\Gamma,P(\ol{\A})) \rw H^2(\A,T \rw U). \]
We also have the map $H^2(\Gamma,P(\ol{F})) \to H^2(\Gamma,P(\bar \A))$ induced by the natural embedding $\ol{F} \to \bar\A$.

\begin{pro} \label{pro:canclass} There exists a unique element of $H^2(\Gamma,P(\ol{F}))$ whose image in $\varprojlim H^2(\Gamma,P_i(\ol{F}))$ equals the canonical system $(\xi_i)$ and whose image in $H^2(\A,T \to U)$ coincides with the image of the class $x \in H^2(\Gamma,P(\bar \A))$ there.
\end{pro}
\begin{proof}
Let $\tilde\xi \in H^2(\Gamma,P(\ol{F}))$ be any preimage of $(\xi_i) \in \varprojlim H^2(\Gamma,P_i(\ol{F}))$ and let $\tilde\xi_\A$ be its image in $H^2(\Gamma,P(\bar \A))$. We map $\tilde\xi_\A$ and $x$ via the maps
\[ H^2(\Gamma,P(\ol{\A})) \rw H^2(\A,T \rw U) \rw \varprojlim H^2(\A,T_i \rw U_i). \]
According to \cite[Theorem C.1.B]{KS99} and the surjectivity of $T_i(\ol{F_v}) \rw U_i(\ol{F_v})$ we have
\[ H^2(\A,T_i \rw U_i) = \prod\nolimits_v' H^2(F_v,T_i \rw U_i) = \prod\nolimits_v' H^2(F_v,P_i). \]

By construction, for any $\dot v \in \dot V$, the image of $p_i(\tilde\xi_\A)$ in $H^2(\Gamma_{\dot v},P_i(\ol{F_v}))$ under restriction to $\Gamma_{\dot v}$ followed by projection to the $\dot v$-component is the same as the image of $p_i(\tilde\xi)=\xi_i$ under the usual localization map $H^2(\Gamma,P_i(\ol{F})) \rw H^2(\Gamma_{\dot v},P_i(\ol{F_v}))$ given by restriction to $\Gamma_{\dot v}$ followed by inclusion $\ol{F} \rw \ol{F_{\dot v}}$. At the same time, the image of $p_i(x)$ in $H^2(\Gamma_{\dot v},P_i(\ol{F_v}))$ is $p_i(\tx{loc}_v(\xi_v))$. According to Corollary \ref{cor:xiloc}, these two elements are cohomologous in $H^2(\Gamma_v,P_i(\ol{F_v}))$ for all $i$. It follows that the images of $\tilde\xi_\A$ and $x$ in $\varprojlim H^2(\A,T_i \rw U_i)$ are equal. The next lemma will allow us to improve this equality.

\begin{lem} \label{lem:l1iso} The natural map
\[ \varprojlim\nolimits^1 H^1(\Gamma,P_i(\ol{F})) \rw \varprojlim\nolimits^1 H^1(\A,T_i \rw U_i) \]
is an isomorphism.
\end{lem}

\begin{proof}
Recall the cohomology group $H^1(\A/F,T_i \rw U_i)$, also defined in \cite[Appendix C]{KS99}. The authors define (see (C.2.6)) a closed subgroup $H^1(\A/F,T_i \rw U_i)_1$ which they show \cite[Lemma C.2.D]{KS99} is compact. Since $T_i \rw U_i$ is an isogeny, this closed subgroup is in fact all of $H^1(\A/F,T_i \rw U_i)$, which is thus itself compact. From \cite[(C.1.1)]{KS99} we obtain the short exact sequence
\[ 1 \rw H^1(F,T_i \rw U_i)/\tx{ker}^1(F,T_i \rw U_i)\rw
H^1(\A,T_i \rw U_i) \rw \tx{cok}^1(F,T_i \rw U_i) \rw 1. \]
According to \cite[Lemma C.3.A]{KS99} the last term in this sequence is an open subgroup of $H^1(\A/F,T_i \rw U_i)$, hence also closed, and thus compact. It follows that $\varprojlim^1\tx{cok}^1(F,T_i \rw U_i)=0$.

We claim that also $\varprojlim\tx{cok}^1(F,T_i \rw U_i)=0$. To see this, we apply \cite[Lemma C.3.B]{KS99} and see that the compact group $\tx{cok}^1(F,T_i \rw U_i)$ is Pontryagin dual to the discrete group $H^1(W_F,\hat U_i \rw \hat T_i)_\tx{red}/\tx{ker}^1(W_F,\hat U_i \rw \hat T_i)_\tx{red}$. Since the map $\hat U_i \to \hat T_i$ is an isogeny, the injective map $H^1(F,\hat U_i \to \hat T_i) \to H^1(W_F,\hat U_i \to \hat T_i)$ is in fact an isomorphism, and the same is true for the local Weil groups $W_{F_v}$ in place of $W_F$ as well, so the above quotient is equal to $H^1(F,\hat U_i \rw \hat T_i)_\tx{red}/\tx{ker}^1(F,\hat U_i \rw \hat T_i)_\tx{red}$. The argument in the proof of \cite[Lemma C.3.C]{KS99} shows that $H^1(F,\hat U_i \rw \hat T_i)_\tx{red} = H^2(F,X^*(U_i) \rw X^*(T_i))=H^1(F,M_i)$ and in the same way $\tx{ker}^1(F,\hat U_i \rw \hat T_i)_\tx{red} = \tx{ker}^1(F,M_i)$. The claim will follow once we show that $\varinjlim H^1(F,M_i)/\tx{ker}^1(F,M_i)=0$. The exactness of $\varinjlim$ reduces this to showing $\varinjlim H^1(F,M_i)=0$. This follows from the exact sequence of Lemma \ref{lem:mesn_phiseq} and Lemmas \ref{lem:mesn_t1} and \ref{lem:mesn_t2}.

The (not too) long exact sequence for $\varprojlim$ now tells us that the first map in the above displayed short exact sequence becomes an isomorphism on the level of $\varprojlim^1$. But the finiteness of $\tx{ker}^1(F,T_i \rw U_i)$ by \cite[Lemma C.3.B]{KS99} implies that the map $H^1(F,T_i \rw U_i) \rw H^1(F,T_i \rw U_i)/\tx{ker}^1(F,T_i \rw U_i)$ also becomes an isomorphism on the level of $\varprojlim^1$. The lemma now follows from $H^1(F,P_i)=H^1(F,T_i \rw U_i)$.
\end{proof}

We may thus modify $\tilde\xi$ by an element of $\varprojlim^1 H^1(\Gamma,P_i(\ol{F}))$ to achieve that the images of $\tilde\xi_\A$ and $x$ in $H^2(\A,T \rw U)$ are equal. This proves the existence claim of the proposition. The uniqueness follows immediately from the injectivity of $\varprojlim^1H^1(\Gamma,P_i(\ol{F})) \to \varprojlim^1 H^1(\A,T_i \to U_i) \to H^2(\A,T \to U)$.
\end{proof}

\begin{dfn} \label{dfn:canclass} The canonical class $\xi \in H^2(\Gamma,P(\ol{F}))$ is the element whose existence and uniqueness is asserted in Proposition \ref{pro:canclass}.
\end{dfn}

Note that $\xi$ is independent of the choice of tower of isogenies $f_i : T_i \to U_i$. Indeed, if $f_i^{(1,2)} : T_i^{(1,2)} \to U_i^{(1,2)}$ are two choices of towers of isogenies we can define a third tower $f_i^{(3)} : T_i^{(3)} \to U_i^{(3)}$ by taking $T_i^{(3)}=(T_i^{(1)} \times T_i^{(2)}) / P_i$, with $P_i$ embedded via the anti-diagonal map $p \mapsto (p,p^{-1})$, and taking $U_i^{(3)} = U_i^{(1)} \times U_i^{(2)}$ and $f_i^{(3)} = f_i^{(1)} \times f_i^{(2)}$. Let $\xi^{(1,2,3)}$ be the three versions of $\xi$ obtained from applying Proposition \ref{pro:canclass} with $H^2(\A,T^{(1,2,3)} \to U^{(1,2,3)})$. Since the exact sequence $1 \to P_i \to T_i^{(3)} \to U_i^{(3)} \to 1$ maps to the two exact sequences $1 \to P_i \to T_i^{(1,2)} \to U_i^{(1,2)} \to 1$, we see that $\xi^{(1)}=\xi^{(3)}=\xi^{(2)}$.

\subsection{The cohomology set $H^1(P_{\dot V} \rw \mc{E}_{\dot V},Z \rw G)$} \label{sub:coh}

Let $\xi$ stand for the canonical class in $H^2(\Gamma,P_{\dot V})$ of Definition \ref{dfn:canclass}. In fact, for most of our purposes we can take $\xi$ to be any preimage in $H^2(\Gamma,P_{\dot V})$ of the canonical system $(\xi_i) \in \varprojlim H^2(\Gamma,P_{E_i,\dot S_i,N_i})$. The reader may freely assume that $\xi$ is an arbitrary such lift until Subsection \ref{sub:ram}.

Let $1\rw P_{\dot V}(\ol{F}) \rw \mc{E}_{\dot V} \rw \Gamma  \rw 1$ be any extension in the isomorphism class given by $\xi$. Let $G$ be an affine algebraic group and $Z \subset G$ a finite central subgroup, both defined over $F$. The pro-finite group $\mc{E}_{\dot V}$ acts continuously on the discrete group $G(\ol{F})$ through its quotient $\Gamma$ and the set $H^1(\mc{E}_{\dot V},G(\ol{F}))$ of cohomology classes of continuous 1-cocycles of $\mc{E}_{\dot V}$ valued in $G(\ol{F})$ is defined. The restriction to $P_{\dot V}(\ol{F})$ of such a 1-cocycle is a continuous group homomorphism $P_{\dot V}(\ol{F}) \rw G(\ol{F})$. We define $H^1(P_{\dot V} \rw \mc{E}_{\dot V},Z \rw G)$ to be the subset of $H^1(\mc{E}_{\dot V},G(\ol{F}))$ consisting of the classes of those 1-cocycles whose restriction to $P_{\dot V}(\ol{F})$ takes image in $Z(\ol{F})$. This restriction is then a $\Gamma$-equivariant continuous group homomorphism $P_{\dot V}(\ol{F}) \rw Z(\ol{F})$. By continuity it factors through the projection $P_{\dot V} \rw P_{E_i,\dot S_i,N_i}$ for a suitable index $i$ and is then given by a $\Gamma$-equivariant group homomorphism $P_{E_i,\dot S_i,N_i}(\ol{F}) \rw Z(\ol{F})$. The source and target of this homomorphism both being finite, this homomorphism is automatically algebraic.

According to Proposition \ref{pro:globrig} the automorphisms of the extension $\mc{E}_{\dot V}$ are all inner automorphisms coming from $P_{\dot V}$. Each such inner automorphism induces the identity automorphism on $H^1(P_{\dot V} \rw \mc{E}_{\dot V},Z \rw G)$. It follows that this cohomology set does not depend on the choice of extension $\mc{E}_{\dot V}$ that realizes the isomorphism class $\xi$.

Let $\mc{A}$ be the category whose objects are pairs $(Z,G)$ as above and where a morphism $(Z_1,G_1) \rw (Z_2,G_2)$ is a morphism $G_1 \rw G_2$ of algebraic groups defined over $F$ mapping $Z_1$ to $Z_2$. We will usually denote an object of $\mc{A}$ by $[Z \rw G]$. It is then clear that $[Z \rw G] \mapsto H^1(P_{\dot V} \rw \mc{E}_{\dot V},Z \rw G)$ is a functor $\mc{A} \rw \tx{Sets}$.

By construction we have the inflation-restriction exact sequence
\begin{equation} \label{eq:infres1}
1\rw H^1(\Gamma,G)\rw H^1(P_{\dot V}\rw \mc{E}_{\dot V},Z \rw G)\rw \tx{Hom}_F(P_{\dot V},Z)\rw H^2(\Gamma,G)
\end{equation}
where the $H^2$-term is to be ignored unless $G$ is abelian. Just as in the local case treated in \cite[\S3.3]{KalRI}, the map $\tx{Hom}_F(P_{\dot V},Z) \rw H^2(\Gamma,G)$ is given as the composition of the map $\Theta^P_{\dot V}$ defined in \eqref{eq:m_h2} with the natural map $H^2(\Gamma,Z) \rw H^2(\Gamma,G)$. In fact, the inflation-restriction sequence fits into the commutative diagram with exact rows
\[ \xymatrix{
H^1(\Gamma,G)\ar[r]\ar@{=}[d]&H^1(P_{\dot V}\rw \mc{E}_{\dot V},Z \rw G)\ar[r]\ar[d]&\tx{Hom}_F(P_{\dot V},Z)\ar[r]\ar[d]^{\Theta^P_{\dot V}} &H^2(\Gamma,G)\ar@{=}[d]\\
H^1(\Gamma,G)\ar[r]&H^1(\Gamma,G/Z)\ar[r]&H^2(\Gamma,Z)\ar[r]&H^2(\Gamma,G)
}
\]

\begin{lem} \label{lem:cohsurj} If $G$ is either abelian or connected and reductive, then the map
\[ H^1(P_{\dot V}\rw \mc{E}_{\dot V},Z \rw G) \rw H^1(\Gamma,G/Z) \]
is surjective.
\end{lem}
\begin{proof} If $G$ is abelian this follows immediately from the above diagram and the five-lemma. Assume now that $G$ is connected and reductive. Let $h \in H^1(\Gamma,G/Z)$. By Lemma \ref{lem:cohapprox} there exists a maximal torus $T \subset G$ such that $h$ belongs to the image of $H^1(\Gamma,T/Z) \rw H^1(\Gamma,G/Z)$. The claim now follows from the functoriality of $H^1(P_{\dot V} \rw \mc{E}_{\dot V},-)$ and the already established surjectivity for $[Z \rw T]$.
\end{proof}

\begin{lem} \label{lem:cohapprox1}
If $G$ is connected and reductive, then
for each $x \in H^1(P_{\dot V} \rw \mc{E}_{\dot V},Z \rw G)$ there exists a maximal torus $T \subset G$ such that $x$ is in the image of $H^1(P_{\dot V} \rw \mc{E}_{\dot V},Z \rw T)$.
\end{lem}
\begin{proof}
Let $T \subset G$ be a maximal torus such that the image of $x$ in $H^1(\Gamma,G/Z)$ belongs to the image of $H^1(\Gamma,T/Z) \rw H^1(\Gamma,G/Z)$ by Lemma \ref{lem:cohapprox}. Chasing around the commutative diagram with exact columns
\[ \xymatrix{
H^1(\mc{E}_{\dot V},Z)\ar@{=}[r]\ar[d]&H^1(\mc{E}_{\dot V},Z)\ar[d]\\
H^1(P_{\dot V} \rw \mc{E}_{\dot V},Z \rw T)\ar[r]\ar[d]&H^1(P_{\dot V} \rw \mc{E}_{\dot V},Z \rw G)\ar[d]\\
H^1(\Gamma,T/Z)\ar[r]&H^1(\Gamma,G/Z)
}\]
shows that $x$ belongs to the image of $H^1(P_{\dot V} \rw \mc{E}_{\dot V},Z \rw T)$.
\end{proof}

Let $v \in \dot V$. We will construct a localization map
\begin{equation} \label{eq:locmap2}
loc_v : H^1(P_{\dot V}\rw \mc{E}_{\dot V},Z \rw G) \rw H^1(u_v \rw W_v,Z \rw G)
\end{equation}
that fits into the commutative diagram
\[ \xymatrix{
	H^1(\Gamma,G)\ar[r]\ar[d]&H^1(P_{\dot V}\rw \mc{E}_{\dot V},Z \rw G)\ar[r]\ar[d]&\tx{Hom}_F(P_{\dot V},Z)\ar[d]\\
	H^1(\Gamma_v,G)\ar[r]&H^1(u_v \rw W_v,Z \rw G)\ar[r]&\tx{Hom}_{F_v}(u_v,Z)
}\]
where the left vertical map is the usual localization map in Galois cohomology, and the right vertical map is given by the map $\tx{loc}_v^P$ defined in \eqref{eq:locvp}.

The construction of $loc_v$ involves the following diagram
\begin{equation} \label{eq:diagloc2} \xymatrix{
1\ar[r]&u_v(\ol{F_v})\ar[r]\ar[d]^{loc_v^{P}}&W_v\ar[r]\ar@{.>}[d]&\Gamma_v\ar[r]\ar@{=}[d]&1\\
1\ar[r]&P_{\dot V}(\ol{F_v})\ar[r]&\Box_2\ar[r]&\Gamma_v\ar[r]&1\\
1\ar[r]&P_{\dot V}(\ol{F})\ar[u]\ar[r]&\Box_1\ar[r]\ar[d]\ar[u]&\Gamma_v\ar[r]\ar[d]\ar@{=}[u]&1\\
1\ar[r]&P_{\dot V}(\ol{F})\ar@{=}[u]\ar[r]&\mc{E}_{\dot V}\ar[r]&\Gamma\ar[r]&1
} \end{equation}
Here $\Box_1$ is formed as the pull-back and $\Box_2$ is formed as the push-out. The dotted arrow is chosen so that the diagram commutes. Its existence is guaranteed by Corollary \ref{cor:xiloc}. While not completely unique, the dotted arrow is unique up to conjugation by elements of $P_{\dot V}(\ol{F_v})$ due to Proposition \ref{pro:locrig}.

Let $z \in Z^1(P_{\dot V} \rw \mc{E}_{\dot V},Z \rw G)$. Let $z_1$ be the composition of $z$ with $\Box_1 \rw \mc{E}_{\dot V}$ and with the inclusion $G(\ol{F}) \rw G(\ol{F_v})$. Recall that $z|_{P_{\dot V}(\ol{F})}$ specifies an element of $\tx{Hom}_F(P_{\dot V},Z)$. Mapping this element to $\tx{Hom}_{F_v}(P_{\dot V},Z)$ and combining it with $z_1$ we obtain a map $z_2 : \Box_2 \rw G(\ol{F_v})$. Composing $z_2$ with the dotted arrow  we obtain finally $z_3 : W_v \rw G(\ol{F_v})$. From the construction it is clear that each of $z_1$, $z_2$, and $z_3$ is a continuous 1-cocycle. The 1-cocycle $z_3$ depends on the choice of dotted arrow, but only up to the inflation to $W_v$ of an element of $B^1(\Gamma,Z)$. In particular, the cohomology class of $z_3$ is independent of the choice of dotted arrow. In this way we obtain the map \eqref{eq:locmap2}. We emphasize that the map \eqref{eq:locmap2} is well-defined not just on the level of cohomology, but already as a map
\[ Z^1(P_{\dot V}\rw \mc{E}_{\dot V},Z \rw G) \rw Z^1(u_v \rw W_v,Z \rw G)/B^1(\Gamma_v,Z). \]

\subsection{Duality for tori} \label{sub:tn+}

Let $\mc{T} \subset \mc{A}$ be the full subcategory consisting of those objects $[Z \rw G]$ for which $G$ is a torus. For $v \in \dot V$, let $\mc{T}_v$ be the category of objects $[Z \rw T]$ with $T$ a torus and $Z \subset T$ a finite subgroup, both defined over $F_v$. We identify $\Gamma_v=\tx{Stab}(v,\Gamma)$ with the absolute Galois group of $F_v$. Base change from $F$ to $F_v$ gives a functor $\mc{T} \rw \mc{T}_v$.

In \cite[\S3]{KalRI} we defined a functor
\[ H^1(u_v \rw W_v,-) : \mc{T}_v \rw \tx{AbGrp}, \]
which we can now view as a functor $\mc{T} \rw \tx{AbGrp}$ by composing it with the base change functor $\mc{T} \rw \mc{T}_v$. In Subsection \ref{sub:coh} we defined a functor
\[ H^1(P_{\dot V} \rw \mc{E}_{\dot V},-) : \mc{T} \rw \tx{AbGrp} \]
as well as a morphism of functors
\[ loc_v : H^1(P_{\dot V} \rw \mc{E}_{\dot V},-) \rw H^1(u_v \rw W_v,-). \]

The purpose of this subsection is to interpret these two functors and this morphism between them in terms of linear algebra. The linear algebraic version of $H^1(u_v \rw W_v,Z \rw T)$ was defined in \cite[\S4.1]{KalRI}. It was denoted by $\bar Y_{+,\tx{tor}}$ there, but in order to emphasize the place $v$ we will now denote it by $\bar Y_{+_v,\tx{tor}}$. Let us recall its construction. Set $Y=X_*(T)$ and $\bar Y=X_*(\bar T)=X_*(T/Z)$. Then $\bar Y_{+_v,\tx{tor}}$ is the torsion subgroup of the quotient $\bar Y/I_vY$, where $I_vY \subset Y$ is the subgroup generated by the elements $\sigma\lambda-\lambda$ for all $\sigma\in\Gamma_v$ and $\lambda \in Y$.

We will now describe the linear algebraic version of $H^1(P_{\dot V} \rw \mc{E}_{\dot V},Z \rw T)$. We have the exact sequence
\[ 0 \rw Y \rw \bar Y \rw A^\vee \rw 0 \]
where $A=X^*(Z)$ and $A^\vee=\tx{Hom}(A,\Q/\Z)$. This is a sequence of $\Gamma_{E_i/F}$-modules for $i>>0$. The $\Gamma_{E_i/F}$-module $\Z[S_{i,E_i}]_0$ is $\Z$-free and tensoring with it leads to the exact sequence
\[ 0 \rw Y[S_{i,E_i}]_0 \rw \bar Y[S_{i,E_i}]_0 \rw A^\vee[S_{i,E_i}]_0 \rw 0. \]
Inside of $A^\vee[S_{i,E_i}]_0$ we have the $\Z$-submodule $A^\vee[\dot S_i]_0$ consisting of those elements that are supported on the subset $\dot S_i$ of $S_{i,E_i}$. We let $\bar Y[S_{i,E_i},\dot S_i]_0$ denote the preimage of $A^\vee[\dot S_i]_0$ in $\bar Y[S_{i,E_i}]_0$. This is a $\Z$-submodule that is not $\Gamma_{E_i/F}$-stable. It contains $Y[S_{i,E_i}]_0$ and we obtain the exact sequence
\[ 0 \rw \frac{Y[S_{i,E_i}]_0}{I_{E_i/F}Y[S_{i,E_i}]_0} \rw \frac{\bar Y[S_{i,E_i},\dot S_i]_0}{I_{E_{i}/F}Y[S_{i,E_i}]_0} \rw A^\vee[\dot S_i]_0  \rw 0 \]

Choose an arbitrary section $s : S_{i,E_i} \rw S_{i+1,E_{i+1}}$ with the property $s(\dot S_i) \subset \dot S_{i+1}$. Define a map
\[ s_! : \bar Y[S_{i,E_i},\dot S_i]_0 \rw \bar Y[S_{i+1,E_{i+1}},\dot S_{i+1}]_0,\quad [s_!f](u) = \begin{cases} f(p(u)),&sp(u)=u\\ 0,&\tx{else} \end{cases}. \]
Note that this is the same definition as the one appearing before Lemma \ref{lem:tnshriek}.

\begin{lem} \label{lem:inftn} The assignment $f \mapsto s_!f$ induces a well-defined homomorphism
\[ !:\frac{\bar Y[S_{i,E_i},\dot S_i]_0}{I_{E_i/F}Y[S_{i,E_i}]_0} \rw \frac{\bar Y[S_{i+1,E_{i+1}},\dot S_{i+1}]_0}{I_{E_{i+1}/F}Y[S_{i+1,E_{i+1}}]_0}
 \]
that is independent of the choice of section $s$.
\end{lem}
\begin{proof}
The argument is essentially the same as in the proof of Lemma \ref{lem:tnshriek}. Indeed, as argued there we see that $s_!(I_{E_i/F}Y[S_{i,E_i}]_0) \subset I_{E_{i+1}/F}Y[S_{i+1,E_{i+1}}]_0$. Since $f(w) \in Y$ for $w \in S_{i,E_i} \sm \dot S_i$ we have $(s_!f)(u) \in Y$ for $u \in S_{i+1,E_{i+1}} \sm \dot S_{i+1}$. Thus we obtain a well-defined homomorphism as claimed. The argument that it does not depend on the choice of section is also the same as the one given in the proof of Lemma \ref{lem:tnshriek}. One just has to note the following: Since we are only considering sections that map $\dot S_i$ into $\dot S_{i+1}$, if $s,s'$ are two such sections and $s(w) \neq s'(w)$, then $w \in S_{i,E_i} \sm \dot S_i$. But for any $f \in \bar Y[S_{i,E_i},\dot S_i]_0$ we then have $f(w) \in Y \subset \bar Y$ and the same argument as before implies $s_!f-s'_!f \in I_{E_{i+1}/F}Y[S_{i+1,E_{i+1}}]_0$.
\end{proof}

We now define
\begin{equation} \label{eq:y+torglob} \bar Y[V_{\ol{F}},\dot V]_{0,+,\tx{tor}} =\varinjlim_i \frac{\bar Y[S_{i,E_i},\dot S_i]_0}{I_{E_i/F}Y[S_{i,E_i}]_0}[\tx{tor}]
\end{equation}
where the transition maps are given by $!$. This is the linear algebraic version of $H^1(P_{\dot V} \rw \mc{E}_{\dot V},Z \rw T)$. We also define
\begin{equation} \label{eq:ytorglob} Y[V_{\ol{F}}]_{0,\Gamma,\tx{tor}} =\varinjlim_i \frac{Y[S_{i,E_i}]_0}{I_{E_i/F}Y[S_{i,E_i}]_0}[\tx{tor}]
\end{equation}
with the same transition maps. This is the linear algebraic version of $H^1(\Gamma,T)$. These two abelian groups then fit into the exact sequence
\[ 0 \rw Y[V_{\ol{F}}]_{0,\Gamma,\tx{tor}} \rw \bar Y[V_{\ol{F}},\dot V]_{0,+,\tx{tor}} \rw A^\vee[\dot V]_{0,\infty}, \]
where the $\Z$-submodule $A^\vee[\dot V]_{0,\infty} \subset A^\vee[\dot V]_0$ was defined in Lemma \ref{lem:hompdotvz}. The fact that the second map takes image in this submodule will follow from Lemma \ref{lem:normtor} below.

The linear algebraic version of the localization map $loc_v$ is a morphism of functors
\[ l_v : \bar Y[V_{\ol{F}},\dot V]_{0,+,\tx{tor}} \rw \bar Y_{+_v,\tx{tor}} \]
defined as follows. Fix an index $i$. Choose a representative $\dot\tau \in \Gamma_{E_i/F}$ for each coset $\tau \in \Gamma_{E_i/F,v} \lmod \Gamma_{E_i/F}$ with the condition that $\dot\tau=1$ for the trivial coset $\tau = \Gamma_{E_i/F,v}$. Then for $f \in \bar Y[S_{i,E_i},\dot S_i]_0$ we define $l_v^if \in \bar Y$ by
\[ l^i_vf = \sum_\tau \dot\tau f(\tau^{-1}v). \]

\begin{lem} \label{lem:loctn}
The assignment $f \mapsto l_v^if$ descends to a group homomorphism
\[ l^i_v : \frac{\bar Y[S_{i,E_i},\dot S_i]_0}{I_{E_i/F}Y[S_{i,E_i}]_0} \rw \frac{\bar Y}{I_vY} \]
that is independent of the choices of $\dot\tau$ and fits in the commutative diagram
\[ \xymatrix{
	\frac{\bar Y[S_{i,E_i},\dot S_i]_0}{I_{E_i/F}Y[S_{i,E_i}]_0}\ar[r]^-{l^i_v}\ar[d]^{!}&\frac{\bar Y}{I_vY}\ar@{=}[d]\\
	\frac{\bar Y[S_{i+1,E_{i+1}},\dot S_{i+1}]_0}{I_{E_{i+1}/F}Y[S_{i+1,E_{i+1}}]_0}\ar[r]^-{l^{i+1}_v}&\frac{\bar Y}{I_vY}
}\]

\end{lem}
\begin{proof}
It is clear that $l^i_v : \bar Y[S_{i,E_i},\dot S_i]_0 \rw \bar Y$ is a group homomorphism. To show that its composition with the projection $\bar Y \rw \bar Y/I_vY$ is independent of the choices of representatives $\dot\tau$, fix a non-trivial coset $\tau_1 \in \Gamma_{E_i/F,v} \lmod \Gamma_{E_i/F}$ and replace its representative $\dot\tau_1$ by $\sigma\dot\tau_1$ with $\sigma \in \Gamma_{E_i/F,v}$. Then $\sum_\tau \dot\tau f(\tau^{-1}v)$ is replaced by $\sum_\tau \dot\tau f(\tau^{-1}v)+\sigma\dot\tau_1 f(\tau_1^{-1}v)-\dot\tau_1 f(\tau_1^{-1}v)$. Since $\tau_1$ is non-trivial and $v \in \dot S_i$, we have $\tau_1^{-1}v \notin \dot S_i$ and consequently $f(\tau_1^{-1}v) \in Y$. This shows that $l^i_v$ is independent of the choice of $\dot\tau_1$. It furthermore shows that if $f \in I_{E_i/F}Y[S_{i,E_i}]_0$, then $l^i_vf \in I_vY$.

We come to the commutativity of the diagram. Recall that the map $!$ is defined in terms of a section $s : S_{i,E_i} \rw S_{i+1,E_{i+1}}$ having the property $s(\dot S_i) \subset \dot S_{i+1}$, which in particular implies that it maps $v_{E_i} \in \dot S_E$ to $v_{E_{i+1}} \in \dot S_{i+1}$. Via the bijections $\Gamma_{E_i/F,v}\lmod \Gamma_{E_i/F} \rw \{w \in S_{i,E_i}:w|v_F\}$ and $\Gamma_{E_{i+1}/F,v}\lmod \Gamma_{E_{i+1}/F} \rw \{u \in S_{i+1,E_{i+1}}:u|v_F\}$ the section $s$ can be thought of as a section $\Gamma_{E_i/F,v} \lmod \Gamma_{E_i/F} \rw \Gamma_{E_{i+1}/F,v}\lmod \Gamma_{E_{i+1}/F}$. We choose a representative $\dot\tau \in \Gamma_{E_{i+1}/F}$ for each coset $\tau \in \Gamma_{E_{i+1}/F,v} \lmod \Gamma_{E_{i+1}/F}$. Then we have
\[
\quad\qquad\ssum{\tau \in \Gamma_{E_{i+1}/F,v}\lmod \Gamma_{E_{i+1}/F}}\dot\tau[(s_!f)(\tau^{-1}v_{E_{i+1}})]\quad =\quad \ssum{\tau \in s(\Gamma_{E_i/F,v}\lmod \Gamma_{E_i/F})}\dot\tau f(p(\tau^{-1}v_{E_{i+1}}))\quad =\quad \ssum{\tau \in \Gamma_{E_i/F,v}\lmod \Gamma_{E_i/F}}\dot\tau f(\tau^{-1}v_{E_i}),
\]
where in the last term we have taken as a representative in $\Gamma_{E_i/F}$ of the coset $\tau \in \Gamma_{E_i/F,v}\lmod \Gamma_{E_i/F}$ the image in $\Gamma_{E_i/F}$ of the representative $\dot\tau \in \Gamma_{E_{i+1}/F}$ of $s(\tau)$.
\end{proof}

We have now completed the linear algebraic descriptions of the cohomology set $H^1(P_{\dot V} \rw \mc{E}_{\dot V},Z \rw T)$ and the localization map $loc_v$. The main result of this subsection is the following.

\begin{thm} \label{thm:tn+} There exists a unique isomorphism
\[ \iota_{\dot V}: \bar Y[V_{\ol{F}},\dot V]_{0,+,\tx{tor}} \rw H^1(P_{\dot V} \rw \mc{E}_{\dot V},Z \rw T) \]
of functors $\mc{T} \rw \tx{AbGrp}$ that fits into the commutative diagram
\[ \xymatrix{
	Y[V_{\ol{F}}]_{0,\Gamma,\tx{tor}}\ar[r]\ar[d]^{TN}&\bar Y[V_{\ol{F}},\dot V]_{0,+,\tx{tor}}\ar[r]\ar[d]^{\iota_{\dot V}}&A^\vee[\dot V]_{0,\infty}\ar[d]\\
	H^1(\Gamma,T)\ar[r]&H^1(P_{\dot V} \rw \mc{E}_{\dot V},Z \rw T)\ar[r]&\tx{Hom}_F(P_{\dot V},Z)
}\]
where $TN$ is obtained from the classical Tate-Nakayama isomorphism by taking the colimit and using Lemma \ref{lem:tnstori} and Corollary \ref{cor:tnshriek}, and the right vertical arrow is the isomorphism of Lemma \ref{lem:hompdotvz}. For each $v \in \dot V$ the following diagram commutes,
\[ \xymatrix{
	\bar Y[V_{\ol{F}},\dot V]_{0,+,\tx{tor}}\ar[r]^{l_v}\ar[d]^{\iota_{\dot V}}&\bar Y_{+_v,\tx{tor}}\ar[d]^{\iota_{v}}\\
	H^1(P_{\dot V} \rw \mc{E}_{\dot V},Z \rw T)\ar[r]^{loc_v}&H^1(u_v \rw W_v,Z \rw T)
}\]
where the right vertical isomorphism is from \cite[\S4]{KalRI}.
\end{thm}

As an immediate consequence of this theorem, we obtain the following description of the collections of local cohomology classes that are the localization of a given global cohomology class.

\begin{cor} \label{cor:tnexact}
We have the following commutative diagram with exact bottom row
\[ \xymatrix{
H^1(P_{\dot V} \rw \mc{E}_{\dot V},Z \rw T)\ar[r]^-{(loc_v)_v}&\bigoplus\limits_{v \in \dot V}H^1(u_v \rw W_v,Z \rw T)\\
\bar Y[V_{\ol{F}},\dot V]_{0,+,\tx{tor}}\ar[u]^{\iota_{\dot V}}\ar[r]^{(l_v)_v}&\bigoplus\limits_{v \in \dot V} \bar Y_{+_v,\tx{tor}}\ar[r]^{\sum}\ar[u]^{(\iota_v)_v}&\frac{\bar Y}{IY}[\tx{tor}]
}\]
\end{cor}

The reader might wonder if there is a cohomology group that fits above the bottom right term in the above diagram. There is indeed such a cohomology group and it is an enlargement of $H^1(\Gamma,T(\A_{\ol{F}})/T(\ol{F}))$. However, since we will not need it for the applications we have in mind, we will not discuss its construction and properties.

\begin{cor} \label{cor:loctriv} Let $[Z \rw G] \in \mc{A}$ with $G$ connected reductive and let $x \in H^1(P_{\dot V} \rw \mc{E}_{\dot V},Z \rw G)$. Then $loc_v(x)$ is the trivial element of $H^1(u_v \rw W_v,Z \rw G)$ for almost all $v \in \dot V$.
\end{cor}
\begin{proof}
This follows immediately from Lemma \ref{lem:cohapprox1}, Corollary \ref{cor:tnexact}, and the functoriality of localization.
\end{proof}

The remainder of this subsection is occupied by the proofs of Theorem \ref{thm:tn+} and Corollary \ref{cor:tnexact}. We will use the notation $\bar Y[S_{i,E_i},\dot S_i]_0^{N_{E_i/F}}$ for the intersection of $\bar Y[S_{i,E_i},\dot S_i]_0$ with $\bar Y[S_{i,E_i}]_0^{N_{E_i/F}}$.

\begin{lem} \label{lem:normtor} We have the equality
\[ \frac{\bar Y[S_{i,E_i},\dot S_i]_0^{N_{E_i/F}}}{I_{E_i/F}Y[S_{i,E_i}]_0}=\frac{\bar Y[S_{i,E_i},\dot S_i]_0}{IY[S_{i,E_i}]_0}[\tx{tor}]. \]
\end{lem}
\begin{proof}
Choose an integer $N$ that is a multiple of the exponent of $\bar Y/Y$. Given any $x \in \bar Y[S_{i,E_i},\dot S_i]_0^{N_{E_i/F}}$ we have $Nx \in Y[S_{i,E_i}]_0^{N_{E_i/F}}$. The Tate-Nakayama isomorphism gives the identification
\[ \frac{Y[S_{i,E_i}]_0^{N_{E_i/F}}}{I_{E_i/F}Y[S_{i,E_i}]_0} \rw H^1(\Gamma_{E_i/F},T(O_{E_i})). \]
The group $H^1(\Gamma_{E_i/F},T(O_{E_i}))$ is finite by Lemma \ref{lem:inf1} and \cite[Theorem 8.3.20]{NSW08} and this implies the inclusion $\subset$ of the equality we are proving.

For the converse inclusion, let $x \in \bar Y[S_{i,E_i},\dot S_i]_0$ be such that for some $N \in \Z_{>0}$ we have $Nx \in I_{E_i/F}[S_{i,E_i}]_0$. Then $N_{E_i/F}(Nx)=0$. Since $\bar Y[S_{i,E_i},\dot S_i]_0$ is torsion-free this implies $N_{E_i/F}(x)=0$ and the claimed equality is proved.
\end{proof}

\begin{lem} \label{lem:sup} Every element of $\bar Y[S_{i,E_i},\dot S_i]_0/I_{E_i/F}Y[S_{i,E_i}]_0$ has a representative supported on $\dot S_i$.
\end{lem}
\begin{proof}
The argument is the same as in the proof of part 3 of Lemma \ref{lem:mesn_h2}. Indeed, let $y=\sum_{w \in S_{i,E_i}} y_w[w] \in \bar Y[S_{i,E_i},\dot S_i]_0$. For any $w \in S_{i,E_i} \sm \dot S_i$ we have $y_w \in Y$. If $y_w \neq 0$, then choose $\sigma \in \Gamma_{E_i/F}$ with $\sigma w \in \dot S_i$ and choose $\dot v_0 \in \dot S_i$ with $\sigma \dot v_0=\dot v_0$. Then
\[ y' := y - (y_w[w]-y_w[\dot v_0]) +\sigma(y_w[w]-y_w[\dot v_0]) \in \bar Y[S_{i,E_i},\dot S_i]_0 \]
represents the same class as $y$ modulo $I_{E_i/F}Y[S_{i,E_i}]_0$, but now $y'_w=0$. Performing this procedure finitely many times yields $y''$ that is supported on $\dot S_i$.
\end{proof}

\begin{proof}[Proof of Corollary \ref{cor:tnexact}]
According to the second diagram in Theorem \ref{thm:tn+} the localization $loc_v(x)$ for a given $x \in H^1(P_{\dot V} \rw \mc{E}_{\dot V},Z \rw T)$ is trivial for all but finitely many $v \in \dot V$ and thus the image of the total localization map $(loc_v)_v$ is indeed contained in the direct sum of the local cohomology groups. The commutativity of the above square is also immediate from Theorem \ref{thm:tn+}.

We now prove the exactness of the bottom row. The fact that the composition of the two arrows is zero follows from Lemma \ref{lem:sup}, because given an element $y =\sum y_w[w] \in \bar Y[S_{i,E_i},\dot S_i]_0^{N_{E_i/F}}$ that is supported on $\dot S_i$, the image of its class under $(l_v)_v$ is represented by the collection $(y_v)_v$.

Conversely, let $(y_v)_{v \in \dot S_i}$ be a collection with $y_v \in \bar Y^{N_{E_i/F,v}}$ and with $\sum_v y_v \in I_{E_i/F}Y$. We can write this element of $I_{E_i/F}Y$ as $(\sigma_1\lambda_1-\lambda_1)+\dots+(\sigma_k\lambda_k-\lambda_k)$ for some $\lambda_j \in Y$ and $\sigma_j \in \Gamma_{E_i/F}$. For each $j$ let $v_j \in \dot S_i$ be such that $\sigma_j v_j=v_j$ and replace $y_{v_j}$ by $y_{v_j}+\lambda_j-\sigma_j\lambda_j$. The new collection $(y_v)_{v \in \dot S_i}$ represents the same element of $\bigoplus_{v \in \dot S_i} \frac{\bar Y^{N_{E_i/F,v}}}{I_vY}$ as the old collection, but now $\sum_v y_v=0$, which makes it evident that $(y_v)_v$ is in the image of $(l_v)_v$.
\end{proof}

The proof of Theorem \ref{thm:tn+} will proceed in multiple steps. For each index $i$, let $\mc{T}_i$ be the full subcategory of $\mc{T}$ consisting of those objects $[Z \rw T]$ for which $T$ splits over $E_i$ and $\exp(Z) \in \N_{S_i}$. We will define a cohomology functor $H^1(P_{E_i,\dot S_i} \rw \mc{E}_{E_i,\dot S_i},Z \rw T)$ from $\mc{T}_i$ to the category of finite abelian groups and an isomorphism between this functor and the (restriction to $\mc{T}_i$) of the functor $\frac{\bar Y[S_{i,E_i},\dot S_i]_0}{I_{E_i/F}Y[S_{i,E_i}]_0}[\tx{tor}]$. We will then construct an inflation functor $H^1(P_{E_i,\dot S_i} \rw \mc{E}_{E_i,\dot S_i},Z \rw T) \rw H^1(P_{E_{i+1},\dot S_{i+1}} \rw \mc{E}_{E_{i+1},\dot S_{i+1}},Z \rw T)$ and show that it is compatible with the map $!$. We will argue that the direct limit of $H^1(P_{E_i,\dot S_i} \rw \mc{E}_{E_i,\dot S_i},Z \rw T)$ is equal to $H^1(P_{\dot V} \rw \mc{E}_{\dot V},Z \rw T)$. This will establish the existence of the isomorphism $\iota_{\dot V}$. We will then prove its uniqueness and its compatibility with localization.

We now begin with the execution of this plan. In order to lighten the notation, we will for a while forget about the fixed tower $E_i$ of Galois extensions and work with an arbitrary finite Galois extension $E/F$ and a pair $(S,\dot S_E)$ satisfying Conditions \ref{cnds:pes}. Let $\mc{T}_{E,S}$ be the full subcategory consisting of those objects $[Z \rw T]$ for which $T$ splits over $E$ and $\exp(Z) \in \N_S$. We define a functor
\[ \mc{T}_{E,S} \rw \tx{FinAbGrp},\quad [Z\rw T] \mapsto H^1(P_{E,\dot S_E} \rw \mc{E}_{E,\dot S_E},Z \rw T) \]
as follows. We take any extension
\[ 1 \rw P_{E,\dot S_E}(O_S) \rw \mc{E}_{E,\dot S_E} \rw \Gamma_S \rw 1 \]
corresponding to the distinguished class $\xi_{E,\dot S_E} \in H^2(\Gamma_S,P_{E,\dot S_E}(O_S))$ discussed in Subsection \ref{sub:pes}. The group $\mc{E}_{E,\dot S_E}$ is profinite and acts on $T(O_S)$ via its map to $\Gamma_S$. Thus we have the group $H^1(\mc{E}_{E,\dot S_E},T(O_S))$ of cohomology classes of continuous 1-cocycles of $\mc{E}_{E,\dot S_E}$ valued in the discrete group $T(O_S)$. We define $H^1(P_{E,\dot S_E} \rw \mc{E}_{E,\dot S_E},Z \rw T)$ to be the subgroup of $H^1(\mc{E}_{E,\dot S_E},T(O_S))$ comprised of those classes whose restriction to $P_{E,\dot S_E}(O_S)$, which is a well-defined $\Gamma_S$-equivariant continuous homomorphism $P_{E,\dot S_E}(O_S) \rw T(O_S)$, factors through the inclusion $Z(O_S) \rw T(O_S)$ and is an algebraic homomorphism. A bit more precisely, we can describe such a cohomology class as a pair $(v,h)$ consisting of $h \in H^1(\mc{E}_{E,\dot S_E},T(O_S))$ and $v \in \tx{Hom}(X^*(Z),M_{E,\dot S_E})^\Gamma$ such that the homomorphism $h|_{P_{E,\dot S_E}} : P_{E,\dot S_E}(O_S) \rw T(O_S)$ is determined by the composition of $v$ with the projection $X^*(T) \rw X^*(Z)$. The datum $v$ is however determined by $h$ and thus need not be kept track of.

We have the inflation-restriction sequence
\begin{equation} \label{eq:infres}
\begin{aligned}
1&\rw H^1(\Gamma_S,T(O_S))\rw H^1(P_{E,\dot S_E} \rw \mc{E}_{E,\dot S_E},Z \rw T)\rw \tx{Hom}(P_{E,\dot S_E},Z)^\Gamma \rw\\
&\rw H^2(\Gamma_S,T(O_S))
\end{aligned}\end{equation}
in which the last map is the composition of $\Theta^P_{E,\dot S_E,N}$ with the natural map $H^2(\Gamma_S,Z(O_S)) \rw H^2(\Gamma_S,T(O_S))$. The argument for this is the same as for \cite[Lemma 3.3]{KalRI}.

At the moment we do not know that the cohomology group $H^1(P_{E,\dot S_E} \rw \mc{E}_{E,\dot S_E},Z \rw T)$ is independent of the choice of extension $\mc{E}_{E,\dot S_E}$. This is not a-priori clear, because the analog of Proposition \ref{pro:locrig} fails here. We will show that nonetheless the cohomology group is independent of the particular extension, but this will be a consequence of the Tate-Nakayama-type isomorphism and certain group-theoretic properties of $H^1(P_{E,\dot S_E} \rw \mc{E}_{E,\dot S_E},Z \rw T)$ that are implied by it.

In order to construct the Tate-Nakayama-type isomorphism we first fix a specific realization of the extension of $\Gamma_S$ by $P_{E,\dot S_E}$ corresponding to the class $\xi_{E,\dot S_E}$. Let $\alpha_3(E,S) \in Z^2(\Gamma_{E/F},\tx{Hom}(\Z[S_E]_0,O_{E,S}^\times))$ represent the Tate-class discussed in Subsection \ref{sub:tatetori}. As in Subsection \ref{sub:h2z} we fix a co-final sequence $N_i \in \N_S$ and a system of maps $k_i : \tx{Maps}(S_E,O_S^\times)/O_S^\times \rw \tx{Maps}(S_E,O_S^\times)/O_S^\times$ satisfying $k_i(x)^{N_i}=x$ and $k_{i+1}(x)^{N_{i+1}/N_i}=k_i(x)$. Under the map $\Psi_{E,S,N_i} : \tx{End}(M_{E,\dot S_E,N_i})^\Gamma \to \hat Z^{-1}(\Gamma_{E/F},\tx{Maps}(S_E,M_{E,\dot S_E,N_i}^\vee)_0)$ of Lemma \ref{lem:mesn_h2} the endomorphism $\tx{id}$ maps to the $(-1)$-cocycle  $\alpha_i$ that sends $w \in S_E$ and $f \in M_{E,\dot S_E,N_i}$ to $f(1,w) \in \frac{1}{N_i}\Z/\Z$.

The class $\xi_{E,\dot S_E,N_i} \in H^2(\Gamma_S,P_{E,\dot S_E,N_i}(O_S))$ was defined as $\Theta^P_{E,\dot S_E,N_i}(\tx{id})$ and is thus represented by the 2-cocycle
\[ \dot \xi_{E,\dot S_E,N_i} = dk_i\alpha_3(E,S) \sqcup_{E/F} \alpha_i. \]
As in Subsection \ref{sub:h2z}, we are using here the unbalanced cup product of \cite[\S4.3]{KalRI} and the pairing \eqref{eq:pair2} provided by $\Phi_{E,S,N_i}$ of Fact \ref{fct:pair2}. This Fact also implies that the projection map $P_{E,\dot S_E,N_{i+1}} \rw P_{E,\dot S_E,N_i}$ maps the 2-cocycle $\dot\xi_{E,\dot S_E,N_{i+1}}$ to the 2-cocycle $\dot\xi_{E,\dot S_E,N_i}$. Let $\dot\xi_{E,\dot S_E} \in Z^2(\Gamma_S,P_{E,\dot S_E}(O_S))$ be the 2-cocycle determined by the inverse system $(\dot\xi_{E,\dot S_E,N_i})$ and let
\[ \mc{\dot E}_{E,\dot S_E} = P_{E,\dot S_E}(O_S) \boxtimes_{\dot\xi_{E,\dot S_E}}\Gamma_S. \]
This is our explicit realization of the extension $\mc{E}_{E,\dot S_E}$. Pushing out this extension along the projection $P_{E,\dot S_E} \rw P_{E,\dot S_E,N_i}$ produces the explicit extension $\mc{\dot E}_{E,\dot S_E,N_i} = P_{E,\dot S_E,N_i}(O_S)\boxtimes_{\dot\xi_{E,\dot S_E,N_i}} \Gamma_S$, so that conversely the explicit extension $\mc{\dot E}_{E,\dot S_E}$ is the inverse limit of the explicit extensions $\mc{\dot E}_{E,\dot S_E,N_i}$. We then define $H^1(P_{E,\dot S_E} \rw \mc{\dot E}_{E,\dot S_E},Z \rw T)$ as described above and see that it is equal to the direct limit of $H^1(P_{E,\dot S_E,N_i} \rw \mc{\dot E}_{E,\dot S_E,N_i},Z \rw T)$.

We will now formulate the Tate-Nakayama-type isomorphism for these cohomology groups. Fix $[Z \rw T] \in \mc{T}_{E,S}$ and let $\bar T=T/Z$. Write $Y=X_*(T)$ and $\bar Y=X_*(\bar T)$. Then we have the exact sequence
\[ 0 \rw Y \rw \bar Y \rw A^\vee \rw 0, \]
which upon applying $\otimes_\Z \Z[S_E]_0$ becomes
\[ 0 \rw Y[S_E]_0 \rw \bar Y[S_E]_0 \rw A^\vee[S_E]_0 \rw 0. \]
We have $Y[S_E]_0 = \tx{Maps}(S_E,Y)_0$ and we have the pairing
\[ \tx{Maps}(S_E,O_S^\times)/O_S^\times \otimes \tx{Maps}(S_E,Y)_0 \rw T(O_S) \]
defined by
\begin{equation} \label{eq:pair3} f \otimes g \mapsto \prod_{w \in S_E} f(w)^{g(w)}. \end{equation}
For $i>>0$ we have $\exp(Z)|N_i$. If $f \in \tx{Maps}(S_E,\mu_{N_i})/\mu_{N_i}$ and $g \in \bar Y[S_E]_0$, then $N_i\cdot g \in Y[S_E]_0$ and the image of $f \otimes N_ig$ under \eqref{eq:pair3} belongs to $Z(O_S)$ and equals the image of $f \otimes [g]$ under \eqref{eq:pair2}, where $[g] \in A^\vee[S_E]_0$ is the image of $g$.

We have the subgroup $A^\vee[\dot S_E]_0 \subset A^\vee[S_E]_0$. It is not $\Gamma_{E/F}$-equivariant, but we will make use of the slightly abusive notation $A^\vee[\dot S_E]_0^{N_{E/F}} = A^\vee[S_E]_0^{N_{E/F}} \cap A^\vee[\dot S_E]_0$ to save space. This abelian group is in bijection with $\tx{Hom}(P_{E,\dot S_E},Z)^\Gamma$ via the map $\Psi_{E,S}$ of Lemma \ref{lem:mesn_h2}. Let $\bar Y[S_E,\dot S_E]_0$ be the preimage in $\bar Y[S_E]_0$ of $A^\vee[\dot S_E]_0$. We set $\bar Y[S_E,\dot S_E]_0^{N_{E/F}}=\bar Y[S_E]_0^{N_{E/F}} \cap \bar Y[S_E,\dot S_E]_0$.

\begin{pro} \label{pro:iota}
\begin{enumerate}
	\item Given $\bar\Lambda \in \bar Y[S_E,\dot S_E]_0^{N_{E/F}}$ and $i>>0$ so that $\exp(Z)|N_i$, the assignment
	\[ z_{\bar\Lambda,i} : \mc{\dot E}_{E,\dot S_E,N_i} \rw T(O_S),\quad x \boxtimes \sigma \mapsto \Psi_{E,S,N_i}^{-1}([\bar\Lambda])(x) \cdot (k_i\alpha_3(E,S) \sqcup_{E/F} N_i\bar\Lambda) \]
	belongs to $Z^1(P_{E,\dot S_E,N_i} \rw \mc{\dot E}_{E,\dot S_E,N_i},Z \rw T)$. Here we have used the unbalanced cup product with respect to the pairing \eqref{eq:pair3}.
	\item The composition of $z_{\bar\Lambda,i}$ with the projection $\mc{E}_{E,\dot S_E,N_{i+1}} \rw \mc{E}_{E,\dot S_E,N_i}$ is equal to $z_{\bar\Lambda,i+1}$. Thus we obtain a well-defined $z_{\bar\Lambda} \in Z^1(P_{E,\dot S_E} \rw \mc{\dot E}_{E,\dot S_E}, Z \rw T)$.
	\item The assignment $\bar\Lambda \mapsto z_{\bar\Lambda}$ establishes an isomorphism
	\[ \dot\iota_{E,\dot S_E} : \frac{\bar Y[S_E,\dot S_E]_0^{N_{E/F}}}{I_{E/F}Y[S_E]_0} \rw H^1(P_{E,\dot S_E} \rw \mc{\dot E}_{E,\dot S_E},Z \rw T) \]
	that is functorial in $[Z \rw T] \in \mc{T}_{E,S}$ and fits into the diagram
	\[ \xymatrix{
	1\ar[d]&1\ar[d]\\
	\hat H^{-1}(\Gamma_{E/F},Y[S_E]_0)\ar[r]^{''TN''}\ar[d]&H^1(\Gamma_S,T(O_S))\ar[d]\\
	\frac{\bar Y[S_E,\dot S_E]_0^{N_{E/F}}}{I_{E/F}Y[S_E]_0}\ar[r]^-{\dot\iota_{E,\dot S_E}}\ar[d]&H^1(P_{E,\dot S_E} \rw \mc{\dot E}_{E,\dot S_E},Z \rw T)\ar[d]\\
	A^\vee[\dot S_E]^{N_{E/F}}\ar[d]\ar[r]^{\Psi_{E,S}^{-1}}&\tx{Hom}(P_{E,\dot S_E},Z)^\Gamma\ar[d]\\
	\hat H^0(\Gamma_{E/F},Y[S_E]_0)\ar[r]^{''-TN''}&H^2(\Gamma_S,T(O_S))
	}\]
\end{enumerate}
\end{pro}

\begin{proof}
For the first point it is enough to show that $z_{\bar\Lambda,i} \in Z^1(\mc{\dot E}_{E,\dot S_E,N_i},T(O_S))$, because the defining formula makes it obvious that the restriction of $z_{\bar\Lambda,i}$ to $P_{E,\dot S_E,N_i}$ takes values in $Z(O_S)$. A direct computation shows that the differential is given by
\[ dz_{\bar\Lambda,i}(x\boxtimes\sigma,y\boxtimes\tau) = \Psi_{E,S,N_i}^{-1}([\bar\Lambda])(\dot\xi_{E,\dot S_E,N_i}(\sigma,\tau))^{-1} \cdot d(k_i\alpha_3(E,S) \sqcup_{E/F} N_i\bar\Lambda)(\sigma,\tau). \]
By \cite[Fact 4.3]{KalRI} we have $d(k_i\alpha_3(E,S) \sqcup_{E/F} N_i\bar\Lambda) = dk_i\alpha_3(E,S) \sqcup_{E/F} N_i\bar\Lambda$. This 2-cocycle takes values in $Z(O_S)$ and is equal to $\Psi_{E,S,N_i}^{-1}([\bar\Lambda])(\dot\xi_{E,\dot S_E,N_i})$.

The second point follows from the compatibility of the maps $\Psi_{E,S,N_i}$ with the projections in the system $P_{E,\dot S_E,N_i}$, which is part 2 of Lemma \ref{lem:mesn_h2}, as well as from the fact that $k_{i+1}(x)^{N_{i+1}/N_i}=k_i(x)$.

We now come to the third point. It is clear from the definition that $\bar\Lambda \mapsto z_{\bar\Lambda}$ is a functorial homomorphism $\bar Y[S_E,\dot S_E]_0^{N_{E/F}} \rw H^1(P_{E,\dot S_E} \rw \mc{\dot E}_{E,\dot S_E},Z \rw T)$. If $\bar\Lambda \in Y[S_E]_0^{N_{E/F}}$, then $[\bar\Lambda] = 0$ and $k_i\alpha_3(E,S) \otimes N_i\bar\Lambda = \alpha_3(E,S) \otimes \bar\Lambda$. Thus $z_{\bar\Lambda,i} = \alpha_3(E,S) \cup \bar\Lambda$. This implies that $\bar\Lambda \mapsto z_{\Bar\Lambda}$ kills $I_{E/F}Y[S_E]_0$ and moreover that the top square in the above diagram commutes. The commutativity of the second square is immediate from the formula defining $z_{\bar\Lambda,i}$. The commutativity of the third square follows from the commutativity of Diagram \eqref{eq:tnzdiag}. The five-lemma now implies that $\dot\iota_{E,\dot S_E}$ is an isomorphism.
\end{proof}

Now let $\mc{E}_{E,\dot S_E}$ again be any extension corresponding to the class $\xi_{E,\dot S_E}$. We want to argue that the group $H^1(P_{E,\dot S_E} \rw \mc{E}_{E,\dot S_E},Z \rw T)$ is independent of the choice of $\mc{E}_{E,\dot S_E}$. There is an isomorphism of extensions $\mc{E}_{E,\dot S_E} \rw \mc{\dot E}_{E,\dot S_E}$. This isomorphism is not unique -- the set of $P_{E,\dot S_E}(O_S)$-conjugacy classes of such isomorphisms is a torsor under the finite abelian group $H^1(\Gamma_S,P_{E,\dot S_E}(O_S))$. Choosing any such isomorphism of extensions we obtain an isomorphism of cohomology groups
\[ H^1(P_{E,\dot S_E} \rw \mc{\dot E}_{E,\dot S_E},Z \rw T) \rw H^1(P_{E,\dot S_E} \rw \mc{E}_{E,\dot S_E},Z \rw T) \]
that is functorial in $[Z \rw T] \in \mc{T}_{E,S}$ and is moreover compatible with the inclusion of $H^1(\Gamma_S,T(O_S))$ into both groups as well as with the maps from both groups to $\tx{Hom}(P_{E,\dot S_E},Z)^\Gamma$. We will argue that there is at most one isomorphism with these properties. For this, the following consequence of Proposition \ref{pro:iota} will be crucial.

\begin{lem} \label{lem:nsdiv}
	The group $\varinjlim_Z H^1(P_{E,\dot S_E}\rw \mc{E}_{E,\dot S_E},Z \rw T)$, where $Z$ runs over all finite subgroups of $T$ defined over $F$ and satisfying $\exp(Z) \in \N_S$, is $\N_S$-divisible.
\end{lem}
\begin{proof}
Choosing any isomorphism of extensions $\mc{E}_{E,\dot S_E} \rw \mc{\dot E}_{E,\dot S_E}$ and applying Proposition \ref{pro:iota} we see that the claim is equivalent to the $\N_S$-divisibility of
\[ \varinjlim_Z\frac{\bar Y[S_E,\dot S_E]_0^{N_{E/F}}}{I_{E/F}Y[S_E]_0}. \]
The colimit can be taken with respect to any co-final sequence $Z_i$ of finite subgroups of $T$ with $\exp(Z_i) \in \N_S$. Such a sequence is obtained by setting $\bar X_i = N_iX$. Then $\bar Y_i = \tx{Hom}_\Z(X_i,\Z) = \frac{1}{N_i}Y$. The above direct limit is thus
\[ \frac{(S_\Q^{-1}Y)[S_E,\dot S_E]_0^{N_{E/F}}}{I_{E/F}Y[S_E]_0}, \]
where $S_\Q^{-1}Y = Y \otimes_\Z \Z[S_\Q^{-1}]$ is the localization of $Y$ as $S_\Q$. According to Lemma \ref{lem:sup} the map
\[ (S_\Q^{-1}Y)[\dot S_E]_0^{N_{E/F}} \rw \frac{(S_\Q^{-1}Y)[S_E,\dot S_E]_0^{N_{E/F}}}{I_{E/F}Y[S_E]_0} \]
is surjective. Since $S_\Q^{-1}Y$ is $\N_S$-divisible and torsion-free, $(S_\Q^{-1}Y)[\dot S_E]_0^{N_{E/F}}$ is also $\N_S$-divisible and the proof is complete.
\end{proof}

\begin{lem} \label{lem:unigen}
Fix a set of places $S$ of $\Q$, finite or infinite. For $i=1,2$ let $\Delta_1^i$ be a functor assigning to each torus $T$ defined over $F$ and split over $E$ an abelian group $\Delta_1^i(T)$. Let $\Delta_3^i$ be a functor assigning to each finite multiplicative group $Z$ defined over $F$ and split over $E$, with $\exp(Z) \in \N_S$, an abelian group $\Delta_3^i(Z)$. Let $\Delta_2^i$ be a functor assigning to each $[Z \rw T] \in \mc{T}_{E,S}$ an abelian group $\Delta_2^i(Z \rw T)$. Assume that for $Z \subset Z'$ the map $\Delta_2^i(Z \rw T) \rw \Delta_2^i(Z' \rw T)$ is injective. Assume further that we have a functorial in $[Z \rw T]$ exact sequence
\[ 0 \rw \Delta_1^i(T) \rw \Delta_2^i(Z \rw T) \rw \Delta_3^i(Z). \]
Finally, assume that for each $T$ the abelian group $\Delta_1^2(T)$ admits a homomorphism to a direct sum of finite abelian groups with finite kernel, that $\varinjlim_Z \Delta_2^1(Z \rw T)$ is $\N_S$-divisible, and that $\Delta_3^1(Z)$ is $\N_S$-torsion.

If $\Delta_1^1(T) \rw \Delta_1^2(T)$ and $\Delta_3^1(Z) \rw \Delta_3^2(Z)$ are fixed functorial homomorphisms, there exists at most one functorial homomorphism $\Delta_2^1(Z \rw T) \rw \Delta_2^2(Z \rw T)$ fitting into the commutative diagram
\[ \xymatrix@C=1.5pc{
	0\ar[r]&\Delta_1^1(T)\ar[r]\ar[d]&\Delta_2^1(Z \rw T)\ar[r]\ar[d]&\Delta_3^1(Z)\ar[d]\\
	0\ar[r]&\Delta_1^2(T)\ar[r]&\Delta_2^2(Z \rw T)\ar[r]&\Delta_3^2(Z)
}\]
\end{lem}
\begin{proof}
Let $f_{[Z \rw T]}^{(j)} : \Delta_2^1(Z \rw T) \rw \Delta_2^2([Z \rw T])$ for $j=1,2$ be two functorial homomorphisms fitting into the above commutative diagram. Fix a torus $T$ defined over $F$ and split over $E$ and let $Z_k$ be an exhaustive sequence of finite subgroups of $T$ defined over $F$ and split over $E$ with $\exp(Z_k) \in \N_S$. For $i=1,2$ we write $\Delta_2^i(T)=\varinjlim_k \Delta_2^i(Z_k \rw T)$ and $\Delta_3^i(T)=\varinjlim_k \Delta_3^i(Z_k)$. We obtain from $f^{(j)}_{[Z \rw T]}$ the homomorphisms $f^{(j)}_T : \Delta^1_2(T) \rw \Delta^2_2(T)$. Due to the injectivity of $\Delta_2^i(Z \rw T) \rw \Delta_2^i(T)$, it is enough to prove $f^{(1)}_T=f^{(2)}_T$.

By assumption $\Delta_2^1(T)$ is $\N_S$-divisible and $\Delta_3^1(T)$ is $\N_S$-torsion. The exactness of $\varinjlim$ gives us the commutative diagram with exact rows
\[ \xymatrix@C=1.5pc{
	1\ar[r]&\Delta_1^1(T)\ar[r]\ar[d]&\Delta_2^1(T)\ar[r]\ar@<+1ex>[d]^{f^{(2)}_{T}}\ar@<-1ex>[d]_{f^{(1)}_{T}}&\Delta_3^1(T)\ar[d]\\
	0\ar[r]&\Delta_1^2(T)\ar[r]&\Delta_2^2(T)\ar[r]&\Delta_3^2(T)
}\]
Let $Q \subset \Delta_3^1(T)$ be the image of $\Delta_2^1(T)$. Then $Q$ is both $\N_S$-divisible and $\N_S$-torsion. The difference $\delta_T=f^{(2)}_T-f^{(1)}_T \in \tx{Hom}(\Delta_2^1(T),\Delta_2^2(T))$ is a homomorphism $\delta_T : Q \rw \Delta_1^2(T)$. By assumption we have the exact sequence
\[ 0 \rw F_0 \rw \Delta_1^2(T) \rw \bigoplus_{l \in L} F_l\]
for some finite abelian groups $F_0$ and $F_l, l \in L$.
The projection of $\delta_T(Q) \subset \Delta_1^2(T)$ to any factor $F_l$ is a finite $\N_S$-divisible and $\N_S$-torsion group, hence trivial. Thus $\delta_T(Q)$ is a subgroup of $F_0$, but then it must be trivial by the same argument.
\end{proof}

As a first application of this general statement, we obtain.

\begin{cor} \label{cor:unique}
	The group $H^1(P_{E,\dot S_E} \rw \mc{E}_{E,\dot S_E},Z \rw T)$ is independent of the choice of extension $\mc{E}_{E,\dot S_E}$ up to a unique isomorphism. It comes equipped with a canonical functorial isomorphism $\iota_{E,\dot S_E}$ to the group
	\[ \frac{\bar Y[S_E,\dot S_E]_0^{N_{E/F}}}{I_{E/F}Y[S_E]_0} \]
	that fits into the commutative diagram of Proposition \ref{pro:iota}.
\end{cor}

\begin{proof}
First we apply Lemma \ref{lem:unigen} with $\Delta_2^1=H^1(P_{E,\dot S_E} \rw \mc{\dot E}_{E,\dot S_E})$ and $\Delta_2^2=H^1(P_{E,\dot S_E} \rw \mc{E}_{E,\dot S_E})$. We know that any isomorphism of extensions $\mc{E}_{E,\dot S_E} \rw \mc{\dot E}_{E,\dot S_E}$ provides an isomorphism $\Delta_2^1 \rw \Delta_2^2$ as in the statement of this lemma. The assumptions of the Lemma are verified by Lemma \ref{lem:nsdiv} and the finiteness of $H^1(\Gamma_S,T(O_S))$. The lemma then guarantees that the isomorphism is unique, i.e. independent of the choice of isomorphism of extensions $\mc{E}_{E,\dot S_E} \rw \mc{\dot E}_{E,\dot S_E}$.

Next recall that in order to construct the explicit extension $\mc{\dot E}_{E,\dot S_E}$ and the isomorphism $\dot\iota_{E,\dot S_E}$ of Proposition \ref{pro:iota} we had to choose the cocycle $\alpha_3(E,S)$ representing the Tate-class as well as the sequences $N_i$ and $k_i$. Say that we now made different choices for these and obtained an explicit extension $\mc{\ddot E}_{E,\dot S_E}$ and isomorphism $\ddot\iota_{E,\dot S_E}$. In order to complete the proof we must show that the composition $\dot\iota_{E,\dot S_E}^{-1}\circ\ddot\iota_{E,\dot S_E}$ coincides with the canonical isomorphism $H^1(P_{E,\dot S_E} \rw \mc{\dot E}_{E,\dot S_E}) \rw H^1(P_{E,\dot S_E} \rw \mc{\ddot E}_{E,\dot S_E})$ established in the above paragraph. This however is another application of Lemma \ref{lem:unigen}.
\end{proof}

We will now study how the group $H^1(P_{E,\dot S_E} \rw \mc{E}_{E,\dot S_E},Z \rw T)$ changes when we enlarge $E$ and $S$. Let $K/F$ be a finite Galois extension containing $E$. Let $(S',\dot S'_K)$ be a pair satisfying Conditions \ref{cnds:pes} with respect to $K/F$. Assume further that $S \subset S'$ and $\dot S_E \subset (\dot S'_K)_E$. Consider any extension $\mc{E}_{K,\dot S'_K}$ corresponding to the class $\xi_{K,\dot S'_K} \in H^2(\Gamma_{S'},P_{K,\dot S'_K}(O_{S'}))$ as well as any extension $\mc{E}_{E,\dot S_E}$ corresponding to the class $\xi_{E,\dot S_E}$. Given $[Z \rw T] \in \mc{T}_{E,S}$ we construct an inflation map
\begin{equation} \label{eq:infmap}
\tx{Inf}: H^1(P_{E,\dot S_E} \rw \mc{E}_{E,\dot S_E},Z \rw T)\rw H^1(P_{K,\dot S'_K} \rw \mc{E}_{K,\dot S'_K},Z \rw T)  \end{equation}
as follows.

Take $\Box_1$ to be the pull-back of $\mc{E}_{E,\dot S_E}$ along the projection $\Gamma_{S'} \rw \Gamma_S$, then take $\Box_2$ to be the push-out of $\Box_1$ along the inclusion $P_{E,\dot S_E}(O_S) \rw P_{E,\dot S_E}(O_{S'})$, and finally choose a homomorphism of extensions $\mc{E}_{K,\dot S'_K} \rw \Box_2$. The latter exists by Lemma \ref{lem:mesn_es2}. This results in the commutative diagram
\begin{equation} \label{eq:diaginf} \xymatrix{
	1\ar[r]&P_{K,S'_K}(O_{S'})\ar[r]\ar[d]&\mc{E}_{K,\dot S'_K}\ar[r]\ar@{.>}[d]&\Gamma_{S'}\ar[r]\ar@{=}[d]&1\\
	1\ar[r]&P_{E,\dot S_E}(O_{S'})\ar[r]&\Box_2\ar[r]&\Gamma_{S'}\ar[r]\ar@{=}[d]&1\\
	1\ar[r]&P_{E,\dot S_E}(O_{S})\ar[u]\ar[r]&\Box_1\ar[r]\ar[d]\ar[u]&\Gamma_{S'}\ar[r]\ar[d]&1\\
	1\ar[r]&P_{E,\dot S_E}(O_{S})\ar[r]\ar@{=}[u]&\mc{E}_{E,\dot S_E}\ar[r]&\Gamma_{S}\ar[r]&1
}\end{equation}
where the dotted arrow is the one we chose. Given $z \in Z^1(P_{E,\dot S_E} \rw \mc{E}_{E,\dot S_E},Z \rw T)$, let $\phi \in \tx{Hom}(P_{E,\dot S_E},Z)^\Gamma$ be its image. We compose $z$ with the map $\Box_1 \rw \mc{E}_{E,\dot S_E}$ and with the inclusion $T(O_S) \rw T(O_{S'})$ to obtain a 1-cocycle $z_1 : \Box_1 \rw T(O_{S'})$. It's restriction to $P_{E,\dot S_E}(O_S)$ is still given by $\phi$. Now $\phi$ determines a homomorphism $P_{E,\dot S_E}(O_{S'}) \rw Z(O_{S'}) \rw T(O_{S'})$, which glues with $z_1$ to a 1-cocycle $z_2 : \Box_2 \rw T(O_{S'})$.  We then compose $z_2$ with the homomorphism $\mc{E}_{K,\dot S'_K} \rw \Box_2$ to obtain $z_K \in Z^1(P_{K,\dot S'_K} \rw \mc{E}_{K,\dot S'_K},Z \rw T)$. The assignment $z \mapsto z_K$ provides a functorial in $[Z \rw T] \in \mc{T}_{E,S}$ homomorphism
\[ Z^1(P_{E,\dot S_E} \rw \mc{E}_{E,\dot S_E},Z \rw T)\rw Z^1(P_{K,\dot S'_K} \rw \mc{E}_{K,\dot S'_K},Z \rw T) \]
which respects coboundaries. The corresponding homomorphism on the level of cohomology gives \eqref{eq:infmap}. We also have the homomorphism
\[ !: \frac{\bar Y[S_E,\dot S_E]_0^{N_{E/F}}}{I_{E/F}Y[S_E]_0} \rw \frac{\bar Y[S'_K,\dot S'_K]_0^{N_{K/F}}}{I_{K/F}Y[S'_K]_0}
 \]
discussed in Lemma \ref{lem:inftn}.

\begin{pro} \label{pro:inf} The inflation map \eqref{eq:infmap} is independent of the choice of dotted arrow in Diagram \eqref{eq:diaginf}. It is injective, functorial in $[Z \rw T] \in \mc{T}_{E,S}$, and fits into the following commutative diagrams
\[ \xymatrix{
H^1(P_{E,\dot S_E} \rw \mc{E}_{E,\dot S_E},Z \rw T)\ar[r]^{\tx{Inf}}&H^1(P_{K,\dot S'_K} \rw \mc{E}_{K,\dot S'_K},Z \rw T) \\
\frac{\bar Y[S_E,\dot S_E]_0^{N_{E/F}}}{I_{E/F}Y[S_E]_0}\ar[u]^{\iota_{E,\dot S_E}}\ar[r]^{!}&\frac{\bar Y[S'_K,\dot S'_K]_0^{N_{K/F}}}{I_{K/F}Y[S'_K]_0}\ar[u]_{\iota_{K,\dot S'_K}}
} \]
and
\[ \xymatrix{
1\ar[d]&1\ar[d]\\
H^1(\Gamma_S,T(O_S))\ar[r]^{\tx{Inf}}\ar[d]&H^1(\Gamma_{S'},T(O_{S'}))\ar[d]\\
H^1(P_{E,\dot S_E} \rw \mc{E}_{E,\dot S_E},Z \rw T)\ar[d]\ar[r]^{\tx{Inf}}&H^1(P_{K,\dot S'_K} \rw \mc{E}_{K,\dot S'_K},Z \rw T)\ar[d]\\
\tx{Hom}(P_{E,\dot S_E},Z)^\Gamma\ar[r]^{}&\tx{Hom}(P_{K,\dot S'_K},Z)^\Gamma
}\]
\end{pro}
\begin{proof}
The commutativity of the second of the two diagrams is obvious from the construction of the inflation map. That this map is injective follows from the five-lemma, Lemma \ref{lem:inf2}, and the surjectivity of $P_{K,\dot S'_K} \rw P_{E,\dot S_E}$.

We will now apply Lemma \ref{lem:unigen} to show that \eqref{eq:infmap} does not depend on the choice of dotted arrow in Diagram \eqref{eq:diaginf}, and also that the first diagram in the statement of the proposition commutes. For this we take $\Delta_2^1$ to be $H^1(P_{E,\dot S_E} \rw \mc{E}_{E,\dot S_E})$ and take $\Delta_2^2$ to be  the restriction to $\mc{T}_{E,S}$ of $H^1(P_{K,\dot S'_K} \rw \mc{E}_{K,\dot S'_K})$. We take as $f_1$ the map \eqref{eq:infmap} constructed using one choice of dotted arrow, as $f_2$ the map \eqref{eq:infmap} constructed using another choice of dotted arrow, and as $f_3$ the composition $\iota_{K,\dot S'_K} \circ ! \circ \iota_{E,\dot S_E}^{-1}$. Then Lemma \ref{lem:unigen} implies that $f_1=f_2=f_3$.
\end{proof}

We now recall the exhaustive tower $E_i$ of finite Galois extensions of $F$ and the exhaustive tower $S_i$ of finite full subsets of $V_F$. Each $S_i$ was paired with $\dot S_i \subset S_{i,E_i}$ such that $(S_i,\dot S_i)$ satisfies Conditions \ref{cnds:pes} with respect to $E_i/F$ and moreover $\dot S_i \subset \dot S_{i+1,E_i}$. We will now show that
\begin{equation} \label{eq:h1lim}
H^1(P_{\dot V} \rw \mc{E}_{\dot V},Z \rw T)=\varinjlim_i H^1(P_{E,\dot S_E} \rw \mc{E}_{E,\dot S_E},Z \rw T),
\end{equation}
the colimit being taken with respect to \eqref{eq:infmap}. In order to do this, we first define inflation maps
\begin{equation} \label{eq:infmap1} \tx{Inf} : H^1(P_{E_i,\dot S_i} \rw \mc{E}_{E_i,\dot S_i},Z \rw T) \rw H^1(P_{\dot V} \rw \mc{E}_{\dot V},Z \rw T) \end{equation}
functorial in $\mc{T}_{E_i,S_i}$ in a manner analogous to the definition of \eqref{eq:infmap}. Namely, we consider the diagram
\begin{equation} \label{eq:diaginf1} \xymatrix{
	1\ar[r]&P_{\dot V}(\ol{F})\ar[r]\ar[d]&\mc{E}_{\dot V}\ar[r]\ar@{.>}[d]&\Gamma\ar[r]\ar@{=}[d]&1\\
	1\ar[r]&P_{E_i,\dot S_i}(\ol{F})\ar[r]&\Box_2\ar[r]&\Gamma\ar[r]\ar@{=}[d]&1\\
	1\ar[r]&P_{E_i,\dot S_i}(O_{S_i})\ar[u]\ar[r]&\Box_1\ar[r]\ar[d]\ar[u]&\Gamma\ar[r]\ar[d]&1\\
	1\ar[r]&P_{E_i,\dot S_i}(O_{S_i})\ar[r]\ar@{=}[u]&\mc{E}_{E_i,\dot S_i}\ar[r]&\Gamma_{S_i}\ar[r]&1
}\end{equation}
analogous to \eqref{eq:diaginf} and use the same procedure to produce the inflation map \eqref{eq:infmap1}.

\begin{pro} \label{pro:hl} The maps \eqref{eq:infmap1} splice together to an isomorphism
\[ \varinjlim_i H^1(P_{E_i,\dot S_i} \rw \mc{E}_{E_i,\dot S_i},Z \rw T) \rw H^1(P_{\dot V} \rw \mc{E}_{\dot V},Z \rw T) \]
of functors $\mc{T} \rw \tx{AbGrp}$ that is independent of any choices.
\end{pro}

\begin{proof}
\ul{Step 1:} We argue that \eqref{eq:infmap1} is injective and independent of the choice of dotted arrow in Diagram \eqref{eq:diaginf1}. For injectivity, note if the class of $z$ maps to zero, then already $z|_{P_{\dot V}}=0$ due to the surjectivity of $P_{\dot V} \rw P_{E,\dot S_i}$. In other words $z$ is inflated from $Z^1(\Gamma_{S_i},T(O_{S_i}))$ and the image in $Z^1(P_{\dot V} \rw \mc{E}_{\dot V},Z \rw T)$ under the inflation map just defined is inflated from $Z^1(\Gamma,T(\ol{F}))$. But the inflation map $H^1(\Gamma_{S_i},T(O_{S_i})) \rw  H^1(\Gamma,T(\ol{F}))$ is injective according to Lemma \ref{lem:inf2}. The independence of \eqref{eq:infmap1} from the choice of dotted arrow follows from Lemma \ref{lem:unigen} applied to $\Delta_2^1=H^1(P_{E,\dot S_E} \rw \mc{E}_{E,\dot S_E})$ and $\Delta_2^2=H^1(P_{\dot V} \rw \mc{E}_{\dot V})$.

\ul{Step 2:} For $j>i$ the diagram
\[ \xymatrix{
	H^1(P_{E_i,\dot S_i} \rw \mc{E}_{E_i,\dot S_i},Z \rw T)\ar[rd]\ar[dd]&\\
	&H^1(P_{\dot V} \rw \mc{E}_{\dot V},Z \rw T)\\
	H^1(P_{E_j,\dot S_j} \rw \mc{E}_{E_j,\dot S_j},Z \rw T)\ar[ru]&
}\]
commutes. Here the vertical arrow is the inflation map \eqref{eq:infmap}, while the diagonal arrows are \eqref{eq:infmap1}. This follows immediately from Lemma \ref{lem:unigen}.

\ul{Step 3:} Combining Steps 1,2 we obtain the homomorphism stated in the proposition and we already know that it is functorial, injective, and independent of choices. We will now argue that it is surjective.

For this, consider the following diagram for an arbitrary index $i$
\[ \xymatrix@C=1pc{
	1 \ar[r]&{H^1(\Gamma_{S_i},T(O_{S_i}))} \ar[r]\ar[d]& H^1(P_{E_i,\dot S_i} \rw \mc{E}_{E_i,\dot S_i},Z \rw T) \ar[r]\ar[d]& \tx{Hom}(P_{E_i,\dot S_i},Z)^\Gamma \ar[r]\ar[d]& H^2(\Gamma_{S_i},T(O_{S_i}))\ar[d]\\
	1 \ar[r]&{H^1(\Gamma,T(\ol{F}))} \ar[r]& H^1(P_{\dot V} \rw \mc{E}_{\dot V},Z \rw T) \ar[r]& \tx{Hom}(P_{\dot V},Z)^\Gamma \ar[r]& H^2(\Gamma,T(\ol{F}))\\
}\]
Taking the direct limit over all $i$ preserves the exactness of the top row, the first and fourth vertical maps stay injective by Lemma \ref{lem:inf2} and become surjective, and so also does the third vertical map. The five-lemma implies that the second vertical map is an isomorphism.
\end{proof}

The proof of Theorem \ref{thm:tn+} is almost complete. We have established the existence of the functorial isomorphism $\iota_{\dot V}$ and the commutativity of the first diagram. The uniqueness of $\iota_{\dot V}$ follows immediately from Lemma \ref{lem:unigen} applied to $\Delta_2^1=\bar Y[V_{\ol{F}},\dot V]_{0,+,\tx{tor}}$ and $\Delta_2^2=H^1(P_{\dot V} \rw \mc{E}_{\dot V})$. In this case the set $S$ in the Lemma is the set of all places of $\Q$. The commutativity of the second diagram in the statement of the Theorem is another immediate application of Lemma \ref{lem:unigen}. This time we take $\Delta_2^1=H^1(P_{\dot V} \rw \mc{E}_{\dot V})$ and $\Delta_2^2=H^1(u_v \rw W_v)$ and compare the two functorial homomorphisms $\tx{loc}_v$ and $\iota_v\circ l_v\circ\iota_{\dot V}^{-1}$. The theorem is now proved.

\subsection{Duality for connected reductive groups} \label{sub:tn+g}

In this subsection we are going to extend the duality results of the previous subsection to the case of connected reductive groups. Let $\mc{R} \subset \mc{A}$ be the full subcategory consisting of those $[Z \rw G]$ for which $G$ is connected and reductive. In \cite[\S3.4]{KalRI} we defined in the setting of a local field $F$ a quotient $H^1_\tx{ab}(u \rw W,Z \rw G)$ of the set $H^1(u \rw W,Z \rw G)$. The same definition works in the context of a global field $F$. Namely, on the set $H^1(P_{\dot V} \rw \mc{E}_{\dot V},Z \rw G)$ we impose the following equivalence relation. Let $z_1,z_2 \in Z^1(P_{\dot V} \rw \mc{E}_{\dot V},Z \rw G)$. Let $G^1$ be the twist of $G$ by the image of $z_1$ in $Z^1(\Gamma,G_\tx{ad})$. Tautologically we have $z_2 \cdot z_1^{-1} \in Z^1(P_{\dot V} \rw \mc{E}_{\dot V},Z \rw G^1)$ and we declare the classes of $z_1$ and $z_2$ equivalent if $z_2\cdot z_1^{-1}$ belongs to the image of $Z^1(\Gamma,G^1_\tx{sc})$. We denote by $H^1_\tx{ab}(P_{\dot V} \rw \mc{E}_{\dot V},Z \rw G)$ the set of equivalence classes under this equivalence relation. In this way we obtain a functor
\[ \mc{R} \rw \tx{Sets},\quad [Z \rw G] \mapsto H^1_\tx{ab}(P_{\dot V} \rw \mc{E}_{\dot V},Z \rw G) \]
and a functorial surjective map $H^1(P_{\dot V} \rw \mc{E}_{\dot V}) \rw H^1_\tx{ab}(P_{\dot V} \rw \mc{E}_{\dot V})$. The localization map \eqref{eq:locmap2} descends to a map $loc_v : H^1_\tx{ab}(P_{\dot V} \rw \mc{E}_{\dot V}) \rw H^1_\tx{ab}(u_v \rw W_v)$.

Next we extend the functor $\bar Y[V_{\ol{F}},\dot V]_{0,+,\tx{tor}}$ from $\mc{T}$ to $\mc{R}$ by following the constructions of \cite[\S4.1]{KalRI}. Let $T_1,T_2 \subset G$ be two maximal tori defined over $F$. Write $\bar Y_i=X_*(T_i/Z)$ and $Q_i^\vee=X_*(T_{i,\tx{sc}})$. According to \cite[Lemma 4.2]{KalRI} any $g \in G(\ol{F})$ with $\tx{Ad}(g)T_1=T_2$ provides the same (and hence $\Gamma$-equivariant) isomorphism $\tx{Ad}(g) : \bar Y_1/Q_1^\vee \rw \bar Y_2/Q_2^\vee$. Let $E/F$ be a finite Galois extension splitting $T_1$ and $T_2$ and $(S_E,\dot S_E)$ a pair satisfying Conditions \ref{cnds:pes}. Writing $\bar Y_1/Q_1^\vee[S_E,\dot S_E]_0^{N_{E/F}}$ for the subgroup of $Y_1/Q_1^\vee[S_E]_0^{N_{E/F}}$ consisting of functions whose value at $w \in S_E \sm \dot S_E$ lies in $Y_1/Q_1^\vee$ we obtain from $\tx{Ad}(g)$ an isomorphism
\begin{equation} \label{eq:ytti} \frac{\bar Y_1/Q_1^\vee[S_E,\dot S_E]_0^{N_{E/F}}}{I_{E/F}(Y_1/Q_1^\vee[S_E]_0)} \rw \frac{\bar Y_2/Q_2^\vee[S_E,\dot S_E]_0^{N_{E/F}}}{I_{E/F}(Y_2/Q_2^\vee[S_E]_0)}. \end{equation}
We let the value of the functor $\bar Y[V_{\ol{F}},\dot V]_{0,+,\tx{tor}}$ at $[Z \rw G]$ be given by the colimit of
\[ \frac{\bar Y/Q^\vee[S_{i,E_i},\dot S_i]_0^{N_{E_i/F}}}{I_{E_i/F}(Y/Q^\vee[S_{i,E_i}]_0)}, \]
where $i$ runs over the set of natural numbers, with the transition map $i \mapsto i+1$ given by the map $!$ of Lemma \ref{lem:inftn}, and where $T \subset G$ runs over all maximal tori of $G$, with the transition maps between two maximal tori given by the isomorphism \eqref{eq:ytti}.

The map $l_v : \bar Y[V_{\ol{F}},\dot V]_{0,+,\tx{tor}} \rw \bar Y_{+_v,\tx{tor}}$ extends to a map of functors $\mc{R} \rw \tx{Sets}$ by the same formula.

There is a related but simpler functor, which we shall call $\bar Y_{+,\tx{tor}}$. Its definition mimics that of the local functor $\bar Y_{+_v,\tx{tor}}$. Its value at $[Z \rw G]$ is defined as the colimit of
\[ \frac{\bar Y/Q^\vee}{I(Y/Q^\vee)}[\tx{tor}] \]
over all maximal tori $T \subset G$, where the transition maps are again induced by $\tx{Ad}(g)$.

\begin{thm} \label{thm:tn+g} The isomorphism $\iota_{\dot V}$ of Theorem \ref{thm:tn+} extends to an isomorphism
\[ \iota_{\dot V}: \bar Y[V_{\ol{F}},\dot V]_{0,+,\tx{tor}} \rw H^1_\tx{ab}(P_{\dot V} \rw \mc{E}_{\dot V}) \]
of functors $\mc{R} \rw \tx{Sets}$. We have the commutative diagram of sets with exact bottom row
\[ \xymatrix{
H^1_\tx{ab}(P_{\dot V} \rw \mc{E}_{\dot V},Z \rw G)\ar[r]^-{(loc_v)_v}&\coprod\limits_{v \in \dot V}H^1_\tx{ab}(u_v \rw W_v,Z \rw G)\\
\bar Y[V_{\ol{F}},\dot V]_{0,+,\tx{tor}}(Z \rw G)\ar[u]^{\iota_{\dot V}}\ar[r]^{(l_v)_v}&\bigoplus\limits_{v \in \dot V} \bar Y_{+_v,\tx{tor}}\ar[r]^-{\sum}\ar[u]^{(\iota_v)_v}(Z \rw G)&\bar Y_{+,\tx{tor}}(Z \rw G)
}\]
\end{thm}
We have used $\coprod$ to denote the subset of the direct product of pointed sets consisting of those elements almost all of whose coordinates are trivial. The bottom row is in fact an exact sequence of abelian groups.

\begin{cor} \label{cor:tn+gexact} The image of
\[ H^1(P_{\dot V} \rw \mc{E}_{\dot V},Z \rw G)\stackrel{(loc_v)_v}{\lrw} \coprod\limits_{v \in \dot V}H^1(u_v \rw W_v,Z \rw G) \]
consists precisely of those elements which map trivially under
\[ \begin{aligned} \coprod\limits_{v \in \dot V}H^1(u_v \rw W_v,Z \rw G)&\rw \coprod\limits_{v \in \dot V}H^1_\tx{ab}(u_v \rw W_v,Z \rw G)\\
&\rw \bigoplus\limits_{v \in \dot V} \bar Y_{+_v,\tx{tor}}(Z \rw G) \rw \bar Y_{+,\tx{tor}}(Z \rw G)
\end{aligned}\]
\end{cor}
\begin{proof}
It is clear that the image of the first map is contained in the kernel of the second. Conversely, let $(x_v) \in \coprod_{v \in \dot V}H^1(u_v \rw W_v,Z \rw G)$ belong to the kernel of the second map. Denote by $(\bar x_v) \in \coprod_{v \in \dot V} H^1_\tx{ab}(u_v \rw W_v,Z \rw G)$ its image. By Theorem \ref{thm:tn+g} there exists $\bar x \in H^1_\tx{ab}(P_{\dot V} \rw \mc{E}_{\dot V},Z \rw G)$ mapping to $(\bar x_v)_v$. Let $x^1 \in H^1(P_{\dot V} \rw \mc{E}_{\dot V},Z \rw G)$ be an arbitrary lift of $x$ and let $(x^1_v)_v \in \coprod_{v \in \dot V}H^1(u_v \rw W_v,Z \rw G)$ be its image. Then for each $v \in \dot V$, $\delta_v := x_v \cdot (x^1_v)^{-1}$ lifts to an element of $H^1(\Gamma_v,G^1_\tx{sc})$, where $G^1$ is the twist of $G$ by $x^1$. By work of Kneser, Harder, and Chernousov, \cite[\S6.1, Thm 6.6]{PR94} the localization map
\[ H^1(\Gamma,G^1_\tx{sc}) \rw \prod_{v \in \dot V} H^1(\Gamma_v,G^1_\tx{sc}) \]
is bijective, so there exists (a unique) $\delta \in H^1(
\Gamma,G^1_\tx{sc})$ whose localization is equal to $(\delta_v)_v$. Then $x := \delta \cdot x^1 \in H^1(P_{\dot V} \rw \mc{E}_{\dot V},Z \rw G)$ has the desired localizations.
\end{proof}

\begin{proof}[Proof of Theorem \ref{thm:tn+g}]
Let $T \subset G$ be a maximal torus. Consider the map
\[ \bar Y[V_{\ol{F}},\dot V]_{0,+,\tx{tor}}(Z \rw T) \rw H^1(P_{\dot V} \rw \mc{E}_{\dot V},Z \rw T) \rw H^1_\tx{ab}(P_{\dot V} \rw \mc{E}_{\dot V},Z \rw G). \]
We claim that the fibers of this map are torsors for the image of
\[ Y[V_{\ol{F}}]_{0,\Gamma,\tx{tor}}(T_\tx{sc}) \rw \bar Y[V_{\ol{F}},\dot V]_{0,+,\tx{tor}}(Z \rw T). \]
By the usual twisting argument it is enough to consider the fiber over the trivial element. Let $x \in H^1(P_{\dot V} \rw \mc{E}_{\dot V},Z \rw T)$ map to the trivial element in $H^1_\tx{ab}(P_{\dot V} \rw \mc{E}_{\dot V},Z \rw G)$. Then the image of $x$ in $H^1(P_{\dot V} \rw \mc{E}_{\dot V},Z \rw G)$ is equal to the image of a class from $H^1(\Gamma,G_\tx{sc})$. In particular, the restriction of $x$ to $P_{\dot V}$ is trivial, thus $x$ is inflated from $H^1(\Gamma,T)$. The crossed modules $T_\tx{sc} \rw T$ and $G_\tx{sc} \rw G$ are quasi-isomorphic, which implies that $x$ lifts to a class $x_\tx{sc} \in H^1(\Gamma,T_\tx{sc})$. This proves the claim.

The claim implies that we have an injective map
\[ \tx{cok}(Y[V_{\ol{F}}]_{0,\Gamma,\tx{tor}}(T_\tx{sc}) \rw \bar Y[V_{\ol{F}},\dot V]_{0,+,\tx{tor}}(Z \rw T)) \rw H^1_\tx{ab}(P_{\dot V} \rw \mc{E}_{\dot V},Z \rw G). \]
To understand the source of this map we consider the exact sequence
\begin{equation} \label{eq:tn+gseq} \frac{Q^\vee[S_E]_0^N}{IQ^\vee[S_E]_0}\rw \frac{\bar Y[S_E,\dot S_E]_0^N}{IY[S_E]_0} \rw \frac{\bar Y/Q^\vee[S_E,\dot S_E]_0^N}{I(Y/Q^\vee[S_E]_0)}\rw \frac{Q^\vee[S_E]_0^\Gamma}{N(Q^\vee[S_E]_0)}.\end{equation}
Here the subscript $N$ signifies the kernel of the norm map $N_{E/F}$, $I$ stands for the augmentation ideal in $\Z[\Gamma_{E/F}]$, and the last map is given by taking a representative $x \in \bar Y[S_E,\dot S_E]_0$ and mapping it to $N(x) \in Q^\vee[S_E]_0^\Gamma$. Taking the colimit as $E$ traverses the sequence $E_i$, the first term becomes $Y[V_{\ol{F}}]_{0,\Gamma,\tx{tor}}(T_\tx{sc})$, the second term becomes $\bar Y[V_{\ol{F}},\dot V]_{0,+,\tx{tor}}(Z \rw T)$, and the third term becomes a subset of $\bar Y[V_{\ol{F}},\dot V]_{0,+,\tx{tor}}(Z \rw G)$. Thus the above cokernel is equal to the image of $\bar Y[V_{\ol{F}},\dot V]_{0,+,\tx{tor}}(Z \rw T)$ in $\bar Y[V_{\ol{F}},\dot V]_{0,+,\tx{tor}}(Z \rw G)$ and we have obtained an injective map
\begin{equation} \label{eq:tn+ginj}
\xymatrix{
	\tx{im}(\bar Y[V_{\ol{F}},\dot V]_{0,+,\tx{tor}}(Z \rw T) \rw \bar Y[V_{\ol{F}},\dot V]_{0,+,\tx{tor}}(Z \rw G))\ar[d]\\
	H^1_\tx{ab}(P_{\dot V} \rw \mc{E}_{\dot V},Z \rw G)
}
\end{equation}

We will now show the following:
\begin{enumerate}
	\item Every two elements of $\bar Y[V_{\ol{F}},\dot V]_{0,+,\tx{tor}}(Z \rw G)$ belong to the image of $\bar Y[V_{\ol{F}},\dot V]_{0,+,\tx{tor}}(Z \rw T)$ for some $T \subset G$.
	\item If $x_i \in \bar Y[V_{\ol{F}},\dot V]_{0,+,\tx{tor}}(Z \rw T_i)$ for $i=1,2$ map to the same element in $\bar Y[V_{\ol{F}},\dot V]_{0,+,\tx{tor}}(Z \rw G)$, then $\iota_{\dot V}(x_i) \in H^1(P_{\dot V} \rw \mc{E}_{\dot V},Z \rw T_i)$ map to the same element in $H^1_\tx{ab}(P_{\dot V} \rw \mc{E}_{\dot V},Z \rw G)$
\end{enumerate}
These two points imply that the maps \eqref{eq:tn+ginj} splice together to an injective map
\[ \iota_{\dot V} : \bar Y[V_{\ol{F}},\dot V]_{0,+,\tx{tor}}(Z \rw G) \rw H^1_\tx{ab}(P_{\dot V} \rw \mc{E}_{\dot V},Z \rw G) \]
which is also surjective according to Lemma \ref{lem:cohapprox1}. By construction, this map is functorial for maps of the form $[Z \rw T] \rw [Z \rw G]$, where $T \subset G$ is any maximal torus. Its general functoriality follows from the functoriality of $\iota_{\dot V}$ on the category $\mc{T}$. The commutativity of the square in the diagram follows from the functoriality of $\iota_{\dot V}$ and Corollary \ref{cor:tnexact}. The exactness of the bottom row is proved by the same argument as the exactness of the bottom row in that corollary. The proof of the theorem will be complete once the above two points are shown.

To show the first point, let $x_1,x_2 \in \bar Y[V_{\ol{F}},\dot V]_{0,+,\tx{tor}}(Z \rw G)$. This set is defined as a double colimit, but in one direction the transition maps are isomorphisms, while in the other direction the system is indexed by natural numbers. Thus there exists an object in this system which gives rise to both $x_1$ and $x_2$. This object is of the form
\[ \frac{\bar Y/Q^\vee[S_{i,E_i},\dot S_i]_0^{N_{E_i/F}}}{I_{E_i/F}(Y/Q^\vee[S_{i,E_i}]_0)}, \]
for some $i$ and some maximal torus $T \subset G$ split by $E_i/F$.
Let $T'$ be a maximal torus of $G$ that is fundamental \cite[\S10]{Kot86} at each place $v \in S_i$ as well as at at least one non-archimedean place, in case $S_i$ only contains archimedean places. Such a torus exists by \cite[\S7.1,Cor. 3]{PR94}. Let $j>i$ be such that $E_j/F$ splits $T'$. We map $x_1$ and $x_2$ into the above displayed quotient, but with $i$ replaced by $j$ and with $T$ replaced by $T'$. To prove the first point it is enough to show that both $x_1$ and $x_2$ belong to the image of the middle map in the exact sequence \eqref{eq:tn+gseq}. For this we must show that their image under the third map is zero. By construction, the localization at each $v \in \dot S_j$ of this image is zero: For $v \in \dot S_j$ lying away from $\dot S_i$ the localization is zero because $x_1$ and $x_2$ are not supported on $v$, and for $v \in \dot S_j$ lying over $\dot S_i$ the localization is an element of $\hat H^0(\Gamma_{E_j/F,v},Q^\vee)=H^2(\Gamma_{E_j/F,v},T'_\tx{sc}(E_{j,v})) \subset H^2(\Gamma_v,T'_\tx{sc})$, which is trivial by \cite[Lemma 10.4]{Kot86}. Thus the images of $x_1$ and $x_2$ are locally everywhere trivial elements of $\hat H^0(\Gamma_{E_j/F},Q^\vee[S_{E_j}]_0)$. But this group is isomorphic to $H^2(\Gamma_{E_j/F},T'_\tx{sc}(O_{E_j,S_j}))$, which by Lemma \ref{lem:inf2} is a subgroup of $H^2(\Gamma,T')$. According to \cite[\S6.3,Prop. 6.12]{PR94} the latter has no non-trivial but everywhere trivial elements. This completes the proof of the first point.

To prove the second point, choose $j$ so that $E_j$ splits $T_i$, $\exp(Z) \in \N_{S_j}$, and $x_i$ comes from $\Lambda_i \in \bar Y_i[S_{j,E_j},\dot S_j]_0^{N_{E_j/F}}$, where $\bar Y_i=X_*(T_i/Z)$. We will use the explicit construction of $\dot\iota_{E_j,\dot S_j}$ of Proposition \ref{pro:iota}. For this we fix a Tate-cocycle $\alpha_3(E_j,S_j) \in Z^2(\Gamma_{E_j/F},\tx{Hom}(\Z[S_{j,E_j}]_0,O_{E_j,S_j}^\times))$, let $N \in \N_{S_j}$ be a multiple of $\exp(Z)$, and let $k : O_{S_j}^\times \rw O_{S_j}^\times$ be an $N$-th root function. This data leads to the explicit extension $\mc{\dot E}_{E_j,\dot S_j,N}$ and to the elements of $Z^1(P_{E_j,\dot S_j,N} \rw \mc{\dot E}_{E_j,\dot S_j,N}, Z \rw T_i)$ defined by
\[ z_{\bar\Lambda_i,N} : \mc{\dot E}_{E_j,S_j,N} \rw T_i(O_{S_j}),\quad x \boxtimes \sigma \mapsto \Psi^{-1}_{E_j,S_j,N}([\bar\Lambda_i])(x) \cdot (k\alpha_3(E_j,S)\sqcup_{E_j/F}N\bar\Lambda_i). \]
We inflate each $z_{\bar\Lambda_i,N}$ to an element $z_{\bar\Lambda_i}$ of $Z^1(P_{E_j,\dot S_j} \rw \mc{\dot E}_{E_j,\dot S_j}, Z \rw T_i)$ via the projection $\mc{\dot E}_{E_j,S_j} \rw \mc{\dot E}_{E_j,S_j,N}$. We then choose a diagram \eqref{eq:diaginf1} and obtain elements $z_{\bar\Lambda_i}$ of $Z^1(P_{\dot V} \rw \mc{E}_{\dot V}, Z \rw T_i)$ whose classes represent $\iota_{\dot V}(x_i)$. We compose $z_{\bar\Lambda_i}$ with the inclusion $T_i \rw G$ and want to show that $z_{\bar\Lambda_2} \cdot (z_{\bar\Lambda_1})^{-1} \in Z^1(P_{\dot V} \rw \mc{E}_{\dot V},Z \rw G^1)$ lifts to $Z^1(\Gamma,G^1_\tx{sc})$. Since we are using the same diagram \eqref{eq:diaginf1} for both $z_{\bar\Lambda_i}$, it will be enough to show this for the elements of $Z^1(P_{E_j,\dot S_j} \rw \Box_2,Z \rw T_i)$ whose composition with $\mc{E}_{\dot V} \rw \Box_2$ is equal to $z_{\bar\Lambda_i}$. Let us reuse the notation $z_{\bar\Lambda_i}$ for these elements. The explicit extension $\mc{\dot E}_{E_j,S_j}$ has the form $P_{E_j,\dot S_j}(O_{S_j}) \boxtimes \Gamma_{S_j}$ and so we have $\Box_2 = P_{E_j,S_j}(\ol{F}) \boxtimes \Gamma$. The elements $z_{\bar\Lambda_i}$ are given by the same explicit formula as $z_{\bar\Lambda_i,N}$. The images $[\bar\Lambda_i] \in \tx{Hom}(A,M_{E_j,S_j,N})^\Gamma=A^\vee[S_{j,E_j}]_0^{N_{E_j/F}}$ of $\Lambda_i$ are equal and we see
\[ [z_{\bar\Lambda_2} \cdot (z_{\bar\Lambda_1})^{-1}](x \boxtimes \sigma) = (k\alpha_3(E_j,S)\sqcup_{E_j/F}N\bar\Lambda_2)(\sigma)(k\alpha_3(E_j,S)\sqcup_{E_j/F}N\bar\Lambda_1)(\sigma)^{-1}. \]
Thus $z_{\bar\Lambda_2} \cdot (z_{\bar\Lambda_1})^{-1}$ is inflated from an element of $Z^1(\Gamma,G^1)$. We now show that this element comes from $Z^1(\Gamma,G^1_\tx{sc})$. By assumption there exists $g \in G(\ol{F})$ and $M \in Q_2^\vee[S_{j,E_j}]_0$ such that $\tx{Ad}(g)\bar\Lambda_1=\bar\Lambda_2+M$. We have the $\Gamma_{E_j/F}$-invariant decomposition $Y_i \otimes \Q = P^\vee_i \otimes \Q \oplus Q_i^\perp \otimes \Q$, where $P^\vee_i=X_*(T_{i,\tx{ad}})$ and $Q_i=X^*(T_{i,\tx{ad}})$. It is realized by the projection $Y_i \rw P^\vee_i$ given by interpreting elements of $Y_i$ as linear forms on $X^*(T_i)$ and restricting them to $Q_i$. We write $\bar\Lambda_i = p_i + r_i$ according to this decomposition. We then have $p_2=\tx{Ad}(g)p_1+M$ and $r_2=\tx{Ad}(g)r_1$. We enlarge $N$ if necessary so that $Np_i \in Q_i^\vee$ and $Nr_i \in Q_i^\perp$, where $Q_i^\vee=X_*(T_{i,\tx{sc}})$. The inclusions $Z(G)^\circ \rw T_i$ give the identifications $X_*(Z(G)^\circ) \rw Q_i^\perp$. It follows that the contributions of $r_1$ and $r_2$ to $z_{\bar\Lambda_2} \cdot (z_{\bar\Lambda_1})^{-1}$ cancel out and we are left with
\[ [z_{\bar\Lambda_2} \cdot (z_{\bar\Lambda_1})^{-1}](x \boxtimes \sigma) = \underbrace{(k\alpha_3(E_j,S)\sqcup_{E_j/F}Np_2)}_{c_2}(\sigma)\underbrace{(k\alpha_3(E_j,S)\sqcup_{E_j/F}Np_1)}_{c_1}(\sigma)^{-1}. \]
We have $c_i \in C^1(\Gamma,T_{i,\tx{sc}})$. The right-hand side is thus a 1-cochain of $\Gamma$ valued in $G_\tx{sc}$ and we want to argue that it is a 1-cocycle for the twisted group $G^1_\tx{sc}$. Its differential is given by
\[ [c_2(\sigma)c_1(\sigma)^{-1}][c_1(\sigma)\cdot{^\sigma c_2(\tau)}\cdot {^\sigma c_1(\tau)^{-1}}\cdot c_1(\sigma)^{-1}][c_2(\sigma\tau)c_1(\sigma\tau)^{-1}]^{-1}.\]
Opening brackets, using the commutativity of $T_{1,\tx{sc}}$ and rearranging, this is equal to
\[ c_2(\sigma)\cdot {^\sigma c_2(\tau)} \cdot dc_1(\sigma,\tau)^{-1} \cdot c_2(\sigma\tau)^{-1}. \]
The image of $c_1$ in $C^1(\Gamma,T_{1,\tx{ad}})$ is equal to the image of $z_{\bar\Lambda_1}$ there and is thus a 1-cocycle, hence $dc_1 \in Z^2(\Gamma,Z(G_\tx{sc}))$. This allows us to rearrange once more and obtain
\[ dc_2(\sigma,\tau) \cdot dc_1(\sigma,\tau)^{-1}. \]
Since $Np_i \in [Q_i^\vee]^{N_{E_j/F}}$ we conclude from \cite[Fact 4.3]{KalRI} that $dc_i = dk\alpha_3(E_j,S) \cup Np_i$. The inclusions $Z(G_\tx{sc}) \rw T_{i,\tx{sc}}$ provide the identifications $P_i^\vee/Q_i^\vee \rw \tx{Hom}(\mu_N,Z(G_\tx{sc}))$ under which the images of $p_1$ and $p_2$ coincide and hence $dc_1=dc_2$.
\end{proof}

\subsection{A ramification result} \label{sub:ram}

We review here a result \cite{TaibGRI} of Ta\"ibi that is essential for our applications. For this result is is important that we take $\xi$ to be the canonical class of Definition \ref{dfn:canclass} and let $\mc{E}_{\dot V}$ be the corresponding gerbe.

Let $G$ be a connected reductive group and $Z \subset G$ a finite central subgroup, both defined over $F$. For an element $z \in Z^1( P_{\dot V} \rw \mc{E}_{\dot V},Z \rw G)$ we have the collection of elements $\tx{loc}_v(z) \in Z^1(u_v \rw W_v,Z \rw G)$ for all $v \in \dot V$. Each element $\tx{loc}_v(z)$ is well-defined up to an element of $B^1(\Gamma_v,Z)$. Since $z|_{P_{\dot V}}$ factors through $P_{E_i,S_i,N_i}$ for some $i$, it follows that for almost all $v$ the restriction $\tx{loc}_v(z)|_{u_v}$ is trivial, so $\tx{loc}_v(z)$ is in fact inflated from $Z^1(\Gamma_v,G)$. At the same time, for almost all $v$ we have $Z(\ol{F_v})=Z(O_{F_v^u})$ and this is a $\Gamma_v/I_v$-module, hence $B^1(\Gamma_v,Z)=B^1(\Gamma_v/I_v,Z(O_{F_v^u}))$. Fix an integral model of $G$, i.e. an extension of $G$ to a reductive group scheme over $O_{F,S}$ for some finite set $S \subset V$. For almost all $v$, the subgroup $G(O_{F^u_v}) \subset G(F^u_v)$ does not depend on the choice of an integral model. We may thus ask the following question: Is it true that for almost all $v$ the 1-cocycle $\tx{loc}_v(z)$ is inflated from $Z^1(\Gamma_v/I_v,G(O_{F^u_v}))$? It was shown in \cite{TaibGRI} that the answer is indeed positive:

\begin{pro}[{\cite[Proposition 6.1.1]{TaibGRI}}] \label{pro:ramcoh} Let $z \in Z^1( P_{\dot V} \rw \mc{E}_{\dot V},Z \rw G)$. For all but finitely many $v \in \dot V$, the cocycle $\tx{loc}_v(z) \in Z^1(u_v\to W_v,Z \to G)/B^1(\Gamma_v,Z)$ is inflated from $Z^1(\Gamma_v/I_v,G(O_{F_v^u}))$.
\end{pro}

\section{Applications to the trace formula and automorphic multiplicities} \label{sec:mult}

In this section we are going to use the refined local endoscopic objects introduced in \cite{KalRI} to produce two essential
global endoscopic objects -- the adelic transfer factor used in the stabilization of the Arthur-Selberg trace formula and the conjectural pairing between an adelic $L$-packet and its global $S$-group. The first of these objects -- the adelic transfer factor -- is already defined, see e.g. \cite[\S6.3]{LS87} or \cite[\S7.3]{KS99}, where it is denoted by $\Delta_\A(\gamma,\delta)$. The point here is to show that this factor admits a decomposition as the product of the normalized local transfer factors introduced in \cite[\S5.3]{KalRI}. The second global object -- the pairing between an adelic $L$-packet and its global $S$-group -- has so far not been constructed for general non-quasi-split reductive groups. As discussed in \cite[\S12]{Kot84} and \cite[\S3.4]{BR94}, it is expected that this pairing is canonical and has a decomposition as a product of local pairings. We will construct this pairing here using the refined local pairings of \cite[\S5.4]{KalRI} and then show that the result is canonical.

In the last subsection, we will summarize the results of this section in the language of \cite{Art06}. We will show that the local conjecture of \cite{KalRI} implies the local conjecture of \cite{Art06}, and we will show that Proposition \ref{pro:tfprod} (which is unconditional) implies \cite[Hypothesis 9.5.1]{Art13}.

\subsection{Notation for this section} \label{sub:multnot}

Let $G$ be a connected reductive group over the number field $F$ with quasi-split inner form $G^*$ and let $\Psi$ be a $G$-conjugacy class of inner twists $\psi : G^* \rw G$. We write $G_\tx{der}$, $G_\tx{sc}$, and $G_\tx{ad}$, for the derived subgroup of $G$ and its simply connected cover and adjoint quotient, respectively. We use the superscript $*$ to denote the same objects relative to $G^*$. Let $\hat G^*$ be the Langlands dual group of $G^*$ and $^LG^*=\hat G^* \rtimes W_F$ the Weil-form of the $L$-group of $G^*$. We will write $\hat G^*_\tx{der}$ for the derived subgroup of $\hat G^*$ and $\hat G^*_\tx{sc}$ and $\hat G^*_\tx{ad}$ for the simply connected cover and adjoint quotient of $\hat G^*_\tx{der}$. We will write $Z_\tx{sc}$ for the center of $G^*_\tx{sc}$ and $\hat Z_\tx{sc}$ for the center of $\hat G^*_\tx{sc}$. We will write $Z_\tx{der}$ for the center of $G_\tx{der}^*$.

We set $\bar G^* = G^*/Z_\tx{der} = G^*_\tx{ad} \times Z(G^*)/Z_\tx{der}$ and $\bar G^*_\tx{sc} = G^*_\tx{sc}/Z_\tx{sc}=G^*_\tx{ad}$. Then we have $\hat{\bar G^*}=\hat G^*_\tx{sc} \times Z(\hat G^*)^\circ$ and $\hat{\bar G^*_\tx{sc}}=\hat G^*_\tx{sc}$. Recall the notation $Z(\hat{\bar G^*})^+$ for the elements of $Z(\hat{\bar G^*})$ which map to $Z(\hat G^*)^\Gamma$. We will use the analogous notation $Z(\hat{\bar G^*})^{+_v}$ for those elements that map to $Z(\hat G^*)^{\Gamma_v}$ when $v$ is a place of $\ol{F}$.

With this notation, we can give the following interpretation of Theorem \ref{thm:tn+}, Corollary \ref{cor:tnexact}, and Theorem \ref{thm:tn+g}. As was argued in \cite[Proposition 5.3]{KalRIBG} there is a functorial embedding $\bar Y_{+_v,\tx{tor}}(Z \rw G^*) \rw \pi_0(Z(\hat{\bar G^*})^{+_v})^*$. The same argument shows that there is a functorial isomorphism
\[ \bar Y_{+,\tx{tor}}(Z \rw G) \rw \pi_0(Z(\hat{\bar G^*})^+)^*.\]
Moreover, the map $\sum$ in the bottom row of the diagram in Theorem \ref{thm:tn+g} is dual to the diagonal embedding
\[ Z(\hat{\bar G^*})^+ \rw \prod_{v \in \dot V} Z(\hat{\bar G^*})^{+_v}. \]

Let $\xi$ be the canonical class of Definition \ref{dfn:canclass} and let $\mc{E}_{\dot V}$ be the corresponding gerbe.

\subsection{From global to refined local endoscopic data} \label{sub:endolg}

Let us first recall the notion of a global endoscopic datum for $G$ following \cite[\S1.2]{LS87}. It is a tuple $(H,\mc{H},s,\xi)$ consisting of a quasi-split connected reductive group $H$ defined over $F$, a split extension $\mc{H}$ of $W_F$ by $\hat H$, an element $s \in Z(\hat H)$, and an $L$-embedding $\xi : \mc{H} \rw {^LG^*}$ satisfying the following conditions:
\begin{enumerate}
\item The homomorphism $W_F \rw \tx{Out}(\hat H)=\tx{Out}(H)$ determined by the extension $\mc{H}$ coincides with the homomorphism $\Gamma \rw \tx{Out}(H)$ determined by the rational structure of $H$, in the sense that both homomorphisms factor through some finite quotient $W_F/W_E=\Gamma_F/\Gamma_E$ and are equal
\item $\xi$ induces an isomorphism of complex algebraic groups $\hat H \rw \tx{Cent}(t,\hat G^*)^\circ$, where $t=\xi(s)$.
\item The image $\bar s \in Z(\hat H)/Z(\hat G^*)$ of $s$ is fixed by $W_F$ and maps under the connecting homomorphism $H^0(W_F,Z(\hat H)/Z(\hat G^*)) \rw H^1(W_F,Z(\hat G^*))$ to a locally trivial element.
\end{enumerate}
In condition 3 we have used the fact that conditions 1 and 2 provides a $\Gamma$-equivariant embedding $Z(\hat G^*) \rw Z(\hat H)$. Recall that a locally trivial element of $H^1(W_F,Z(\hat G^*))$ is one whose image in $H^1(W_{F_v},Z(\hat G^*))$ is trivial for all $v \in V(\ol{F})$. Up to equivalence, the datum $(H,\mc{H},s,\xi)$ depends only on the image of $s$ in $\pi_0((Z(\hat H)/Z(\hat G^*))^\Gamma)$.

On the other hand, a refined endoscopic datum is a tuple $(H,\mc{H},\dot s,\xi)$ where $H$, $\mc{H}$, and $\xi$ are as above (but now over a local field $F$), and $\dot s$ is an element of $Z(\hat{\bar H})^+$ (whose image in $Z(\hat H)^\Gamma$ we denote by $s$) such that conditions 1 and 2 are satisfied. Condition 3 is unnecessary. We have used here the fact that conditions 1 and 2 provide an $F$-embedding $Z(G) \rw Z(H)$ which allows us to form the quotient $\bar H=H/Z_\tx{der}$ and hence obtain $\hat{\bar H}$. Recall that $Z(\hat{\bar H})^+$ is the preimage of $Z(\hat H)^\Gamma$ under the isogeny $\hat{\bar H} \rw \hat H$. Up to equivalence, the datum $(H,\mc{H},\dot s,\xi)$ depends only on the image of $\dot s$ in $\pi_0(Z(\hat{\bar H})^+)$. Note that, unlike in the global case, we are \emph{not} allowing to change $\dot s$ by elements of $Z(\hat G^*)$.

We will now show how a global endoscopic datum $\mf{e}=(H,\mc{H},s,\xi)$ gives rise to a collection, indexed by $\dot V$, or refined local endoscopic data. Using the embedding $Z(G) \rw Z(H)$ we form $\bar H=H/Z_\tx{der}$ and obtain from $\xi$ the embedding $\hat{\bar H} \rw \hat{\bar G^*}$ whose image is equal to $\tx{Cent}(t,\hat{\bar G^*})^\circ$. In order to simplify the notation, we will use the same letter for elements of $\hat H$ or $\hat{\bar H}$ and their image under $\xi$, so in particular we have $s=t$. Let $s_\tx{sc} \in \hat G^*_\tx{sc}$ be a preimage in $\hat G^*_\tx{sc}$ of the image $s_\tx{ad} \in \hat G^*_\tx{ad}$ of $s$. According to condition 3 above, given a place $v \in \dot V$ there exists $y_v \in Z(\hat G^*)$ such that $s_\tx{der} \cdot y_v \in Z(\hat H)^{\Gamma_v}$, where $s_\tx{der} \in \hat G^*_\tx{der}$ is the image of $s_\tx{sc}$. Write $y_v=y'_v \cdot y''_v$ with $y'_v \in Z(\hat G^*_\tx{der})$ and $y''_v \in Z(\hat G^*)^\circ$ and choose a lift $\dot y'_v \in \hat Z_\tx{sc}$. Then $\dot s_v := (s_\tx{sc}\cdot \dot y'_v,y''_v)$ belongs to $Z(\hat{\bar H})^{+_v}$ and $\mf{\dot e}_v =(H,\mc{H},\dot s_v,\xi)$ is a refined local endoscopic datum at the place $v$. We thus obtain the collection $(\mf{\dot e}_v)_{v\in \dot V}$ of refined local endoscopic data.

We emphasize that we use the same $s_\tx{sc}$ to form all $\mf{\dot e}_v$. The collection $(\mf{\dot e}_v)_v$ depends on the choice of $s_\tx{sc}$ and each member $\mf{\dot e}_v$ depends furthermore on the choices of $\dot y'_v$ and $y''_v$. We shall see that these choices do not influence the resulting global objects, which means that they depend only on the equivalence class of the global datum $\mf{e}$.

There are two special cases in which the construction of the collection $(\mf{\dot e}_v)$ simplifies. The first case is when $G^*_\tx{der}$ is simply connected. Then $Z(\hat G^*)$ is connected and we can take $y''_v=y_v$ and $\dot y'_v=1$, so that $\dot s_v=(s_\tx{sc},y_v)$. The second case is when $G^*$ satisfies the Hasse principle. Then the choices of $\dot y'_v$ and $y''_v$ can be made globally and this leads to a global element $\dot s$. Indeed, according to \cite[(4.2.2), Lemma 11.2.2]{Kot84} we have $\tx{ker}^1(W_F,Z(\hat G^*))=0$ and we can thus choose $y \in Z(\hat G^*)$ such that $s_\tx{der}\cdot y \in Z(\hat H)^\Gamma$. Write $y=y'\cdot y''$ with $y' \in Z(\hat G^*_\tx{der})$ and $y'' \in Z(\hat G^*)^\circ$ and choose a lift $\dot y' \in Z(\hat G^*_\tx{sc})$. Then $\dot s := (s_\tx{sc}\cdot \dot y',y'')$ belongs to $Z(\hat{\bar H})^+$. We let $\dot s_v := \dot s$ for each $v \in \dot V$.

\subsection{Coherent families of local rigid inner twists} \label{sub:cohfam}

We are given an equivalence class $\Psi$ of inner twists $G^* \rw G$ over $F$. At each place $v \in \dot V$ this leads to an equivalence class $\Psi_v : G^* \rw G$ of inner twists over $F_v$ and we would like to enrich it to an equivalence class $\dot\Psi_v : G^* \rw G$ of rigid inner twists such that the collection $(\dot\Psi_v)_{v \in \dot V}$ is in some sense coherent. We do this as follows. The class $\Psi$ provides an element of $H^1(\Gamma,G^*_\tx{ad}(\ol{F}))$. According to Lemma \ref{lem:cohsurj} this class has a lift to $H^1(P_{\dot V}\rw \mc{E}_{\dot V},Z_\tx{sc} \rw G^*_\tx{sc})$. This means that for every $\psi \in \Psi$ there is an element $z_\tx{sc} \in Z^1(P_{\dot V}\rw \mc{E}_{\dot V},Z \rw G^*_\tx{sc})$ with the property that $\psi^{-1}\sigma(\psi)=\tx{Ad}(z_\tx{ad}(\sigma))$, where $z_\tx{ad} \in Z^1(\Gamma,G^*_\tx{ad})$ is the composition of $z_\tx{sc}$ with the projection $G^*_\tx{sc} \rw G^*_\tx{ad}$. For each $v \in \dot V$, let $z_{\tx{sc},v} \in Z^1(u_v \rw W_v,Z \rw G^*_\tx{sc})$ be the localization of $z_\tx{sc}$ at $v$ as defined in Subsection \ref{sub:coh} and let $z_v \in Z^1(u_v \rw W_v,Z \rw G^*)$ be its image. Then $(\psi,z_v) : G^* \rw G$ is a rigid inner twist over $F_v$.

We have thus obtained a collection $(\psi,z_{v})_v$ of local rigid inner twists indexed by $v \in \dot V$. The collection depends on the choice of $z_\tx{sc}$, as well as on the choices of localizing diagrams \eqref{eq:diagloc2} (the dotted arrows in these diagrams needs to be chosen). As we discussed in Subsection \ref{sub:coh}, Proposition \ref{pro:locrig} implies that the choice of localizing diagrams is irrelevant, because changing it will change $z_{\tx{sc},v}$ by an element of $B^1(\Gamma_v,Z_\tx{sc})$ and thus no cohomology class constructed from $z_{\tx{sc},v}$ will detect this change. The potential influence of $z_\tx{sc}$ is stronger, as it could influence the local cohomology classes. We shall see however that the global objects obtained from these local cohomology classes remain unaffected and thus depend only on $\Psi$.

It is interesting to note the parallel of choices made in this subsection and in Subsection \ref{sub:endolg}. The choice of $z_\tx{sc}$ lifting $\Psi$ here corresponds to the choice of $s_\tx{sc}$ lifting $s_\tx{ad}$ there.

\subsection{A product formula for the adelic transfer factor} \label{sub:tfprod}

In this subsection we will show that the adelic transfer factor admits a decomposition as a product of normalized local transfer factors. We are given an equivalence class $\Psi$ of inner twists $G^* \rw G$, an endoscopic datum $\mf{e}=(H,\mc{H},s,\xi)$ for $G^*$, and a $z$-pair $\mf{z}=(H_1,\xi_{H_1})$ for $\mf{e}$. The adelic transfer factor is a function
\[ \Delta_\A : H_{1,G\tx{-sr}}(\A) \times G_\tx{sr}(\A) \rw \C, \]
which associates to a pair of semi-simple and strongly $G$-regular elements $\gamma_1 \in H_{1,G\tx{-sr}}(\A)$ and $\delta \in G_\tx{sr}(\A)$ a complex number $\Delta_\A(\gamma_1,\delta) \in \C$. This factor is defined in \cite[\S6.3]{LS87}, but see also \cite[\S7.3]{KS99}, where $z$-pairs are explicitly used. It is identically zero unless there exists a pair of related elements $\gamma_{1,0} \in H_{1,G\tx{-sr}}(F)$ and $\delta_0 \in G_\tx{sr}(F)$, which we now assume.

We let $(\mf{\dot e}_v)_{v\in \dot V}$ be the collection of refined local endoscopic data obtained from $\mf{e}$ as in Subsection \ref{sub:endolg}. We obtain in a straightforward way a collection $(\mf{z}_v)_{v\in \dot V}$ of local $z$-pairs from the global $z$-pair $\mf{z}$. Let $\mf{w}$ be a Whittaker datum for the quasi-split group $G^*$, and let $\mf{w}_v$ denote its localization at a place $v$. We furthermore let $(\psi,z_v)_{v\in \dot V}$ be the collection of local rigid inner twists obtained from $\Psi$ as in Subsection \ref{sub:cohfam}. Note here that $\psi \in \Psi$ was chosen during that construction. For each $v \in \dot V$ we have the normalized local transfer factor
\[ \Delta[\mf{w}_v,\mf{\dot e}_v,\mf{z}_v,\psi,z_v] : H_{1,G\tx{-sr}}(F_v) \times G_\tx{sr}(F_v) \rw \C \]
defined in \cite[\S5.3]{KalRI}.

\begin{pro} \label{pro:tfprod} For any $\gamma_1 \in H_{1,G\tx{-sr}}(\A)$ and $\delta \in G_\tx{sr}(\A)$ we have
\[ \Delta_\A(\gamma_1,\delta) = \prod_{v \in \dot V} \<z_{\tx{sc},v},\dot y'_v\>\Delta[\mf{w}_v,\mf{\dot e}_v,\mf{z}_v,\psi,z_v](\gamma_{1,v},\delta_v). \]
For almost all $v \in \dot V$, the corresponding factor in the product is equal to $1$. For all $v$, the factor in the product is independent of the choices of $\dot y'_v$ and $y''_v$ made in Subsection \ref{sub:endolg}.
\end{pro}

Before we give the proof, we want to make a remark on the statement. At first sight it may seem surprising that we had to include the factor $\<z_{\tx{sc},v},\dot y'_v\>$ in the formula. This is an explicitly given complex number and can of course be hidden by including it into the definition of the normalized transfer factor. This seems natural from the global point of view and is part of the reformulation of our results given in Subsection \ref{sub:arthur}. However, from the local point of view this appears less natural to us, which is why we have kept the formulation as it is. Note that when $G_\tx{der}$ is simply connected, which is often assumed in the literature when dealing with the passage from global to local endoscopy, the discussion at the end of Subsection \ref{sub:endolg} shows that we may choose $\dot y'_v=1$,  and this factor disappears. If $G$ satisfies the Hasse principle, then the same discussion shows that $\dot y'_v=\dot y'$ for a global element $\dot y' \in Z(\hat G^*_\tx{sc})$. Then the product over all $v$ of the terms $\<z_{\tx{sc},v},\dot y'_v\>$ is equal to $1$ by Corollary \ref{cor:tn+gexact}, so again these terms can be removed from the above formula.

\begin{proof}
For each $\mf{\dot e}_v$, let $\mf{e}_v$ be the corresponding usual (non-refined) local endoscopic datum. Then the collection $(\mf{e}_v)_v$ is associated to $\mf{e}$ as in \cite[\S6.2]{LS87}. It is shown in \cite[Corollary 6.4.B]{LS87} that if $\Delta^{(v)} : H_{1,G\tx{-sr}}(F_v) \times G_\tx{sr}(F_v) \rw \C$ is an absolute transfer factor normalized so that $\Delta^{(v)}(\gamma_{1,0},\delta_{0})=1$ for almost all $v$ and such that $\prod_{v \in \dot V}\Delta^{(v)}(\gamma_{1,0},\delta_0)=1$, then
\[ \Delta_\A(\gamma_1,\delta) = \prod_{v \in \dot V} \Delta^{(v)}(\gamma_{1,v},\delta_v). \]
We will show that our normalized transfer factors $\<z_{\tx{sc},v},\dot y'_v\>\Delta[\mf{w}_v,\mf{\dot e}_v,\mf{z}_v,\psi,z_v]$ satisfy these properties. First off, they are indeed absolute transfer factors, according to \cite[Proposition 5.6]{KalRI}. Now let $\gamma_0 \in H(F)$ be the image of $\gamma_{1,0}$ and let $T_0^H \subset H$ be the centralizer of $\gamma_0$. Choose an admissible embedding $T_0^H \rw G^*$ and let $\delta_0^*$ be the image of $\gamma_0$. Then $\delta^*_0$ and $\delta_0$ are stably conjugate, i.e. there exists $g \in G^*(\ol{F})$ such that $\psi(g\delta^*_0g^{-1})=\delta_0$. We recall \cite[(5.1)]{KalRI} that the factor $\Delta[\mf{w}_v,\mf{\dot e}_v,\mf{z}_v,\psi,z_v](\gamma_{1,0},\delta_0)$ is defined as the product
\[
\Delta[\mf{w}_v,\mf{e}_v,\mf{z}_v](\gamma_{1,0},\delta^*_0) \cdot \<\tx{inv}(\delta^*_0,(G,\psi,z_v,\delta_0)),\dot s_{v,\gamma_0,\delta_0^*})\>^{-1},
\]
where $\Delta[\mf{w}_v,\mf{e}_v,\mf{z}_v]$ is the Whittaker normalization of the transfer factor for $\mf{e}_v$ and the quasi-split group $G^*$. It is the product of terms $\epsilon_v$, $\Delta_I$, $\Delta_{II}$, $\Delta_2$, and $\Delta_{IV}$ (we have arranged that $\Delta_1=1$ by taking $\delta^*_0$ to be the image of $\gamma_0$ under the admissible embedding $T_0^H \rw G^*$ that we are using to construct the factors). Here $\epsilon_v$ is the local $\epsilon$-factor defined in \cite[\S5.3]{KS99} and the other terms are defined in \cite{LS87}. For almost all $v$ the local $\epsilon$-factor is equal to $1$ and the product over all $v$ gives the global $\epsilon$-factor of a virtual representation which is defined over $\Q$, and hence self-dual, so the result is equal to $1$. All the $\Delta_?$ terms are also equal to $1$ for almost all $v$ and their product over all $v$ is equal to $1$ according to \cite[Theorem 6.4.A]{LS87}. It will thus be enough to show that for almost all $v$ the term
\begin{equation} \label{eq:p1} \<z_{\tx{sc},v},\dot y'_v\>^{-1}\<\tx{inv}(\delta^*_0,(G,\psi,z_v,\delta_0)),\dot s_{v,\gamma_0,\delta_0^*})\> \end{equation}
is equal to $1$ and that the product over all $v$ of thee terms is equal to $1$. For this, recall that the class $\tx{inv}(\delta^*_0,(G,\psi,z_v,\delta_0)) \in H^1(u_v \rw W_v,Z_\tx{der} \rw T_0)$ is represented by the $1$-cocycle
\[ x_v : W_v \rw T_0(\ol{F_v}),\qquad w \mapsto g^{-1}z_v(w)\sigma_w(g) \]
where $\sigma_w \in \Gamma_v$ is the image of $w$. The class of $x_v$ does not depend on the choice of $g$ and it will be convenient to assume that $g$ is the image of $g_\tx{sc} \in G^*_\tx{sc}(\ol{F})$. Then $x_{v,\tx{sc}}(w)=g^{-1}_\tx{sc}z_{\tx{sc},v}(w)\sigma_w(g_\tx{sc})$ is a lift of $x_v$ to an element of $Z^1(u_v \rw W_v,Z_\tx{sc} \rw T_{0,\tx{sc}})$.

On the other hand, $\dot s_{v,\gamma_0,\delta^*_0} \in \pi_0(\hat{\bar T_0}^{+_v})$ is the image of $\dot s_v \in \pi_0(Z(\hat{\bar H})^{+_v})$ under the map $\hat\varphi_{\gamma_0,\delta_0^*} : Z(\hat{\bar H}) \rw \hat{\bar T_0^H} \rw \hat{\bar T_0}$ coming from the chosen admissible embedding. We have $\hat{\bar T_0} = [\hat T_0]_\tx{sc} \times Z(\hat G^*)^\circ$ and by construction $\dot s_{v,\gamma_0,\delta^*_0} = (\hat\varphi_{\gamma_0,\delta^*_0}(s_\tx{sc})\dot y_v',y''_v)$. The map $[Z_\tx{sc} \rw T_{0,\tx{sc}}] \rw [Z_\tx{der} \rw T_0]$ dualizes to the map $[\hat T_0]_\tx{sc} \times Z(\hat G)^\circ \rw [\hat T_0]_\tx{sc}$ given by projection onto the first factor. It follows that \eqref{eq:p1} is equal to
\begin{equation} \label{eq:p2} \<z_{\tx{sc},v},\dot y'_v\>^{-1}\<x_{v,\tx{sc}},\hat\varphi_{\gamma_0,\delta^*_0}(s_\tx{sc})\dot y_v'\>. \end{equation}
Now $\<x_{v,\tx{sc}},-\>$ is a character of $\pi_0([\hat{\ol{T_{0,\tx{sc}}}}]^{+_v})$ whose restriction to $Z(\hat{\bar G_\tx{sc}^*})^{+_v}$ is equal to $\<z_{\tx{sc},v},-\>$. This means that \eqref{eq:p2} becomes
\begin{equation} \label{eq:p3} \<x_{v,\tx{sc}},\hat\varphi_{\gamma_0,\delta^*_0}(s_\tx{sc})\>. \end{equation}

Consider the element of $Z^1(P_{\dot V} \rw \mc{E}_{\dot V},Z_\tx{sc} \rw T_{0,\tx{sc}})$ given by
\[ x_\tx{sc} : e \mapsto g_\tx{sc}^{-1}z_\tx{sc}(e)\sigma_e(g_\tx{sc}). \]
The 1-cocycles $loc_v(x_\tx{sc})$ and $x_{v,\tx{sc}}$ are equal, where $loc_v$ is the localization map \eqref{eq:locmap2} constructed via Diagram \eqref{eq:diagloc2}. According to Corollary \ref{cor:tnexact}, for almost all $v$ the cohomology class $loc_v([x_\tx{sc}])$ is trivial and hence the corresponding term \eqref{eq:p3} is equal to $1$. Moreover, if we restrict each character $\<loc_v([x_\tx{sc}]),-\>$ of $\pi_0([\hat{\ol{T_{0,\tx{sc}}}}]^{+_v})$ to the group $\pi_0([\hat{\ol{T_{0,\tx{sc}}}}]^{+})$ and take the product over all $v$, the result is the trivial character. But since the image of $s$ in $Z(\hat H)/Z(\hat G)$ is $\Gamma$-fixed, $\hat\varphi_{\gamma_0,\delta^*_0}(s_\tx{sc})$ belongs to $\pi_0([\hat{\ol{T_{0,\tx{sc}}}}]^{+})$ and the proof of the product formula is complete.

The statement of the independence of choices is immediate from equation \eqref{eq:p3}. Indeed, neither $\dot y'_v$ or $y''_v$ appear in this formula. Moreover, only the cohomology class of $x_\tx{sc}$ and its localizations are involved, but these are independent of the particular localization diagrams.
\end{proof}

The remaining choices that we made -- of the element $s_\tx{sc}$ in Subsection \ref{sub:endolg}, which we may multiply by any element of $\hat Z_\tx{sc}$; of the lift $z_\tx{sc}$ of $\Psi$, which we may multiply by any element of $H^1(\mc{E}_{\dot V},Z_\tx{sc})$; and of the global Whittaker datum $\mf{w}$, which we may conjugate by an element of $G^*_\tx{ad}(F)$ -- can influence the individual factors in Proposition \ref{pro:tfprod}. However, the proposition implies that they do not influence the product.

\subsection{The multiplicity formula for discrete automorphic representations} \label{sub:mult}
In this subsection we will recall the conjectural multiplicity formula for tempered discrete automorphic representations due to Kottwitz \cite[\S12]{Kot84} and then, assuming the validity of the local conjecture stated in \cite[\S5.4]{KalRI}, we will construct the global pairing that occurs in this formula.

Let $\varphi$ be a generic global Arthur parameter for $G^*$. This can either take the form of an $L$-homomorphism $\varphi : L_F \rw {^LG^*}$ with bounded image, where $L_F$ is the hypothetical Langlands group of $F$, or it can be a formal global parameter modelled on the cuspidal spectrum of $\tx{GL}_N$, as in \cite[\S1.4]{Art13}. We will assume here that we are dealing with an $L$-homomorphism, since the arguments needed in the case of formal parameters are independent of the problem we are discussing. At each place $v \in \dot V$, the parameter $\varphi$ has a localization, which is a parameter $\varphi_v : L_{F_v} \rw {^LG^*}$. The validity of the local conjecture ensures the existence of an $L$-packet $\Pi_{\varphi_v}$ of tempered representations of rigid inner twists of $G^*$ together with a bijection
\[ \iota_{\varphi_v,\mf{w}_v} : \Pi_{\varphi_v} \rw \tx{Irr}(S_{\varphi_v}^+). \]
The set $\Pi_{\varphi_v}$ consists of equivalence classes of tuples $(G'_v,\psi'_v,z'_v,\pi'_v)$, where $(\psi'_v,z'_v): G^* \rw G'_v$ is a rigid inner twist over $F_v$ and $\pi'_v$ is an irreducible tempered representation of $G'(F_v)$. The group $S_{\varphi_v}^+$ is the preimage in $\hat{\bar G^*}$ of the group $S_{\varphi_v} = \tx{Cent}(\varphi_v,\hat G^*)$. The bijection $\iota_{\varphi_v,\mf{w}_v}$ depends on a choice of local Whittaker datum $\mf{w}_v$. As in Subsection \ref{sub:tfprod}, we chose a global Whittaker datum $\mf{w}$ for $G^*$ and let $\mf{w}_v$ be its localization at each place $v$.

We are interested in the particular group $G$. We choose a family of local rigid inner twists $(\psi,z_v) : G^* \rw G$ indexed by $v \in \dot V$ as in Subsection \ref{sub:cohfam}. Then the subset of $\Pi_{\varphi_v}$ consisting of tuples $(G,\psi,z_v,\pi_v)$ is the $L$-packet of representations of $G(F_v)$ corresponding to $\varphi_v$.
We consider the adelic $L$-packet
\[ \Pi_\varphi = \{ \pi=\otimes'_v\pi_v| (G,\psi,z_v,\pi_v) \in \Pi_{\varphi_v}, \iota_{\varphi_v}((G,\psi,z_v,\pi_v)) = 1\tx{ for almost all }v \}.\]
\begin{lem}[Ta\"ibi] \label{lem:ramification}

The set $\Pi_\phi$ consists of irreducible admissible tempered representations of $G(\A)$.
\end{lem}
\begin{proof} The only non-obvious point is that $\pi_v$ is unramified for almost all $v$. In fact, this is a non-trivial statement that relies in Ta\"ibi's result Proposition \ref{pro:ramcoh}, and the work of Casselman-Shalika \cite{CS80}.

There exists a finite set $S$ of places of $F$ such that $G^*$ and $G$ have integral models over $O_{F,S}$, $\psi : G^* \to G$ is an isomorphism over $O_S$, and for all $v \notin S$ the Whittaker datum $\mf{w}_v$ is unramified, the local parameter $\varphi_v$ is unramified, and by Proposition \ref{pro:ramcoh} the 1-cocycle $z_v$ is inflated from $Z^1(\Gamma_{F_v^\tx{ur}/F_v},G^*(O_{F_v^{\tx{ur}}}))$. The triviality of $H^1(\Gamma_{F_v^\tx{ur}/F_v},G^*(O_{F_v^{\tx{ur}}}))$ implies the existence of $g \in G^*(O_{F_v^{\tx{ur}}})$ such that $f = \psi \circ \tx{Ad}(g^{-1})$ is an isomorphism $G^* \to G$ defined over $O_{F_v}$. Now $(g,f)$ is an isomorphism from the rigid inner twist $(\psi,z_v) : G^* \to G$ to the trivial rigid inner twist $(\tx{id},1) : G^* \to G^*$. It follows that $\iota_{\varphi_v}(G^*,\tx{id},1,\pi_v \circ f)=1$. In other words, the representation $\pi_v \circ f$ of $G^*(F)$ is $\mf{w}_v$-generic. According to \cite{CS80} $\pi_v \circ f$ is unramified with respect to the hyperspecial subgroup $G^*(O_{F_v})$ of $G^*(F_v)$. Since $f$ is an isomorphism defined over $O_{F_v}$, $\pi$ is unramified with respect to the hyperspecial subgroup $G(O_{F_v})$ of $G(F_v)$.
\end{proof}

It is expected that every tempered discrete automorphic representation of $G(\A)$ belongs to such a set $\Pi_\varphi$ for a suitable $\varphi$ that is discrete (elliptic in the language of \cite[\S10.3]{Kot84}). Moreover, Kottwitz has given a conjectural formula \cite[(12.3)]{Kot84} for the multiplicity in the discrete spectrum of $G$ of any $\pi \in \Pi_\varphi$. For this, consider the exact sequence
\[ 1 \rw Z(\hat G^*) \rw \hat G^* \rw \hat G^*_\tx{ad} \rw 1 \]
of complex algebraic groups with action of $L_F$ given by $\tx{Ad}\circ\varphi$. It leads to the connecting homomorphism
\[ \tx{Cent}(\varphi,\hat G^*_\tx{ad}) \rw H^1(L_F,Z(\hat G^*)). \]
Let $S_\varphi^\tx{ad}$ be the subgroup of $\tx{Cent}(\varphi,\hat G^*_\tx{ad})$ consisting of elements whose image in $H^1(L_F,Z(\hat G^*))$ is locally trivial, i.e. has trivial image in $H^1(L_{F_v},Z(\hat G^*))$ for every place $v \in \dot V$. This group is denoted by $S_\varphi/Z(\hat G^*)$ in \cite[(10.2.3)]{Kot84}. Let $\mc{S}_\varphi = \pi_0(S_\varphi^\tx{ad})$. Kottwitz conjectures that there exists a pairing
\begin{equation} \label{eq:gp} \<-,-\> : \mc{S}_\varphi \times \Pi_\varphi \rw \C \end{equation}
realizing each element of $\Pi_\varphi$ as the character of a finite-dimensional representation of $\mc{S}_\varphi$, so that the integer
\[ m(\varphi,\pi) = |\mc{S}_\varphi|^{-1}\sum_{x \in\mc{S}_\varphi} \<x,\pi\> \]
gives the $\varphi$-contribution of $\pi$ to the discrete spectrum of $G$. In other words, the multiplicity with which $\pi$ occurs in the discrete spectrum of $G$ should be the sum  $\sum_{\varphi} m(\varphi,\pi)$, where $\varphi$ runs over the set of equivalence classes \cite[\S10.4]{Kot84} of discrete generic global Arthur parameters with $\pi \in \Pi_\varphi$.

We will now give a construction of \eqref{eq:gp}. Let $s_\tx{ad} \in S_\varphi^\tx{ad}$ and choose a lift $s_\tx{sc} \in S_\varphi^\tx{sc}$, where $S_\varphi^\tx{sc}$ is the preimage in $\hat G^*_\tx{sc}$ of $S_\varphi^\tx{ad}$. We follow the argument of Subsection \ref{sub:endolg} to obtain from $s_\tx{sc}$ an element $\dot s_v \in S_{\varphi_v}^+$ for each $v \in \dot V$. Namely, by assumption there exists an element $y_v \in Z(\hat G^*)$ such that $s_\tx{der}\cdot y_v \in S_{\varphi_v}$ for each $v \in \dot V$. Decompose $y_v=y'_v \cdot y''_v$ with $y'_v \in Z(\hat G^*_\tx{der})$ and $y''_v \in Z(\hat G^*)^\circ$ and choose a lift $\dot y_v' \in Z(\hat G^*_\tx{sc})$ of $y'_v$. Then $(s_\tx{sc}\cdot\dot y'_v,y''_v) \in S_{\varphi_v}^+$. Let us write $\<(s_\tx{sc}\cdot\dot y'_v,y''_v),(G,\psi,z_v,\pi_v)\>$ for the character of the representation $\iota_{\varphi_v}((G,\psi,z_v,\pi_v))$ of $\pi_0(S_{\varphi_v}^+)$ evaluated at the element $(s_\tx{sc}\cdot\dot y'_v,y''_v)$.

\begin{pro} \label{pro:pairprod}
Almost all factors in the product
\[ \<s_\tx{ad},\pi\> = \prod_{v \in \dot V} \<z_{\tx{sc},v},\dot y'_v\>^{-1}\<(s_\tx{sc}\cdot\dot y'_v,y''_v),(G,\psi,z_v,\pi_v)\> \]
are equal to $1$ and the product is independent of the choices of $s_\tx{sc}$, $\dot y'_v$, $y''_v$, the collection $(z_{\tx{sc},v})_v$, and the global Whittaker datum $\mf{w}$. The function $s_\tx{ad} \mapsto \<s_\tx{ad},\pi\>$ is the character of a finite-dimensional representation of $\mc{S}_\varphi$.
\end{pro}

As in the case of Proposition \ref{pro:tfprod} we remark that if $G^*_\tx{der}$ is simply connected, or if $G^*$ satisfies the Hasse principle, then the factors $\<z_{\tx{sc},v},\dot y'_v\>^{-1}$ can be removed from the formula.

\begin{proof}
By construction of $\Pi_\varphi$ we have $\<(s_\tx{sc}\cdot\dot y'_v,y''_v),(G,\psi,z_v,\pi_v)\>=1$ for almost all $v$. Recall from Subsection \ref{sub:endolg} that the class of $z_{\tx{sc},v}$ in $H^1(u_v \rw W_v,Z_\tx{sc} \rw G^*_\tx{sc})$ is the localization $loc_v([z_\tx{sc}])$ of the global class $[z_\tx{sc}] \in H^1(P_{\dot V} \rw \mc{E}_{\dot V},Z_\tx{sc} \rw G^*_\tx{sc})$. According to Corollary \ref{cor:loctriv} this localization is trivial for almost all $v$.

We have established that almost all factors in the product are equal to $1$. Now we will show that each factor is the character of a finite-dimensional representation of the group $S_\varphi^\tx{sc}$.  Recall from Subsection \ref{sub:multnot} that $\hat{\bar G^*} = \hat G^*_\tx{sc} \times Z(\hat G^*)^\circ$. We have the maps $Z(\hat{\bar G^*})^{+_v} \rw S_{\varphi_v}^+$ and $Z(\hat{\bar G^*})^{+_v} \rw Z(\hat G^*_\tx{sc})$, the second given by projection onto the first factor of $Z(\hat{\bar G^*}) = Z(\hat G^*_\tx{sc}) \times Z(\hat G^*)^\circ$. We claim that we have a well-defined homomorphism
\begin{equation} \label{eq:sschomo} S_{\varphi}^\tx{sc} \rw S_{\varphi_v}^+ \times_{Z(\hat{\bar G^*})^{+_v}} Z(\hat G^*_\tx{sc}), \end{equation}
sending an element $s_\tx{sc} \in S_{\varphi}^{\tx{sc}}$ to the element $[(s_\tx{sc}\cdot \dot y'_v,y''_v),(\dot y'_v)^{-1}]$. Indeed, the tuple $(\dot y'_v,y''_v)$ can only be replaced by $(\dot y'_v \dot x'_v,y''_vx''_v)$, where $\dot x'_v \in Z(\hat G^*_\tx{sc})$, $x''_v \in Z(\hat G^*)^\circ$, and if $x'_v \in Z(\hat G^*_\tx{der})$ is the image of $\dot x'_v$, then $x=x'_vx''_v \in Z(\hat G^*)^{\Gamma_v}$. But then $(\dot x'_v,x''_v) \in Z(\hat{\bar G^*})^{+_v}$ and thus in the target of \eqref{eq:sschomo} we have the equality
\[ [(s_\tx{sc}\cdot \dot y'_v\dot x'_v,y''_vx''_v),(\dot y'_v\dot x'_v)^{-1}] = [(s_\tx{sc}\cdot \dot y'_v,y''_v),(\dot y'_v)^{-1}]. \]
This shows that \eqref{eq:sschomo} is well-defined. Checking that it is a homomorphism is straightforward.
Since the character of $Z(\hat{\bar G})^{+_v}$ by which the representation $\iota_{\varphi_v}((G,\psi,z_v,\pi_v))$ of $S_{\varphi_v}^+$ transforms is the pull-back of the character $\<z_{\tx{sc},v},-\>$ of $Z(\hat G_\tx{sc})$, the tensor product $\iota_{\varphi_v}((G,\psi,z_v,\pi_v)) \otimes \<z_{\tx{sc},v},-\>$ descends to a representation of the target of \eqref{eq:sschomo} and pulls back to a representation of $S_{\varphi}^\tx{sc}$, whose character is the factor corresponding to $v$ in the product above.

We have thus proved that this product is the character of a finite-dimensional representation of $S_\varphi^\tx{sc}$ and that it is independent of the choices of $\dot y'_v$ and $y''_v$ for each place $v$. We shall now argue that this representation is independent of the choice of $z_\tx{sc}$ and descends to the group $\mc{S}_\varphi = S_\varphi^\tx{sc}/Z(\hat G^*_\tx{sc})S_\varphi^{\tx{sc},\circ}$. First, $S_\varphi^{\tx{sc},\circ}$ maps to the connected component of the target of \eqref{eq:sschomo}, but the latter maps trivially to the finite group $\pi_0(S_{\varphi_v}^+) \times_{Z(\hat{\bar G^*})^{+_v}} Z(\hat G^*_\tx{sc})$, from which the representation $\iota_{\varphi_v}((G,\psi,z_v,\pi_v)) \otimes \<z_{\tx{sc},v},-\>$ is inflated. Second, if we either choose a different lift $z_\tx{sc}$ of $\Psi$, or a different lift $s_\tx{sc}$ of $s_\tx{ad}$, or a different global Whittaker datum $\mf{w}$, then some individual factors in the product might change. However, each such change will simultaneously occur in the same factor of the product in Proposition \ref{pro:tfprod}, because the two factors are related by the endoscopic character identities \cite[(5.11)]{KalRI}, which are part of the local conjecture whose validity we are assuming in our construction. However, we know from Proposition \ref{pro:tfprod} that the product over all places equals the adelic transfer factor and thus remains unchanged. This means that the product over all places of the changes in the factors is equal to $1$.
\end{proof}

\subsection{Relationship to  Arthur's framework} \label{sub:arthur}

In \cite{Art06}, Arthur states a version of the local Langlands conjecture and endoscopic transfer for general connected reductive groups, motivated by the stable trace formula. His formulation is different from the one given in \cite[\S5.4]{KalRI}. For one, it uses a modification of the classical local $S$-group that is different from ours. Moreover, besides the conjectural local pairings between $L$-packets and $S$-groups, it introduces further conjectural objects -- the mediating functions $\rho(\Delta,\tilde s)$ and the spectral transfer factors $\Delta(\phi',\pi)$. Their purpose is to encode the lack of normalization of the (geometric) transfer factors of Langlands-Shelstad as well as the non-invariance of these transfer factors under automorphisms of local endoscopic data, both of which are problems in local endoscopy that arise in the case of non-quasi-split groups. The local pairing is then supposed to be given as the product of the mediating function and the spectral transfer factor.

In this subsection, we will show that the local conjecture formulated in \cite[\S5.4]{KalRI} implies a stronger version of Arthur's conjecture of \cite{Art06}. The strengthening is due to the fact that the conjecture in \cite[\S5.4]{KalRI} implies that all objects in Arthur's formulation are canonically specified, and that furthermore the mediating functions have a simple explicit formula which in fact allows them to be eliminated if so desired. The reason we can achieve this is that the main local problems introduced by non-quasi-split groups -- namely the lack of canonical normalization of the Langlands-Shelstad transfer factor and its non-invariance under automorphisms of endoscopic data -- were resolved in \cite[\S5.3]{KalRI} based on the cohomology of local Galois gerbes.

We feel that the translation between our statement and that of Arthur, even though not difficult, allows one to combine the strengths of both formulations of the local conjectures. Arthur's framework makes the application of the local conjecture to global questions somewhat simpler. In particular, it makes the extraneous factors in the product expansions of Propositions \ref{pro:tfprod} and \ref{pro:pairprod} disappear. On the other hand, our local conjecture is more intrinsic to the group $G$ and makes many local operations, for example descent to Levi subgroups, more transparent. It is also closely related (and in some sense extends to arbitrary local fields) some of the discussion of \cite{ABV92}.

We begin by recalling Arthur's conjecture from \cite[\S3]{Art06}. It is stated for non-archimedean local fields of characteristic zero. We let $F$ be such a field and we let $\Psi : G^* \rw G$ be an equivalence class of inner twists with $G^*$ quasi-split. Let $\mf{w}$ be a Whittaker datum for $G^*$ and let $\varphi : L_F \rw {^LG^*}$ be a tempered Langlands parameter. Let $S_\varphi=\tx{Cent}(\varphi,\hat G^*)$ be its centralizer, $S_\varphi^\tx{ad}$ the image of $S_\varphi$ in $\hat G^*_\tx{ad}$, and $S_\varphi^\tx{sc}$ the preimage of $S_\varphi^\tx{ad}$ in $\hat G^*_\tx{sc}$. From $\Psi$ we obtain an element of $H^1(\Gamma,G^*_\tx{ad})$ and hence by Kottwitz's isomorphism \cite[\S1]{Kot86} a character $\zeta_\Psi : \hat Z_\tx{sc}^\Gamma \rw \C^\times$. Arthur proposes to choose an arbitrary extension $\dot\zeta_\Psi : Z_\tx{sc} \rw \C^\times$ of this character and to consider the set $\tx{Irr}(\pi_0(S_\varphi^\tx{sc}),\dot\zeta_\Psi)$ of irreducible representations of the finite group $\pi_0(S_\varphi^\tx{sc})$ that transform under $\hat Z_\tx{sc}$ by $\dot\zeta_\Psi$. Arthur expects this set to be in non-canonical bijection with the $L$-packet $\Pi_\varphi(G)$ of representations of $G(F)$ corresponding to $\varphi$. This non-canonical bijection depends on the choice of a normalization $\Delta$ of the Langlands-Shelstad transfer factor as well as on the (independent) choice of a mediating function $\rho(\Delta,\tilde s)$. Part of the conjecture is the expectation that these two choices can be made in such a way that $\tilde s \mapsto \<\tilde s,\pi\> = \rho(\Delta,\tilde s)^{-1}\Delta(\varphi',\pi)$ is the character of the irreducible representation of $\pi_0(S_\varphi^\tx{sc})$ corresponding to $\pi$, see \cite[\S3]{Art06}. The endoscopic character identities are supposed to match the stable character of an endoscopic parameter $\varphi'$ with the virtual character of the $L$-packet $\Pi_\varphi(G)$ in which the character of $\pi \in \Pi_\varphi(G)$ is weighted by the scalar $\Delta(\varphi',\pi)$. For global applications, Arthur formulates the following hypothesis \cite[Hypothesis 9.5.1]{Art13}: If $F$ is a number field and at each place $v$ of $F$ one has fixed a normalization $\Delta_v$ of the Langlands-Shelstad transfer factor such that $\prod_v \Delta_v = \Delta_\A$, then the mediating functions $\rho(\Delta_v,\tilde s_v)$ can be chosen coherently so as to satisfy the formula $\prod_v \rho(\Delta_v,\tilde s)=1$ whenever $\varphi$ is a global parameter and $\tilde s \in S_\varphi^\tx{sc}$.

We will now discuss how the local conjecture of \cite{KalRI} and the global results of the current paper imply a stronger form of Arthur's local and global expectations. Let $F$ be a number field, and $\Psi : G^* \rw G$ an equivalence class of inner twists of connected reductive groups defined over $F$ with $G^*$ quasi-split, and $\mf{w}$ a global Whittaker datum for $G^*$.

To discuss the local conjecture, we focus on a non-archimedean place $v \in \dot V$. Fix a lift $\dot\Psi_{v,\tx{sc}} \in H^1(u_v \rw W_v,Z_\tx{sc} \rw G^*_\tx{sc})$ of the localization $\Psi_v \in H^1(\Gamma_v,G^*_\tx{ad})$ of $\Psi$. It exists by \cite[Corollary 3.8]{KalRI}. The Tate-Nakayama-type duality theorem of \cite[\S4]{KalRI} interprets $\dot\Psi_{v,\tx{sc}}$ as a character $\zeta_{\dot\Psi_{v,\tx{sc}}} : \hat Z_\tx{sc} \rw \C^\times$ that extends $\zeta_\Psi$ and thus naturally provides a version of the character $\dot\zeta_\Psi$ in Arthur's formulation. However, we are now in a considerably stronger position, because besides providing the character $\zeta_{\dot\Psi_{v,\tx{sc}}}$, the cohomology class $\dot\Psi_{v,\tx{sc}}$ also normalizes the Langlands-Shelstad transfer factor and the pairings between $L$-packets and $S$-groups. In fact, since in the non-archimedean case $\dot\Psi_{v,\tx{sc}}$ and $\zeta_{\dot\Psi_{v,\tx{sc}}}$ determine each other \cite[Theorem 4.11 and Proposition 5.3]{KalRI}, this implies in Arthur's language that the choice of $\zeta_{\dot\Psi_{v,\tx{sc}}}$ normalizes the transfer factor and local pairings. Note however that in the archimedean case $\dot\Psi_{v,\tx{sc}}$ cannot be recovered from $\zeta_{\dot\Psi_{v,\tx{sc}}}$ and is thus the more fundamental object.

More precisely, let $\dot\Psi_v \in H^1(u_v \rw W_v,Z_\tx{der} \rw G^*)$ be the image of $\dot\Psi_{v,\tx{sc}}$, which in turn gives a character $\zeta_{\dot\Psi_v} : \pi_0(Z(\hat{\bar G^*})^{+_v}) \rw \C^\times$. We recall that since we are working with $Z_\tx{der}$, we have $\hat{\bar G^*}=\hat G^*_\tx{sc} \times Z(\hat G^*)^\circ$. The character $\zeta_{\dot\Psi_v}$ that we obtain is the pull-back of $\zeta_{\dot\Psi_{v,\tx{sc}}}$ under the map
\[ \pi_0(Z(\hat{\bar G^*})^{+_v}) \rw \hat Z_\tx{sc} \]
given by projection onto the first factor. The conjecture of \cite[\S5.4]{KalRI} then states that given a tempered parameter $\varphi_v : L_{F_v} \rw {^LG^*}$ there is a canonical bijection
\[ \Pi_{\varphi_v}(G) \rw \tx{Irr}(\pi_0(S_{\varphi_v}^+),\zeta_{\dot\Psi_v}), \]
where $\Pi_{\varphi_v}(G)$ is the $L$-packet on $G$ corresponding to $\varphi_v$. This packet is related to the compound $L$-packet $\Pi_{\varphi_v}$ of \cite[\S5.4]{KalRI} by $\Pi_{\varphi_v}(G)=\{\pi| (G,\psi,z_v,\pi) \in \Pi_{\varphi_v}\}$, where $(\psi,z_v) \in \dot\Psi_v$ is any representative. Here, as well as below, we find it convenient to think of $\dot\Psi_v$ as an equivalence class of rigid inner twists $G^* \rw G$.

In order to translate this into Arthur's formulation we need to relate the two sets $\tx{Irr}(\pi_0(S_{\varphi_v}^+),\zeta_{\dot\Psi_v})$ and $\tx{Irr}(\pi_0(S_{\varphi_v}^\tx{sc}),\zeta_{\dot\Psi_{v,\tx{sc}}})$. We claim that there is in fact a canonical bijection between these sets. For this, we consider the homomorphism
\begin{equation} \label{eq:sschomov} S_{\varphi_v}^\tx{sc} \rw S_{\varphi_v}^+ \times_{Z(\hat{\bar G^*})^{+_v}} Z(\hat G^*_\tx{sc}), \end{equation}
defined in the same way as \eqref{eq:sschomo}, the only difference being that the source is $S_{\varphi_v}^\tx{sc}$ instead of the global $S_{\varphi}^\tx{sc}$ used there. Namely, to $s_\tx{sc} \in S_{\varphi_v}^\tx{sc}$ we choose $y \in Z(\hat G^*)$ such that $s_\tx{der}\cdot y \in S_{\varphi_v}$, where $s_\tx{der} \in \hat G^*_\tx{der}$ is the image of $s_\tx{sc}$. Then we choose a decomposition $y=y'\cdot y''$ with $y' \in Z(\hat G^*_\tx{der})$ and $y'' \in Z(\hat G^*)^\circ$, and a lift $\dot y' \in Z(\hat G^*_\tx{sc})$ of $y'$, and then define the image of $s_\tx{sc}$ under \eqref{eq:sschomov} to be $[(s_\tx{sc}\dot y',y''),\dot y'^{-1}]$. One sees immediately that \eqref{eq:sschomov} is an isomorphism and that its inverse sends the pair $(\dot s,(\dot y')^{-1})$ with $\dot s=(\dot s_1,x) \in S_{\varphi_v}^+ \subset \hat{\bar G^*}=\hat G^*_\tx{sc} \times Z(\hat G^*)^\circ$ and $\dot y' \in Z(\hat G^*_\tx{sc})$ to $\dot s_1\cdot (\dot y')^{-1}$. We now obtain the bijection
\begin{equation} \label{eq:irrsbij}
\tx{Irr}(\pi_0(S_{\varphi_v}^+),\zeta_{\dot\Psi_v}) \rw \tx{Irr}(S_{\varphi_v}^\tx{sc}, \zeta_{\dot\Psi_{v,\tx{sc}}})
\end{equation}
that sends $\rho$ to the pull-back under \eqref{eq:sschomov} of $\rho \otimes \zeta_{\dot\Psi_{v,\tx{sc}}}$.

In order to discuss endoscopic transfer we fix a representative $(\psi,z_{v,\tx{sc}}) \in \dot\Psi_{v,\tx{sc}}$, where $\psi \in \Psi$ is an inner twist $\psi : G^* \rw G$ and $z_{v,\tx{sc}} \in Z^1(u_v \rw W_v,Z_\tx{sc} \rw G^*_\tx{sc})$ lifts the 1-cocycle $\psi^{-1}\sigma(\psi) \in Z^1(\Gamma,G^*_\tx{ad})$. Given an element $s_\tx{sc} \in S_{\varphi_v}^\tx{sc}$ let $(\dot s,(\dot y'_v)^{-1}) \in S_{\varphi_v}^+ \times Z(\hat G_\tx{sc})$ correspond to $s_\tx{sc}$ under \eqref{eq:sschomov}. We obtain from the pair $(\dot s,\varphi_v)$ the refined endoscopic datum $\mf{\dot e}_v$ and consider the product
\[ \Delta[s_\tx{sc},v] := \<z_{\tx{sc},v},\dot y'_v\>\Delta[\mf{w},\mf{\dot e}_v,\mf{z}_v,\psi,z_v], \]
which was already studied in Proposition \ref{pro:tfprod}. This product is of course an absolute transfer factor, being a scalar multiple of the normalized transfer factor of \cite[\S5.3]{KalRI}. Moreover, it depends only on the image of $(\dot s,(\dot y'_v)^{-1})$ in $S_{\varphi_v}^+ \times_{Z(\hat{\bar G^*})^{+_v}} Z(\hat G^*_\tx{sc})$, and hence on $s_\tx{sc}$, which justifies the notation. Proposition \ref{pro:tfprod} can then be reformulated as the equality
\begin{equation} \label{eq:tfprod1} \Delta_\A = \prod_{v \in \dot V} \Delta[s_\tx{sc},v], \end{equation}
for any $s_\tx{sc} \in S_{\varphi}^\tx{sc}$.

Although we will not need it, we can give a description of $\Delta[s_\tx{sc},v]$ analogous to \cite[(5.1)]{KalRI}. It is
\[ \Delta[s_\tx{sc},v](\gamma_z,\delta') = \Delta[\mf{w},\mf{e},\mf{z}](\gamma_\mf{z},\delta) \cdot \<g_\tx{sc}^{-1}z_{\tx{sc},v}(w)\sigma_w(g_\tx{sc}),s_{\tx{sc},\gamma,\delta}\>^{-1}, \]
where $g_\tx{sc} \in G^*_\tx{sc}$ and $\delta \in G^*(F)$ satisfy $\psi(g_\tx{sc}\delta g_\tx{sc}^{-1})=\delta' \in G(F)$. Thus $w \mapsto g_\tx{sc}^{-1}z_{\tx{sc},v}(w) \sigma_w(g_\tx{sc}^{-1})$ is an element of $Z^1(u_v \rw W_v,Z_\tx{sc} \rw S_\tx{sc})$, where $S=\tx{Cent}(\delta,G^*)$, and we are pairing the class of this element with an element $s_{\tx{sc},\gamma,\delta} \in \pi_0(\hat{\bar S_\tx{sc}})$ obtained from $s_\tx{sc}$ as follows. Identify $\hat H$ with $\tx{Cent}(s,\hat G^*)^\circ$ and write $\hat H_\tx{sc}$ for the preimage of $\hat H$ in $\hat G^*_\tx{sc}$. Then we have $s_\tx{sc} \in Z(\hat H_\tx{sc})$. Setting $S^H=\tx{Cent}(\gamma,H)$ we have the admissible isomorphism $\phi_{\gamma,\delta} : S^H \rw S$ satisfying $\phi_{\gamma,\delta}(\gamma)=\delta$. Together with the canonical embedding $Z(\hat H) \rw \hat{S^H}$ this leads to an embedding $Z(\hat H_\tx{sc}) \rw \hat S_\tx{sc} = \hat{\bar S_\tx{sc}}$ and $s_{\tx{sc},\gamma,\delta}$ is the image of $s_\tx{sc}$ under this embedding.

Finally, the endoscopic character identities \cite[(5.11)]{KalRI} state that when the transfer factor $\Delta[s_\tx{sc}^{-1},v]$ is used to match orbital integrals, the stable character of the endoscopic parameter corresponding to $(\varphi_v,s)$ is matched with the virtual character
\[ e(G)\sum_{\pi: (G,\psi,z_v,\pi) \in \Pi_{\varphi_v}} \<\dot s,(G,\psi,z_v,\pi)\> \cdot \<z_{\tx{sc},v},\dot y'_v\>^{-1} \Theta_\pi. \]
The scalar in front of $\Theta_\pi$ again depends only on the image of $(\dot s,(\dot y'_v)^{-1})$ in $S_{\varphi_v}^+ \times_{Z(\hat{\bar G})^{+_v}} Z(\hat G_\tx{sc})$, as shown in Proposition \ref{pro:pairprod}, and hence only on $s_\tx{sc}$. If we denote it by $\<s_\tx{sc},\pi\>$ for short, then the map $s_\tx{sc} \mapsto \<s_\tx{sc},\pi\>$ is the character of an irreducible representation of $S_{\varphi_v}^\tx{sc}$ and the map $\pi \mapsto \<-,\pi\>$ is the bijection $\Pi_{\varphi_v} \rw \tx{Irr}(S_{\varphi_v}^\tx{sc},\zeta_{\dot\Psi_{v,\tx{sc}}})$ given by \eqref{eq:irrsbij}. Proposition \ref{pro:pairprod} implies that for any $s_\tx{ad} \in S_\varphi^\tx{ad}$ the product
\begin{equation} \label{eq:pairprod1} \<s_\tx{ad},\pi\> = \prod_{v \in \dot V} \<s_\tx{sc},\pi\> \end{equation}
is independent of the lift $s_\tx{sc} \in S_\varphi^\tx{ad}$ of $s_\tx{ad}$ and provides the character of a finite-dimensional representation of $S_\varphi^\tx{ad}$.

Arthur's mediating function is now given by the simple formula
\[ \rho(\Delta[s_\tx{sc},v],s_\tx{sc})=1, \]
where $\Delta[s_\tx{sc},v]$ is of course the transfer factor corresponding to $s_\tx{sc} \in S_{\varphi_v}^\tx{sc}$ defined above. This formula, together with equation \eqref{eq:tfprod1}, which itself is a consequence of Proposition \ref{pro:tfprod}, imply the validity of \cite[Hypothesis 9.5.1]{Art13}.

\appendix
\section*{Appendices}
\addcontentsline{toc}{section}{Appendices}
\renewcommand{\thesubsection}{\Alph{subsection}}

\subsection{Approximating the cohomology of a connected reductive group by the cohomology of tori} \label{app:coh}

In this appendix we will formulate and prove a well-known lemma about the Galois cohomology of reductive groups, for which we have not been able to find a convenient reference.

\begin{lem} \label{lem:cohapprox}
	Let $h \in H^1(\Gamma,G(\ol{F}))$. There exists a maximal torus $T \subset G$ such that $h$ is in the image of $H^1(\Gamma,T(\ol{F})) \rw H^1(\Gamma,G(\ol{F}))$.
\end{lem}
\begin{proof}
We first reduce to the case that $G$ is semi-simple. Let $h_\tx{ad} \in H^1(\Gamma,G_\tx{ad})$ be the image of $h$. Let $T \subset G$ be a maximal torus such that $h_\tx{ad}$ is the image of $h_{T,\tx{ad}} \in H^1(\Gamma,T_\tx{ad})$. The image of $h_{T,\tx{ad}}$ in $H^2(\Gamma,Z(G))$ is equal to the image of $h_\tx{ad}$ there, which is trivial. Hence there exists $h_T \in H^1(\Gamma,T)$ mapping to $h_{T,\tx{ad}}$. Both $h_T$ and $h$ have the same image in $H^1(\Gamma,G_\tx{ad})$, so there exists $h_Z \in H^1(\Gamma,Z(G))$ such that $h_Z \cdot h_T \in H^1(\Gamma,T)$ maps to $h$.

We assume from now on that $G$ is semi-simple and let $G_\tx{sc} \rw G$ be its simply connected cover with kernel $C$. Let $S \subset V_F$ be a finite set of places containing all $v \in V_F$ for which the image of $h$ in $H^1(\Gamma_v,G)$ is non-trivial. Assume in addition that $S$ contains all archimedean places and at least one non-archimedean place. Let $T \subset G$ be a maximal torus which is fundamental \cite[\S10]{Kot86} at all $v \in S$. Such a maximal torus exists by \cite[\S7.1, Cor. 3]{PR94}. We consider the diagram
\[ \xymatrix{
	H^1(\Gamma,G_\tx{sc})\ar[r]&H^1(\Gamma,G)\ar[r]&H^2(\Gamma,C)\\
	H^1(\Gamma,T_\tx{sc})\ar[r]\ar[u]&H^1(\Gamma,T)\ar[r]\ar[u]&H^2(\Gamma,C)\ar@{=}[u]\ar[r]&H^2(\Gamma,T_\tx{sc})
}
\]
We claim that the image of $h$ in $H^2(\Gamma,T_\tx{sc})$ is trivial. Indeed, it has the property that for $v \notin S$ its image in $H^2(\Gamma_v,T_\tx{sc})$ is trivial due to the definition of $S$, and for $v \in S$ its image in $H^2(\Gamma_v,T_\tx{sc})$ is also trivial because of \cite[Lemma 10.4]{Kot86}. The claim follows from a theorem of Kneser \cite[\S6.3, Prop 6.12]{PR94} that asserts that $T_\tx{sc}$ satisfies the Hasse-principle in degree $2$.

Let $h'_T \in H^1(\Gamma,T)$ be a preimage of the image of $h$ in $H^2(\Gamma,C)$ and let $h' \in H^1(\Gamma,G)$ be the image of $h'_T$. It may not be the case that $h=h'$. Let $G'$ be the twist of $G$ by $h'$. Then $T$ is a maximal torus in $G'$ and we have the bijection
\[ H^1(\Gamma,G) \rw H^1(\Gamma,G'),\qquad y \mapsto y\cdot h'^{-1}.\]
The image of $h\cdot h'^{-1} \in H^1(\Gamma,G')$ in $H^2(\Gamma,C)$ is trivial. Let $h''_\tx{sc} \in H^1(\Gamma,G'_\tx{sc})$ be a preimage of $h\cdot h'^{-1}$. We claim that there exists $h''_{T,\tx{sc}} \in H^1(\Gamma,T_\tx{sc})$ mapping to $h''_\tx{sc}$. Granting this, let $h''_T \in H^1(\Gamma,T)$ be the image of $h''_{T,\tx{sc}}$. Then the image of $h''_T \cdot h'_T$ in $H^2(\Gamma,G)$ is equal to $h$ and the proof is complete.

We will now prove the outstanding claim. For this, consider the diagram
\[ \xymatrix{
H^1(\Gamma,G'_\tx{sc})\ar[r]&\prod_{v \in V_{F,\infty}} H^1(\Gamma_v,G'_\tx{sc})\\
H^1(\Gamma,T_\tx{sc})\ar[r]\ar[u]&\prod_{v \in V_{F,\infty}} H^1(\Gamma_v,T_\tx{sc})\ar[u]\\
} \]
By work of Kneser, Harder, and Chernousov \cite[\S6.1, Thm 6.6]{PR94} the top map is bijective. The right map is surjective by \cite[10.2]{Kot86} because $T_\tx{sc,v}$ is fundamental for all $v \in V_{F,\infty} \subset S$. Let $(t_v)_{v \in V_{F,\infty}}$ be an element of the bottom right term whose image in the top right term corresponds to $h''_\tx{sc}$. We claim that $(t_v)_v$ has a lift $h''_{T,\tx{sc}} \in H^1(\Gamma,T_\tx{sc})$. For this we inspect the image of $(t_v)_v$ in $H^1(\Gamma,T_\tx{sc}(\A_{\ol{F}})/T_\tx{sc}(\ol{F}))$. Let $v_0 \in S$ be non-archimedean. Then $T_{\tx{sc},v_0}$ is anisotropic, which implies that $\hat{T_\tx{sc}}^{\Gamma_{v_0}}$, and thus also $\hat{T_\tx{sc}}^\Gamma$, is finite. This in turn implies that the map $\pi_0(\hat{T_\tx{sc}}^\Gamma) \rw \pi_0(\hat{T_\tx{sc}}^{\Gamma_{v_0}})$ is injective. This map is dual to the composition
\[ H^1(\Gamma_{v_0},T_\tx{sc}) \stackrel{\cong}{\lrw} H^1(\Gamma,T_\tx{sc}(F_{v_0}\otimes_F \ol{F}))\rw H^1(\Gamma,T_\tx{sc}(\A_{\ol{F}})/T_\tx{sc}(\ol{F})), \]
which therefore must be surjective. Hence we may augment the collection $(t_v)_{v \in V_{F,\infty}}$ to a collection $(t_v)_{v \in V_{F,\infty} \cup \{v_0\}}$ so that the image of the new collection in $H^1(\Gamma,T_\tx{sc}(\A_{\ol{F}})/T_\tx{sc}(\ol{F}))$ vanishes. But then this new collection lifts to an element $h''_{T,\tx{sc}} \in H^1(\Gamma,T_\tx{sc})$. The image of this element in the top right term of the above diagram is the same as the image of the old collection $(t_v)_{v \in V_{F,\infty}}$, hence the same as the image of $h''_\tx{sc}$. By the bijectivity of the top map this means that the image of $h''_{T,\tx{sc}}$ in $H^1(\Gamma,G'_\tx{sc})$ equals $h''_\tx{sc}$.
\end{proof}

\subsection{The Shapiro isomorphism for pro-finite groups} \label{app:shap}

Let $\Delta \subset \Gamma$ be two finite groups and $X$ a group with an action of $\Delta$. All groups are written multiplicatively. We have the $\Delta$-equivariant homomorphisms
\[ X \stackrel{i}{\lrw} \tx{Ind}_\Delta^\Gamma X \stackrel{j}{\lrw} X \]
where $i(x)$ is the function $\Gamma \to X$ defined by $i(x)(\delta)=\delta x$ for $\delta \in \Delta$ and $i(x)(\gamma)=1$ for $\gamma \in \Gamma \sm \Delta$; and $j(f)=f(1)$. It is well known that
\[ \xymatrix{
	H^k(\Delta,X)\ar@<.5ex>[r]^-{\tx{cor}_\Delta^\Gamma \circ i}&H^k(\Gamma,\tx{Ind}_\Delta^\Gamma X)\ar@<.5ex>[l]^-{j\circ\tx{res}_\Delta^\Gamma}
}\]
are mutually inverse isomorphisms for all $k$ (if $X$ is not abelian we only consider $k=0,1$), where $\tx{res}$ and $\tx{cor}$ denote the restriction and corestriction maps respectively.

Let now $\Gamma$ be a second-countable pro-finite group (i.e. an inverse limit of a totally ordered countable inverse system of finite groups), $\Delta \subset \Gamma$ a closed subgroup, and $X$ a smooth $\Delta$-module (a discrete topological group with a continuous $\Delta$-action). Then $j\circ \tx{res}_\Delta^\Gamma$ still makes sense and is still an isomorphism \cite[Ch II,\S2, Theorem 8]{Shatz}. On the other hand, the corestriction map does not make sense in this context. The purpose of this appendix is to show that nonetheless, the composition $\tx{cor}_\Delta^\Gamma \circ i$ does make sense, and to give some explicit constructions related to it.

We shall construct for every $k \in \N$ a map $S_\Delta^\Gamma : C^k(\Delta,X) \rw C^k(\Gamma,\tx{Ind}_{\Delta}^\Gamma X)$, that is functorial in $X$, commutes with the differentials on both complexes, and is a section of $j\circ\tx{res}_\Delta^\Gamma$. For the construction, we will need a set-theoretic section $s : \Delta \lmod \Gamma \to \Gamma$ of the natural projection, normalized by $s(\Delta)=1$. We can think of $s$ as a map $\Gamma \to \Gamma$ that is invariant under multiplication by $\Delta$ on the left. Via the equation $\gamma = r(\gamma)s(\gamma)$ giving such a map $s$ is equivalent to giving a map $r : \Gamma \to \Delta$ that is equivariant under multiplication by $\Delta$ on the left and satisfies $r(1)=1$.

Write $\Gamma = \varprojlim \Gamma_i$, where $\Gamma_i$ are finite quotients indexed by the natural numbers with kernels $\Gamma^i$, and $\Delta = \varprojlim \Delta_i$, where $\Delta_i$ is the image of $\Delta$ in $\Gamma_i$ and $\Delta^i = \Delta \cap \Gamma^i$ is the kernel of $\Delta \to \Delta^i$.

\begin{lem} \label{lem:sh1} The following are equivalent for each $i$:
\begin{enumerate}
	\item The composition $\Gamma \stackrel{s}{\lrw} \Gamma \to \Gamma_i$ is inflated from a map $\Gamma_i \to \Gamma_i$;
	\item $s(a)^{-1}s(ab) \in \Gamma^i$ for $a \in \Gamma$ and $b \in \Gamma^i$;
	\item The composition $\Gamma \stackrel{r}{\lrw} \Delta \to \Delta_i$ is inflated from a map $\Gamma_i \to \Delta_i$;
	\item $r(a)^{-1}r(ab) \in \Delta^i$ for $a \in \Gamma$ and $b \in \Gamma^i$.
\end{enumerate}
\end{lem}
We leave the proof to the reader.

\begin{lem} \label{lem:sh2} There exists a section $s$ satisfying the equivalent conditions of the above lemma.
\end{lem}
\begin{proof}
For each $i$, we will give a section $s_i : \Delta_i \lmod \Gamma_i \rw \Gamma_i$, normalized by $s_i(\Delta_i)=1$, so that for $i$ and $i+1$ we have the commutative diagram of sets
\[ \xymatrix{
	\Delta_{i+1} \lmod \Gamma_{{i+1}}\ar[d]\ar[r]^-{s_{i+1}}& \Gamma_{i+1}\ar[d]\\
	\Delta_i \lmod \Gamma_i\ar[r]^-{s_i}&\Gamma_i
}
\]
This is easy to do inductively: If $s_i$ is given, let $a_1,\dots,a_n \in \Gamma_i$ be its values, a set of representatives in $\Gamma_i$ for the elements of $\Delta_i \lmod \Gamma_i$. Choose a lift $\dot a_i \in \Gamma_{i+1}$ for each $a_i$. Since $s_i$ is normalized we have $a_k=1$ for some $k$ and then we take $\dot a_k=1$. Moreover, choose a set $b_1, \dots, b_m \in \Gamma^{i+1} \lmod \Gamma^i$ of representatives for $\Delta^i \cdot \Gamma^{i+1} \lmod \Gamma^i$, again normalized so that $b_k=1$ for some $k$. Then $\{b_j \dot a_i\}$ is a set of representatives in $\Gamma_{i+1}$ for $\Delta_{i+1} \lmod \Gamma_{i+1}$ and hence provides $s_{i+1}$. Composing each $s_i$ with $\Delta \lmod \Gamma \rw \Delta_i \lmod \Gamma_i$ we obtain an inverse system of maps $\Delta \lmod \Gamma \rw \Gamma_i$ and this system gives the map $s$.
\end{proof}

Given such a section $s$, the Shapiro map is defined by the formula
\begin{eqnarray*} S^\Gamma_{\Delta}(c)(\sigma_0,\dots,\sigma_k)(a)&=&r(a)c(s(a)\sigma_0s(a\sigma_0)^{-1},\dots,s(a)\sigma_ks(a\sigma_k)^{-1})\\
&=&r(a)c(r(a)^{-1}r(a\sigma_0),\dots,r(a)^{-1}r(a\sigma_k)),
\end{eqnarray*}
for any $c \in C^k(\Delta,X)$, where we have presented the cochains in homogeneous form. Here $\sigma_0,\dots,\sigma_k \in \Gamma$ are the arguments of the $k$-cochain $S^k_{\dot v}(c)$.

\begin{lem} \label{lem:sh3} For any $\sigma_0,\dots,\sigma_k \in \Gamma$, the assignment $a \mapsto S^\Gamma_\Delta(c)(\sigma_0,\dots,\sigma_k)$ is a continuous $\Delta$-equivariant map $\Gamma \to X$ and hence belongs to $\tx{Ind}_\Delta^\Gamma X$. Moreover, $S^\Gamma_\Delta(c) : \Gamma^{k+1} \to \tx{Ind}_\Delta^\Gamma$ is continuous and $\Gamma$-equivariant, thus belongs to $C^k(\Gamma,\tx{Ind}_\Delta^\Gamma X)$.
\end{lem}
\begin{proof}
By assumption $c$ is continuous and hence inflated from $C^k(\Delta_i,X^{\Delta^i})$ for some $i$. If we multiply either $a$ or some $\sigma_j$ by an element of $\Gamma^i$, then according to Lemma \ref{lem:sh1} both $r(a)$ and $r(a\sigma_j)$ are multiplied by an element of $\Delta^i$ and the value of $S^\Gamma_\Delta(\sigma_0,\dots,\sigma_k)(a)$ is unchanged. Thus the map $(a,\sigma_0,\dots,\sigma_k) \mapsto S^\Gamma_\Delta(\sigma_0,\dots,\sigma_k)(a)$ is inflated from $\Gamma_i$. The equality $S^\Gamma_\Delta(c)(\sigma_0,\dots,\sigma_k)(\delta a)=\delta S^\Gamma_\Delta(c)(\sigma_0,\dots,\sigma_k)(a)$ is immediately checked. Finally, given $\tau \in \Gamma$ let $r(a\tau)=r(a)\eta_{a,\tau}$ with $\eta_{a,\tau} \in \Delta$. Then the $\Delta$-equivariance of the cochain $c$ implies
\begin{eqnarray*}
S^\Gamma_{\Delta}(c)(\tau\sigma_0,\dots,\tau\sigma_k)(a)&=&r(a)c(r(a)^{-1}r(a\tau\sigma_0),\dots,r(a)^{-1}r(a\tau\sigma_k))\\
&=&r(a)\eta_{a,\tau}c(\eta_{a,\tau}^{-1}r(a)^{-1}r(a\tau\sigma_0),\dots,\eta_{a,\tau}^{-1}r(a)^{-1}r(a\tau\sigma_k))\\
&=&S^\Gamma_{\Delta}(c)(\sigma_0,\dots,\sigma_k)(a\tau).
\end{eqnarray*}
\end{proof}

If we replace the section $s$, or equivalently the map $r$, by another choice, say $\bar s$ and $\bar r$, we obtain another Shapiro map $\bar S_\Delta^\Gamma$. We know a-priori that for any $z \in Z^k(\Delta,X)$ the two $k$-cocycles $S_\Delta^\Gamma(z)$ and $\bar S_\Delta^\Gamma(z)$ will be cohomologous. We record here a formula for a $1$-cochain whose differential realizes the required coboundary in the case $k=2$.

\begin{lem} \label{lem:sh4} Assume that $X$ is abelian. Given $z \in Z^2(\Delta,X)$ let
\[ c(\sigma_0,\sigma_1)(a) = z(r(a\sigma_0)r(a\sigma_1),\bar r(a\sigma_0)) \cdot z(r(a\sigma_1),\bar r(a\sigma_0),\bar r(a\sigma_1))^{-1}.\]
Then $\partial c = S_\Delta^\Gamma(z) \cdot \bar S_\Delta^\Gamma(z)^{-1}$.
\end{lem}
We leave the verification to the reader.

\bibliographystyle{amsalpha}
\bibliography{/Users/tashokaletha/Dropbox/Work/TexMain/bibliography.bib}

\end{document}